\documentclass{article}   
\usepackage{geometry} 
\usepackage{graphicx,color}	
\usepackage[dvipsnames]{xcolor}
\usepackage{amssymb,amsmath,amsthm,amsfonts}
\usepackage{bm} 
\usepackage[english]{babel}
\usepackage[latin1]{inputenc}
\usepackage[colorlinks]{hyperref}
\usepackage{enumitem}

\usepackage{tikz}
\usepackage{pgfplots}
\usepgflibrary{arrows.meta}
\usetikzlibrary{patterns}
\usepgfplotslibrary{colormaps}
\usetikzlibrary{pgfplots.colormaps}

\pgfplotsset{compat=1.11}

\tikzset{> = {Stealth[inset=0pt]}}

\usepackage{newtxtext,newtxmath,helvet,courier}
\usepackage{fix-cm}

\newtheorem{proposition}{Proposition}[section]
\newtheorem{theorem}[proposition]{Theorem}
\newtheorem{corollary}[proposition]{Corollary}
\newtheorem{lemma}[proposition]{Lemma}
\newtheorem{examplem}[proposition]{Example}

\theoremstyle{definition}
\newtheorem{definition}[proposition]{Definition}
\newtheorem{remark}[proposition]{Remark}

\newtheorem{case}{Case}

\numberwithin{equation}{section}

\newcommand{\eps}{\varepsilon}

\newcommand{\sphere}{{\mathbb{S}}}
\newcommand{\from}{\colon}
\newcommand{\loc}{{\mathrm{loc}}}

\def\N{\mathbb{N}}
\def\Z{\mathbb{Z}}

\def\R{\mathbb{R}}
\def\C{\mathbb{C}}
\def\HH{\mathbb{H}}
\def\II{\mathbb{I}}

\def\S{\mathbb{S}}

\def\Xh{\mathcal{X}}
\def\Acal{\mathcal{A}}
\def\Bcal{\mathcal{B}}
\def\BLC{{\mathcal{B}_{\text{LC}}}}
\def\BKS{{\mathcal{B}_{\text{KS}}}}
\def\TBKS{{\widetilde{\mathcal{B}_{\text{KS}}}}}
\def\Dcal{\mathcal{D}}
\def\Ecal{\mathcal{E}}
\def\Qcal{\mathcal{Q}}
\def\Rcal{\mathcal{R}}

\def\Lcal{\mathcal{L}}

\def\Mcal{\mathcal{M}}
\def\Ucal{\mathcal{U}}

\newcommand{\splc}{{W_{-1}}}

\newcommand{\spc}{{X}}

\newcommand{\pro}[1]{{\left. #1\right|_{(0,1)}}}

\newcommand{\scalar}[2]{\left\langle {#1},{#2}\right\rangle}

\DeclareMathOperator{\dist}{dist}

\title{Regularized variational principles\\ for the perturbed Kepler problem}
\author{Vivina Barutello \and Rafael Ortega \and Gianmaria Verzini}
\date{\today}

\begin{document}

\maketitle

\begin{abstract}
The goal of the paper is to develop a method that will combine the use of variational techniques with regularization methods in order to study  existence and multiplicity results for the periodic and the Dirichlet problem associated to the perturbed Kepler system
\[
\ddot x = -\frac{x}{|x|^3} + p(t),
\quad x \in \R^d,
\]
where $d\geq 1$, and $p:\R\to\R^d$ is smooth and  $T$-periodic, $T>0$.

The existence of critical points for the action functional associated to the problem is proved via a non-local change of variables inspired by Levi-Civita and Kustaanheimo-Stiefel techniques.
As an application we will prove that the perturbed Kepler problem has infinitely many generalized $T$-periodic solutions for $d=2$ and $d=3$, without any symmetry assumptions on $p$.
\end{abstract}
\noindent
{\footnotesize \textbf{AMS-Subject Classification}}.
{\footnotesize 70F16, 34C25, 49J35, 70F05.}

\section{Introduction}
Consider the perturbed Kepler problem
\begin{equation}\label{eq:perturbed_kepler}
\ddot x = -\frac{x}{|x|^3} + p(t),
\quad x \in \R^d,
\end{equation}
where $d\geq 1$, and $p:\R\to\R^d$ is smooth and  $T$-periodic, $T>0$.

Recently several papers have examined the existence of generalized $T$-periodic solutions (\cite{OrtegaPort,RebeloSimoes,ZhaoANS} for $d=1$, \cite{BOZ,BDP20} for $d=2,3$), according
to the following definition.

\begin{definition}\label{def:gen_sol}
A \emph{generalized solution} of \eqref{eq:perturbed_kepler} on the interval $J\subset \R$
is a continuous function $x: J \to \R^d$ satisfying the following conditions:
\begin{enumerate}
	\item  the set $Z= \{t \in \overline{J} : x(t) = 0\}$ of \emph{collisions} is discrete,
	\item  for any open interval $I \subset J \setminus Z$, the function $x$ is $C^2(I)$ and satisfies \eqref{eq:perturbed_kepler} on $I$,
	\item  for any $t_* \in Z$, the limits
\begin{equation}\label{eq:gensol}
\lim_{t \to t_*} \frac{x(t)}{\vert x(t) \vert} \qquad \text{and} \qquad \lim_{t \to t_*}\left( \frac12 \vert \dot x(t) \vert^2- \frac{1}{\vert x(t) \vert}\right)
\end{equation}
	of collision direction and collision energy exist and are finite.
\end{enumerate}
\end{definition}
The existence of the right and left limits in 3. of Definition \ref{def:gen_sol} is a consequence of \cite{sper2body} whenever $t_*$ is isolated in the collision set $Z$; therefore, in case of isolated collisions, the only thing to check is that left/right limits agree. In particular if $t_*\in Z \cap \partial J$ (e.g. in the Dirichlet problem) we just have to check that they are isolated collision instants.

We refer to \cite{BOZ} for a discussion on the significance of these solutions. Using the Poincar\'e-Birkhoff theorem it was proven in \cite{OrtegaPort,RebeloSimoes} that the equation \eqref{eq:perturbed_kepler} in one dimension has infinitely many generalized $T$-periodic solutions.
For $d\geq 2$ the knowledge of the periodic problem for \eqref{eq:perturbed_kepler} is more fragmentary. Authors in \cite{BDP20} proved the existence of at least one periodic solution if $d=2$, while in \cite{BOZ} the existence of any number of periodic solutions has been proved when $d=2,3$ and $p$ is small enough. The approach employed in \cite{BOZ} and \cite{BDP20} are quite different. The basic idea in \cite{BOZ} is to regularize the system \eqref{eq:perturbed_kepler} by a change of variables from $x=x(t)$ to $y=y(\tau)$, where
\begin{equation}
\label{eq:cv}
\tau = \int_0^t\frac{d\xi}{|x(\xi)|},\qquad x=\Phi(y).
\end{equation}
Here $\Phi$ is an appropriate homogeneous function of degree 2 (different choices are available, depending on the spatial dimension $d$). The definition of $\tau$ comes from the well known Sundman integral \cite{sundman1907}. In the variables $(\tau,y)$ the system has no singularities, and it is possible to apply some classical bifurcation results due to Weinstein. The approach in \cite{BDP20} is variational, searching for critical points of the action functional
\begin{equation}
\label{action}
\Acal(x) = \int_{0}^{T}\left[\frac12|\dot x(t)|^2+\frac{1}{|x(t)|} + \scalar{p(t)}{x(t)} \right]\,dt.
\end{equation}
They prove the existence of critical points by minimization among loops around the origin
having nontrivial winding number, but they also prove that these critical points do produce
generalized solutions. Incidentally, in the literature the variational framework have been
exploited also in connection with different notions of periodic generalized solutions, see
e.g. \cite{bahrab,MR1077264,MR1218100,Rab1994,SerraCZ1994}; anyhow, most of these papers deal
with
autonomous singular Hamiltonian systems, and their (weaker) notion of generalized solutions
requires only conservation of energy across collisions (i.e. only the second condition
in \eqref{eq:gensol}).

Our goal in this paper is to develop a method that will combine the use of variational techniques with the regularization method. In Section \ref{sec:euristic}, after some heuristic computations, we will show that the change of variables \eqref{eq:cv} transforms the functional $\Acal$ into a new functional $\Bcal = \Bcal(y)$.
Although local changes of variables are traditional in Calculus of Variation (see for instance \cite{olver86}), the change \eqref{eq:cv} is not local and the new functional will not be in a standard class. More precisely, $\Bcal$ takes the form
\[
\Bcal(y) = \int_0^1 \beta\left( \tau,y(\tau),y'(\tau),\int_0^\tau |y(\xi)|^2d\xi, \|y\|_2 \right)\,d\tau,
\]
where $\| \cdot \|_2$ is the norm in $L^2\left((0,1);\R^d\right)$. Furthermore, the transformation \eqref{eq:cv} does not induce a diffeomorphism between the natural classes of functions, and the consistency between \eqref{eq:perturbed_kepler}, 
$\Acal$ and $\Bcal$ becomes subtle. We will discuss such consistency in Section \ref{sec:ELeq}, establishing some regularized variational principles. As a rule of 
thumb, critical points of $\Bcal$ correspond to generalized solutions of \eqref{eq:perturbed_kepler} only imposing further conditions. In Section \ref{sec:counterexamples} we show that such conditions are sharp, providing several counterexamples.

Once the variational principles are established, the rest of the paper is devoted 
to applications, for either Dirichlet or periodic boundary conditions. First, in Section \ref{sec:Diri}, we 
consider the minimization $\Bcal$ in the Sobolev space
$H^1_0(0,1;\R^d)$. This will lead to the existence of a generalized solution of 
\eqref{eq:perturbed_kepler} satisfying the Dirichlet boundary conditions 
$x(0)=x(T)=0$. This can be done for any dimension $d\ge1$.

In our second application, illustrated in Section \ref{sec:periodicBLC}, we assume $d=2$ and adapt the definition of $\Phi$ in \eqref{eq:cv}
to the classical Levi-Civita change of variables. The consequence is the discovery of
a hidden symmetry: the functional $\Bcal = \BLC$ is even. When we look for $T$-periodic solutions of \eqref{eq:perturbed_kepler}, there are two possible choices for the domain of
$\BLC$. We can consider the Hilbert spaces $W_1$ and $W_{-1}$, where
\[
W_{\pm1} := \left\{z\in H^1(0,1;\C): z(1) =\pm z(0)\right\}.
\]
The most convenient choice is $W_{-1}$ because in this case $\BLC$ satisfies the Palais-Smale
condition. Classical minimax theory for even functionals implies that \eqref{eq:perturbed_kepler} has infinitely many generalized $T$-periodic solutions in dimension $d=2$. It is worthwhile noticing that anti-periodic functions $z$ produce periodic solutions $x$, without imposing any symmetry condition on the forcing $p$.

As we show in Section \ref{sec:periodicBKS}, an analogous conclusion is also valid in dimension $d=3$ but the proof is more delicate. The definition of $\Phi$ is now inspired by the Kustaanheimo-Stiefel change of variables. The
consistency of the periodic problem with $\Bcal = \BKS$ requires a more sophisticated
domain, a symmetric Hilbert manifold $\Mcal$ that will be defined later. Again $\BKS$ is even
but the Palais-Smale condition does not hold. Nonetheless, the more flexible Cerami
condition is satisfied, and the existence of infinitely many critical points of 
$\BKS$ over $\Mcal$ can be proved.

We already mentioned the different notions of generalized solution that have been introduced in 
the literature on variational methods. The classical  theory of holomorphic differential 
equations and the modern theory of dynamical systems have also led to other notions of 
generalized solutions (see \cite{Mc_CMH1981} for more information). The notion employed in the 
present paper is in perfect correspondence with the regularization theories by Levi-Civita and 
Kustaanheimo-Stiefel, but it has the advantage of having an intrinsic formulation. The possible 
relevance of this notion of solution for non-Newtonian potentials is a question to be analyzed. 

\tableofcontents

\subsection*{Notation}
%
\begin{itemize}
\item
$|\cdot|$ is the euclidean norm in $\R^d$ and, sometimes, the Lebesgue measure of a set.
\item
$\S^d := \{x \in \R^d : |x| = 1\}$.
\item
$\scalar{\cdot}{\cdot}$ is the euclidean scalar product in $\R^d$.
\item
$\| \cdot \|_p = \| \cdot \|_{L^p(0,1;\R^k)}$, with $k=1$ or $d$.
\item
$a_n = O (b_n)$ (as $n\to+\infty$) if, for some $N,C$, $n\geq N \implies |a_n|\leq C |b_n|$.
\end{itemize}

\section{Heuristic derivation of a class of regularized functionals}\label{sec:euristic}

Let us consider the \emph{action functional}, associated to the perturbed Kepler problem \eqref{eq:perturbed_kepler},
\[
\Acal(x) = \int_{0}^{T}\left[\frac12|\dot x(t)|^2+\frac{1}{|x(t)|} + \scalar{p(t)}{x(t)} \right]\,dt,
\]
which is (well defined and) finite in
\[
\Xh := \left\{x\in H^1(0,T;\R^d) \colon
\int_{0}^{T}\frac{dt}{|x(t)|}<\infty\right\}.
\]
We want to exploit some changes of variables (both for $t$ and $x$) in such a way that the action
functional in the new variables is \emph{regularized}, i.e. it does not contain a singular potential
term. The different choices for such changes of variables are inspired by classical regularization
techniques for the Kepler problem (Sundman, Levi-Civita, Kustaanheimo-Stiefel). Accordingly, we will
obtain different functionals.

As far as the time scale is concerned, the common change of variables we use is that introduced by
Sundman \cite{sundman1907}. For any $x\in\Xh$ let $L_{x}$ be the corresponding (strictly positive and
finite) quantity
\[
L_x := \int_{0}^{T} \frac{dt}{|x(t)|},
\]
and let us define
\begin{equation}\label{eq:ttau}
\tau = \tau(t) := \frac{1}{L_x}\int_{0}^t\frac{d\xi}{|x(\xi)|},\qquad t\in[0,T].
\end{equation}
In order to enlighten the properties of $\tau$, we use the following result.
\begin{lemma}\label{lem:elementary}
Let $a\from[0,T]\to\R$ be a continuous and non-negative function such that
\[
\frac{1}{a}\in L^1(0,T).
\]
Define
\[
A(t)=\int_0^t\frac{d\xi}{a(\xi)},\qquad t\in[0,T].
\]
Then $A$ is a homeomorphism between $[0,T]$ and
$[0,\Xi]$, with
\[
\Xi = \int_0^T\frac{d\xi}{a(\xi)}.
\]
Moreover, the inverse function $B=B(\tau)$ belongs to $C^1([0,\Xi])$ and satisfies
\[
B'(\tau)=a(B(\tau))\qquad\text{for each }\tau\in[0,\Xi].
\]
\end{lemma}
We postpone the proof of the above lemma to Appendix
\ref{app:lemma}. The application of such lemma to $a(t)=L_x |x(t)|$ implies that \eqref{eq:ttau} is invertible, and that the inverse function $t=t(\tau)$ is of class $C^1([0,1])$, with
\[
t'(\tau) = L_x|x(t(\tau))|,\qquad\text{for any }\tau\in[0,1].
\]

Let now $\Phi : \R^d \to \R^d$ be a map such that:
\begin{enumerate}
	\item[($\Phi 1$)] $\left.\Phi\right|_{\S^{d-1}}$ is smooth and $\Phi (\S^{d-1}) \subseteq \S^{d-1}$,
	\item[($\Phi 2$)] $\Phi(\lambda y) = \lambda^2 \Phi(y)$ $\forall y\in \R^d$ and $\lambda > 0$.
\end{enumerate}
Notice that, for every $y\neq0$,
\begin{equation}\label{eq:modPhi}
|\Phi(y)|=|y|^2\left|\Phi\left(\frac{y}{|y|}\right)\right|=|y|^2.
\end{equation}
Let $x \in \Xh$, $\tau$ defined as above, and let us assume at this point that we can find $y:[0,1] \to \R^d$ such that
\begin{equation}\label{eq:xphiy}
x(t) = \Phi(y(\tau(t))).
\end{equation}
Under such assumption we want to write $\mathcal{A}(x)$ in terms of $y$. To start with we remark that, by \eqref{eq:modPhi},
\[
\dot \tau(t) = \frac{1}{L_x\,|y(\tau(t))|^2},
\qquad
\frac{1}{L_x} = |y(\tau(t))|^2 \dot \tau(t)\qquad
\text{hence }
\frac{T}{L_x} = \int_0^1|y(\tau)|^2 d\tau.
\]
Then we can write
\[
L_x = \frac{T}{\|y\|^2_{L^2(0,1)}} =: \Lcal (y) \qquad\text{and}\qquad
t_y = t_y(\tau) = \Lcal(y) \int_{0}^{\tau}|y(\xi)|^2\,d\xi.
\]
Since, at least formally,
\[
\dot x(t) = D
\Phi(y(\tau(t)))y'(\tau(t))\dot{\tau}(t),
\]
the kinetic part transforms into
\[
\int_{0}^{T} \frac12 |\dot x(t)|^2 dt = \frac{1}{\Lcal(y)}\int_{0}^{1} \frac12
\frac{\scalar{[D
\Phi(y(\tau))]^TD
\Phi(y(\tau))y'(\tau)}{y'(\tau)}}{|y(\tau)|^2} \, d\tau=:\frac{1}{\Lcal(y)}
\Qcal(y).
\]
On the other hand,
\[
\int_{0}^{T} \frac{dt}{|x(t)|} = \Lcal(y)
\]
and
\[
\int_{0}^{T}
\scalar{p(t)}{x(t)} dt = \Lcal(y) \int_0^1 |y(\tau)|^2 \scalar{p(t_y(\tau))}{\Phi(y(\tau))}\, d\tau
=:\Lcal(y)\Rcal(y).
\]
Resuming we have that, at least formally, \eqref{eq:xphiy} implies $\Acal(x) = \Bcal(y)$, where
\begin{equation}\label{eq:preB}
\Bcal(y) := \frac{1}{\Lcal(y)}
\Qcal(y) +
\Lcal(y)\Big[ 1+ \Rcal(y)\Big].
\end{equation}
We notice that, in the functional above, only $\Qcal$ and $\Rcal$ depend on the actual choice of
$\Phi$.

As we mentioned, we will deal with three different choices of $\Phi$.

\begin{case}[the functional $\Bcal$ in any $d$]
The more direct choice for $\Phi$, which works in any dimension, consists in taking
$\left.\Phi\right|_{\S^{d-1}}$ to be the identity on the sphere. Then ($\Phi 1$) is obvious,
and ($\Phi 2$) forces
\[
\Phi(y) = |y|y.
\]
Then $D\Phi(y)$ is symmetric and
\[
[D\Phi(y)]^TD\Phi(y) = \left(\frac{y y^T }{|y|} + |y| \mathrm{Id}\right)^2 =
3 y y^T + |y|^2 \mathrm{Id}
\]
(indeed $( y y^T)^2=|y|^2yy^T $). We obtain that $\Bcal$ is as in \eqref{eq:preB}, with
\[
\begin{split}
\Qcal(y) &= \frac12\int_0^1
\left(3\frac{\scalar{y}{y'}^2}{|y|^2} + |y'|^2\right)\, d\tau,
\\
\Rcal(y) &= \int_0^1
|y|^3 \scalar{p\circ t_y}{y}\,d\tau.
\end{split}
\]
We recall that, given any $y\in H^1(0,1;\R^d)$, also $|y|\in H^1(0,1;\R)$ and its (weak)
derivative writes
\begin{equation}\label{eq:weak_der}
|y|' =
\begin{cases}
 \frac{\langle y,y'\rangle}{|y|}, \quad \text{if } y \neq 0, \\
0, \quad \text{if } y = 0.
\end{cases}
\end{equation}
Then we can write
\[
\Qcal(y) := \frac{1}{2} \int_0^1  \left[ 3(|y|')^2 + | y'|^2\right]\,d\tau,
\]
so that $\Bcal$ is well defined in $H^1(0,1; \R^d)\setminus\{0\}$ (otherwise $\Lcal(y)$ is not
defined). Actually, it is possible to read $\Bcal$ as an extended valued functional on
$H^1(0,1; \R^d)$ by choosing $\Bcal(0)=+\infty$.
\end{case}

\begin{case}[The Levi-Civita regularization in $d=2$ and the functional $\BLC$]
In dimension $d=2$ we can exploit the complex structure of the plane and define
\[
\Phi(z) = \Phi_{\text{LC}}(z) = z^2 \sim \left(\begin{array}{c} z_1^2-z_2^2 \\ 2z_1z_2\end{array}\right),\qquad z\in\C\cong\R^2.
\]
It is immediate to check that $\Phi_{\text{LC}}$ verifies ($\Phi 1$) and ($\Phi 2$), and that
\[
[D\Phi_{\text{LC}}(z)]^TD\Phi_{\text{LC}}(z)=4\left(\begin{array}{cc} z_1 & z_2 \\-z_2 & z_1\end{array}\right)
\left(\begin{array}{cc} z_1 & -z_2 \\z_2 & z_1\end{array}\right) =
4|z|^2 \mathrm{Id}.
\]
Writing $z(\tau)$ instead of $y(\tau)$ we obtain that in this case the functional $\Bcal$ writes
\[
\BLC(z) = \frac{2}{T}\int_0^1 |z|^2\,d\tau \int_0^1 |z'|^2\,d\tau + \frac{T}{\int_0^1|z|^2\,d\tau}
\left[1 + \int_0^1 |z|^2 \scalar{p\circ t_z}{z^2}\,d\tau \right].
\]
Also $\BLC$ is well defined in $H^1(0,1; \C)\setminus\{0\}$, and it can be extended as
$\BLC(0)=+\infty$.

\end{case}

\begin{case}[The Kustaanheimo-Stiefel regularization in $d=3$ and the functional $\BKS$]
As it is well known, the regularization in the three-dimensional case is more involved, as
it requires to consider $\Phi:\R^4\to\R^3$, induced by the quaternionic structure. To this
aim, following \cite{LZQuaternions}, we denote with $\HH$ the skew-field of quaternions, and
with $\II\HH$ the subset of purely imaginary quaternions:
\[
\HH:=\{z=z_0+z_1i+z_2j+z_3k:(z_0,z_1,z_2,z_3)\in\R^4\},\qquad \II\HH:=\{z\in\HH:\Re(z)=0\},
\]
where the real part of a quaternion is defined as $\Re(z_0+z_1i+z_2j+z_3k) = z_0$.
Notice that, in a trivial way, both $\HH\cong\R^4$ and $\II\HH\cong\R^3$, in the sense of (real)
vector spaces. Defining the conjugate of a quaternion as $\bar z := 2\Re(z) - z$, direct computations
show that, for any $z\in\HH$,
\[
\bar z i z =  (z^2_0 + z^2_1 - z^2_2 - z^2_3)i + 2(z_1 z_2 - z_0 z_3 )j + 2(z_1 z_3 + z_0 z_2 )k\in\II\HH.
\]
Then we define $\Phi:\R^4\to\R^3$ as
\[
\Phi(z) = \Phi_{\text{KS}}(z) = \bar z i z \sim \left(\begin{array}{c} z^2_0 + z^2_1 - z^2_2 - z^2_3 \\ 2(z_1 z_2 - z_0 z_3 )\\ 2(z_1 z_3 + z_0 z_2 )\end{array}\right),\qquad z\in\HH\cong\R^4.
\]
We observe that $\Phi_{\text{KS}}$ is invariant under the following action of $\S^1$:
\[
\Phi_{\text{KS}}(e^{i\vartheta} z) = \Phi_{\text{KS}}(z),\qquad\text{for every }\vartheta\in\R,\ z\in\HH.
\]
Since $|\bar z i z| = |z|^2$, this implies that $\Phi_{\text{KS}} \from \S^3 \to \S^2$ induces the Hopf fibration (for more details see \cite{cushman}). Although $\Phi_{\text{KS}}$ is not exactly in the previous framework ($\HH$ and $\II\HH$ have different dimension), $\Phi_{\text{KS}}$ is again homogeneous of degree 2 and ($\Phi 2$) and \eqref{eq:modPhi} make sense.
Now, assume that
\begin{equation}\label{eq:xphiz}
x(t) = \Phi_{\text{KS}}(z(\tau(t))).
\end{equation}
Then
\[
\scalar{[D\Phi_{\text{KS}}(z)]^TD\Phi_{\text{KS}}(z)z'}{z'} =\left|\bar z' i z + \bar z i z'\right|^2,
\]
and
\[
\int_{0}^{T} \frac12 |\dot x(t)|^2 dt = \frac{1}{\Lcal(z)}\int_{0}^{1} \frac12
\frac{\left|\bar z' i z + \bar z i z'\right|^2}{|z|^2} \, d\tau.
\]
Motivated by the role of the bilinear form $(z,w) \mapsto \Re(\bar wiz)$ in Kustaanheimo-Stiefel regularization, we assume that $z(\tau)$ satisfies the further condition
\begin{equation}\label{eq:varietaperKS}
\Re(\bar z' i z)=-\Re(\bar z i z')=\scalar{z'}{iz}
=-\scalar{iz'}{z}=0\qquad \text{for every }\tau,
\end{equation}
then $\left|\bar z' i z + \bar z i z'\right| = 2\left|\bar z' i z \right| = 2|z'||z|$.
As a consequence, with similar calculations as in the previous cases, we obtain that
\eqref{eq:xphiz} and \eqref{eq:varietaperKS} imply $\Acal(x) = \BKS(z)$, where
\begin{equation*}
\BKS(z) = \frac{2}{T}\int_0^1 |z|^2\,d\tau \int_0^1 |z'|^2\,d\tau + \frac{T}{\int_0^1|z|^2\,d\tau}
\left[1 + \int_0^1 |z|^2 \scalar{p\circ t_z}{\bar z i z}\,d\tau \right].
\end{equation*}
Notice that, since $\bar z i z \in \II\HH$, in the expression of $\BKS$ also $\R^3\ni p=(p_1,p_2,p_3)$ can be interpreted as an element
of $\II\HH$ (even though any other choice of $p_0$ has no effect on the functional).
\end{case}
\begin{remark}[On the connection between $\BLC$ and $\BKS$ when $p$ takes values in $\R^2$]\label{rem:lcplane}
Let $\Pi\subset \HH$ be a Levi-Civita plane. This means that $\Re(\bar w i z) = 0$
for every $w,z\in\Pi$ ($\Pi$ is indeed a Lagrangian plane with respect to the
corresponding symplectic structure). Let $\{r_1,r_2\}$ be an orthonormal basis of $\Pi$. Define
\[
\hat r_1 = \Phi_{\text{KS}} (r_1) = \bar r_1 i r_1,
\qquad
\hat r_2 = \bar r_1 i r_2.
\]
Then $\hat r_1,\hat r_2\in \II\HH$ are linearly independent unit vectors (see \cite[Lemma 3.4]{LZQuaternions}).

Given $c_1,c_2\in\R$ and $z= c_1 r_1 + c_2 r_2 \in \Pi$, it turns out that
\(
\Phi_{\text{KS}} (z) = (c_1^2 - c_2^2) \hat r_1 + 2c_1c_2 \hat r_2.
\)
Let $\hat\Pi$ be the plane spanned by $\hat r_1, \hat r_2$. Then $\Phi_{\text{KS}}
(\Pi) = \hat \Pi$. Define the isomorphisms of real vector spaces
\[
\begin{split}
\varphi &\from \C \to \Pi,
\qquad \varphi(1) = r_1, \ \varphi(i) = r_2,\\
\psi &\from \hat \Pi \to \C,
\qquad \psi(\hat r_1)=1, \ \psi(\hat r_2)= i.
\end{split}
\]
Then $\psi\circ\Phi_{\text{KS}}\circ\varphi (z) = z^2 = \Phi_{\text{LC}}(z)$.

Let $\psi^*\from \C \to \hat \Pi$ be the adjoint of $\psi$:
\[
\scalar{\psi(p)}{z}_{\C} = \scalar{p}{\psi^*(z)}_{\HH}
\qquad \text{if }z\in\C,\ p\in\hat\Pi.
\]
Let us now assume that $p\from [0,T] \to \C$, $p=p(t)$, and let us consider
$p^* \from [0,T] \to \hat\Pi \subset\II\HH$, $p^* = \psi^* \circ p$. Consider
the functionals $\BKS$, associated to $p^*(t)$,  and $\BLC$, associated to $p(t)$.
Then, given $z\in H^1(0,1;\C)$, $\varphi\circ z \in H^1(0,1;\Pi)$, for some time interval $I$, we have that
\[
\Bcal_{\text{KS},p^*}(\varphi\circ z) = \Bcal_{\text{LC},p}(z).
\]
Hence $\BLC$ can be seen as $\BKS$ restricted to $H^1(0,1;\Pi)$.
\end{remark}

\section{Critical points of the regularized functionals and generalized solutions of (\ref{eq:perturbed_kepler})}\label{sec:ELeq}

In this section we deal with critical points of the functionals defined in Section
\ref{sec:euristic}, and with their relations with generalized solutions of the perturbed Kepler problem (\ref{eq:perturbed_kepler}).

Throughout the section we assume that $p$ is a function of class $C^1$. As we will see, this will imply that the term $\Rcal$ is of class $C^1$ in each regularized functional.
On the contrary, this is not true for the term $\Qcal$: as a matter of fact, the main difference between $\Bcal$ and $\BLC$, $\BKS$, is that, while the latter are differentiable at any non identically zero function, the former needs not to be Gateaux-differentiable at $y\in H^1(0,1;\R^d)$, whenever $y$ vanishes at some point.

Along this section, we do not take into account boundary conditions. For this reason, we consider points $y$ in $H^1$, which are critical with respect to smooth variations $\varphi$, compactly supported in $(0,1)$ (or outside the collision set of $y$, for the functional $\Bcal$). Of course, to impose boundary conditions, one has to choose critical points $y$ in a suitable subspace of $H^1$, and/or variations $\varphi$ in a suitable space containing $\Dcal = C^\infty_0$. This will be done, case by case, in the subsequent sections.

For each functional, once the Euler-Lagrange equations are derived, we analyze when their solutions actually correspond to solutions of the perturbed Kepler problem. As we will see, both for $\Bcal$ and for $\BKS$ this will require further conditions. In the last subsection
we are going to show that such conditions are sharp, providing some counterexamples.

\subsection{Euler-Lagrange equations for \texorpdfstring{$\Bcal$}{B} outside collisions}
\label{sec:ELeqB}

First we choose $\Phi(y) = |y|y$, see Case 1 in the previous section, and we deal with the functional
\[
\Bcal : H^1(0,1; \R^d) \to \R \cup \{+\infty\}
\]
\begin{equation}\label{eq:B}
\Bcal(y) := \frac{1}{\Lcal(y)}
\Qcal(y) +
\Lcal(y)\Big[ 1+ \Rcal(y)\Big],
\qquad \Bcal(0):=+\infty,
\end{equation}
where
\begin{equation}\label{eq:tutti gli altri}
\begin{split}
\Lcal(y) &:= \frac{T}{\int_0^1|y|^2\,d\tau},\\
\Qcal(y) &:= \frac{1}{2} \int_0^1  \left[ 3(|y|')^2 + | y'|^2\right]\,d\tau,
\\
\Rcal(y) &:= \int_0^1
|y|^3 \scalar{p\circ t_y}{y}\,d\tau,\\
t_y(\tau) &:= \Lcal(y) \int_{0}^{\tau}|y(\xi)|^2\,d\xi.
\end{split}
\end{equation}

As we will see, in case $y(\bar\tau)=0$ and $\varphi\in\Dcal(0,1;\R^d)$, $\varphi(\bar\tau)\neq0$,
it is not clear whether the function
\[
\eps\mapsto\Qcal(y+\eps \varphi)
\]
is differentiable at $\eps=0$ (see Remark \ref{rem:diffQ} ahead). For this reason, when searching for the Euler-Lagrange equation associated to $\Bcal$,
it is natural to work on intervals where its argument $y$ is collision free.

We will show the following.
\begin{proposition}\label{prop:EL_for_B}
Let $y\in H^1(0,1; \R^d)$ be such that
\[
|y(\tau)|>0\quad\text{ for }\tau\in(\tau_1,\tau_2)\subset[0,1]
\]
and
\begin{equation}\label{eq:varBvphi}
\frac{d}{d\eps} \left[\Bcal(y+\eps \varphi)\right]_{\eps=0} = 0\qquad \text{for every }\varphi\in\Dcal(\tau_1,\tau_2;\R^d).
\end{equation}
Then the map $\tau \mapsto t_y(\tau)$ is $C^3(\tau_1,\tau_2)\cap C^1([\tau_1,\tau_2])$, with inverse $t\mapsto \tau_y(t)$
which is $C^3(t_1,t_2)\cap C([t_1,t_2])$, where $t_i = t_y(\tau_i)$. Moreover, writing
\[
x(t)=x_y(t)=|y(\tau_y(t))|y(\tau_y(t)),
\]
we have that there exists a constant $\mu\in\R$ such that
\begin{equation}\label{eq:beta}
\ddot x = -\mu \frac{x}{|x|^3} + p(t),\qquad t\in(t_1,t_2).
\end{equation}
Finally, in case $(\tau_1,\tau_2)=(0,1)$, \eqref{eq:beta} holds true in $(0,T)$ with $\mu=1$.
\end{proposition}
Notice that in Section \ref{sec:euristic} we obtained $\Bcal$ starting from the action related to the
Kepler problem \eqref{eq:perturbed_kepler}. Here, to go back to \eqref{eq:perturbed_kepler}, we have two main problems: the first one is that we have to restrict
to collisionless intervals; the second one is that ``critical points'' $y$ of $\Bcal$ (in the sense of the above
proposition) solve \eqref{eq:beta}, which agrees with \eqref{eq:perturbed_kepler} only when $\mu=1$. As we mentioned, we will show that the absence of internal collisions implies also $\mu=1$. On the other hand, in case an
internal collision occurs, one can only expect that \eqref{eq:beta} holds true on each collisionless subinterval, with
$\mu\ge0$ possibly depending on the interval. Furthermore, even though $x$ satisfies \eqref{eq:beta} with $\mu=1$ on any
collisionless subinterval, it may fail to be a generalized solution at collisions. Examples in these directions are
provided in Section \ref{sec:counterexamples}.

We will prove Proposition \ref{prop:EL_for_B} through a sequence of lemmas. As a first step, we determine the Euler-Lagrange equation associated to $\Bcal$.
\begin{lemma}\label{lem:EL_for_B}
Let $y=y(\tau)$ satisfy the assumptions of Proposition \ref{prop:EL_for_B}. Then,
in distributional sense in $(\tau_1,\tau_2)$,
\begin{equation}\label{eq:ELeqB}
\frac{d}{d\tau}\left(y' + 3 \frac{\scalar{y}{y'}}{|y|^2}y  \right) =  \alpha_y +
\frac{2}{T}\Lcal\left(\Qcal-\Lcal^2(1+\Rcal)\right) y + \beta_y,
\end{equation}
where $\Lcal$, $\Qcal$ and $\Rcal$ are evaluated at $y$ and
\[
\begin{split}
\alpha_y &:=3\left( \frac{\langle y,y' \rangle}{|y|^2}y' -
\frac{\langle y,y' \rangle^2}{|y|^4}y\right)\\
\beta_y  &:=\Lcal^2\left[ 3|y| \scalar{p\circ t_y}{y} y + |y|^3 p\circ t_y  \right]
+ \frac{2\Lcal^4}{T} y \, \Gamma_y,
\end{split}
\]
where
\[
\Gamma_y(\tau) = \int_{0}^{1}|y(\xi)|^2
\left( \int_{\tau}^{\xi}|y(s)|^3\scalar{\dot p ( t_y(s))}{y(s)}\,ds\right)\,d\xi.
\]
\end{lemma}
\begin{proof}
Notice that, by assumption, $\Bcal(y)<+\infty$. Let $\varphi\in\Dcal(\tau_1,\tau_2)$.
In particular, the support of $\varphi$ is contained in some $[\hat\tau_1,\hat\tau_2]\subset(\tau_1,\tau_2)$, where
$y(\tau)$ does not vanish, and the functions $\eps\mapsto \Lcal(y+\eps \varphi)$, $\eps\mapsto \Qcal(y+\eps \varphi)$
and $\eps\mapsto \Rcal(y+\eps \varphi)$ are differentiable at $\eps=0$. Recalling that $\Lcal(y) = {T}\|y\|_2^{-2}$, we have
\[
\frac{d}{d\eps} \left[\frac{1}{\Lcal(y+\eps \varphi)}\right]_{\eps=0} =
\frac{2}{{T}} \int_{\tau_1}^{\tau_2} \langle y,\varphi \rangle,
\qquad
\frac{d}{d\eps}\left[\Lcal(y+\eps \varphi)\right]_{\eps=0}  =
-\frac{2}{{T}}\Lcal^2(y) \int_{\tau_1}^{\tau_2} \langle y,\varphi \rangle;
\]
and
\[
\frac{d}{d\eps}\left[t_{y+\eps \varphi}\right]_{\eps=0}  =
-\frac{2}{{T}}\Lcal^2(y) \left(\int_{\tau_1}^{\tau_2} \langle y,\varphi \rangle \right)\int_{0}^{\tau}|y|^2 + 2\Lcal(y) \int_{\tau_1}^{\tau}\scalar{y}{\varphi}.
\]
Furthermore
\begin{equation}\label{eq:diffQ}
\frac{d}{d\eps} \left[ \Qcal(y+\eps\varphi)\right]_{\eps=0}
= \int_{\tau_1}^{\tau_2} \left( \left\langle \alpha_y,\varphi\right\rangle
+ \left\langle 3\frac{\langle y,y' \rangle}{|y|^2}y +y',\varphi'\right\rangle \right).
\end{equation}
On the other hand,
\[
\begin{split}
\frac{d}{d\eps} \left[\mathcal{R}(y+\eps \varphi)\right]_{\eps=0}
= &\int_{\tau_1}^{\tau_2} \left[
3|y|\langle y,\varphi \rangle \scalar{p\circ t_y}{y} + |y|^3
\scalar{p\circ t_y}{\varphi} \right]\\
 &+\int_{0}^{1} |y|^3 \scalar{\dot p \circ t_y}{y} \left(-\frac{2}{{T}}\Lcal^2(y) \int_{\tau_1}^{\tau_2}\langle y,\varphi \rangle \int_0^{\tau}|y|^2
+ 2\Lcal(y) \int_{\tau_1}^{\tau} \langle y,\varphi \rangle \right)
 \, d\tau\\
= & \int_{\tau_1}^{\tau_2}
\scalar{3|y| \scalar{p\circ t_y}{y} y + |y|^3 p\circ t_y}{\varphi} d\tau\\
& +\frac{2\Lcal^2(y)}{T}\int_{0}^{1}|y|^3 \scalar{\dot p \circ t_y}{y} \left(-\int_{\tau_1}^{\tau_2}\langle y,\varphi \rangle \int_0^{\tau}|y|^2
+ \int_{\tau_1}^{\tau} \langle y,\varphi \rangle \int_0^{1}|y|^2
\right) \, d\tau.
\end{split}
\]
Noticing that
\[
\begin{split}
\int_0^1 \left(\int_{\tau_1}^{\tau_2} \left(\int_0^\tau a(\tau) b(\xi) c(s) \,ds\right)\,d\xi\right)\,d\tau
&= \int_{\tau_1}^{\tau_2} b(\xi) \left(\int_0^1 c(s) \left(\int_s^1 a(\tau)\,d\tau\right)\,ds\right)\,d\xi\\
\int_0^1 \left(\int_{\tau_1}^{\tau} \left(\int_0^1 a(\tau) b(\xi) c(s) \,ds\right)\,d\xi\right)\,d\tau
&= \int_{\tau_1}^{1} b(\xi) \left(\int_0^1 c(s) \left(\int_\xi^1 a(\tau)\,d\tau\right)\,ds\right)\,d\xi\\
&= \int_{\tau_1}^{\tau_2} b(\xi) \left(\int_0^1 c(s) \left(\int_\xi^1 a(\tau)\,d\tau\right)\,ds\right)\,d\xi
\end{split}
\]
provided $b(\tau)\equiv0$ on $[\tau_2,1]$, we obtain that the last line in the previous identity can be rewritten as
\[
\frac{2\Lcal^2(y)}{T}\int_{\tau_1}^{\tau_2}\left[\langle y(\tau),\varphi(\tau)\rangle \int_{0}^{1}|y(\xi)|^2
 \left( \int_{\tau}^{\xi}|y|^3\scalar{\dot p \circ t_y}{y}\,ds\right)\,d\xi \right]\, d\tau.
\]
from which the lemma follows.
\end{proof}
\begin{remark}\label{rem:diffQ}
Notice that, in the previous lemma, the assumption $|y(\tau)|>0$ on $(\tau_1,\tau_2)$ is essential, because it is not clear
whether  $\Qcal$ may be differentiable or not in case of collisions. More precisely, the terms in $\alpha_y$, equation \eqref{eq:diffQ}, i.e.
\[
\frac{\langle y,y' \rangle}{|y|^2}y'
\qquad\text{ and }
\frac{\langle y,y' \rangle^2}{|y|^4}y,
\]
need not to be $L^1$ if $y$ vanishes somewhere in $(\tau_1,\tau_2)$. On the other hand, cancellations may
occur, so that $\alpha_y$ may be $L^1$ also when collisions occur.
\end{remark}
To show regularity of $y$ we need the following lemma.
\begin{lemma}\label{lem:prelemma_regularity}
Let $I$ be an open interval, $y\in H^1(I;\R^d)$ be such that
\[
y(\tau)\neq 0\qquad \text{if }\tau\in \overline{I}
\]
and, for some $\lambda\neq-1$,
\[
y'+\lambda\frac{\scalar{y}{y'}}{|y|^2}y \in W^{1,1}(I)
\]
(i.e. it is absolutely continuous on $\overline{I}$).
Then also $y'\in W^{1,1}(I)$.
\end{lemma}
\begin{remark}
The condition $\lambda\neq-1$ is essential. Consider $y(\tau) = |\tau|U$, with $U\in\S^{d-1}$ constant and
$I=(-1,1)$. Then
\[
y'-\frac{\scalar{y}{y'}}{|y|^2}y =0,
\]
but $y'$ is not continuous.
\end{remark}
\begin{proof}
Define
\[
r(\tau)=|y(\tau)|,\qquad U(\tau) = \frac{1}{r(\tau)}y(\tau).
\]
Then $r\in H^1(I)$ with $r'=\scalar{y}{y'}/|y|$, and  $U\in H^1(I;\R^d)$ with
$\scalar{U}{U'}=0$ a.e. in $\overline{I}$. By assumption
\[
w:= y'+\lambda\frac{\scalar{y}{y'}}{|y|^2}y = (1+\lambda) r' U + r U' \in W^{1,1}(I).
\]
Since the space of absolutely continuous functions is a Banach algebra, we infer that
\[
(1+\lambda)r' = \scalar{w}{U}\in W^{1,1}(I).
\]
Using the assumption $1+\lambda\neq 0$, we deduce that $\lambda r' U\in W^{1,1}(I)$ too, and finally
\[
y' = w - \lambda r' U \in W^{1,1}(I).\qedhere
\]
\end{proof}
\begin{corollary}\label{coro:first_regularity}
Let $y=y(\tau)$ satisfy the assumptions of Proposition \ref{prop:EL_for_B}, and let $[\hat\tau_1,\hat\tau_2]\subset(\tau_1,\tau_2)$.
Then $y\in W^{2,1}(\hat\tau_1,\hat\tau_2)$.
\end{corollary}
\begin{proof}
It follows from Lemma \ref{lem:EL_for_B} and Lemma \ref{lem:prelemma_regularity} (with $\lambda=3$), after noticing that
$\alpha_y\in L^1(\hat\tau_1,\hat\tau_2)$ and $\beta_y\in C([\hat\tau_1,\hat\tau_2])$ (recall that $p$ is $C^1$).
\end{proof}
\begin{lemma}\label{lem:EL_for_B_bis}
Let $y=y(\tau)$ satisfy the assumptions of Proposition \ref{prop:EL_for_B}. Then $y\in C^2(\tau_1,\tau_2)$ satisfies
\begin{equation}\label{eq:ELeqB bis}
y'' = \gamma_y +
\frac{\Lcal}{2T}\left(\Qcal-\Lcal^2(1+\Rcal)\right) y + \delta_y,
\end{equation}
where $\Lcal$, $\Qcal$ and $\Rcal$ are evaluated at $y$ and
\[
\begin{split}
\gamma_y &:=\frac34 \frac{\langle y,y' \rangle^2 - |y'|^2|y|^2}{|y|^4}y\\
\delta_y  &:=\Lcal^2|y|^3 p\circ t_y
+ \frac{\Lcal^4}{2 T} y \,\Gamma_y,
\end{split}
\]
where $\Gamma_y$ has been introduced in Lemma \ref{lem:EL_for_B}.
\end{lemma}
\begin{proof}
By Corollary \ref{coro:first_regularity} we have that, both weakly and a.e. in $[\hat\tau_1,\hat\tau_2]\subset(\tau_1,\tau_2)$,
\begin{multline}\label{eq:ELeqBsbagliata}
\frac{1}{\Lcal}\left(-y'' +3 \frac{\langle y,y' \rangle^2 - |y'|^2|y|^2 - \langle y,y'' \rangle|y|^2 }{|y|^4}y \right) + \frac{2}{{T}}\left[\Qcal - \Lcal^2(1+\Rcal)\right] \, y +
\\
+ \Lcal \left[ 3|y| \scalar{p\circ t_y}{y} y + |y|^3 p\circ t_y  \right]
+ \frac{2\Lcal^3(y)}{T} y \int_{0}^{1}|y(\xi)|^2
 \left( \int_{\tau}^{\xi}|y|^3\scalar{\dot p \circ t_y}{y}\right)\,d\xi= 0.
\end{multline}
Then we multiply \eqref{eq:ELeqBsbagliata} with $y$ in order to solve for $\scalar{y}{y''}$ and substitute in \eqref{eq:ELeqBsbagliata} itself. After some cancellations, we deduce that \eqref{eq:ELeqB bis} holds, weakly and a.e. in $[\hat\tau_1,\hat\tau_2]$. Since
$y\in W^{2,1}$ and the functions $\gamma_y$ and $\delta_y$ are continuous, we obtain that $y\in C^2([\hat\tau_1,\hat\tau_2])$. Since in $[\hat\tau_1,\hat\tau_2]$ is arbitrary, the lemma follows.
\end{proof}

\begin{lemma}\label{lem:regularityEL}
Let $y$, $t_y$ and $x$ be as in Proposition \ref{prop:EL_for_B}. Then:
\begin{itemize}
\item $t_y\in C^3(\tau_1,\tau_2) \cap C^1([\tau_1,\tau_2])$, with inverse $\tau_y\in C^3(t_1,t_2)\cap C([t_1,t_2])$;
\item $x\in C^2(t_1,t_2)$.
\end{itemize}
\end{lemma}
\begin{proof}
Once the regularity of $y$ is proved as in Lemma \ref{lem:EL_for_B_bis}, the claims follow by the chain rule and the elementary
inverse function theorem.
\end{proof}
\begin{lemma}\label{lem:ELconmu}
Let $x$ be as in Proposition \ref{prop:EL_for_B}. Then there exists $\mu\in\R$ such that
\eqref{eq:beta} holds true.
\end{lemma}
\begin{proof}
Notice that, in $(\tau_1,\tau_2)$ and
$(t_1,t_2)$ respectively,
\[
x(t)=|y(\tau_y(t))|y(\tau_y(t))
\qquad\iff\qquad
y(\tau) = |x(t_y(\tau))|^{-1/2} x(t_y(\tau))
\]
where $y $ satisfy \eqref{eq:ELeqB bis}. Our aim is to substitute the second relation above into
\eqref{eq:ELeqB bis}; this can be done by the regularity properties obtained in Lemma
\ref{lem:regularityEL}. At the end, no explicit dependence
on $\tau$ will appear, and substituting $t=t_y(\tau)$ we will obtain the differential equation for
$x=x(t)$.

We have that $|y(\tau)| = |x(t(\tau))|^{1/2}$, while
\[
\begin{split}
y'(\tau)  &= \Lcal\left( -\frac12 |x|^{-3/2}\langle {x}/{|x|},\dot x \rangle x +
              |x|^{-1/2}\dot x \right)|x| =
         \Lcal \left( -\frac12 |x|^{-3/2}\langle x,\dot x\rangle x + |x|^{1/2}\dot x \right)\\
y''(\tau) &= \Lcal^2\left( \frac34 |x|^{-5/2}\langle x,\dot x\rangle^2 x -\frac12 |x|^{-1/2}
|\dot x|^2 x-\frac12 |x|^{-1/2}\langle x,\ddot x\rangle x +|x|^{3/2}\ddot x
\right),
\end{split}
\]
In order to substitute in the first line of \eqref{eq:ELeqB bis}, we compute separately
\[
\begin{split}
\langle y,y' \rangle ^2 &= \frac{\Lcal^2}{4}\langle x,\dot x\rangle^2
\\
|y|^2|y'|^2 &= |x| \Lcal^2 \left( |x||\dot x|^2 -\frac34|x|^{-1}\langle x,\dot x \rangle^2 \right)
= \Lcal^2 \left( |x|^2|\dot x|^2 -\frac34\langle x,\dot x \rangle^2 \right).
\end{split}
\]
Hence
\[
y'' - \gamma_y
= \Lcal^2 |x|^{3/2} \left(\ddot x - \frac12 |x|^{-2}\langle x,\ddot x\rangle x + \frac14|x|^{-2}|\dot x|^2 x
\right).
\]
On the other hand, by a change of variables in the integrals,
\[
\Lcal = \int_{0}^{T}\frac{dt}{|x|},
\qquad
\Qcal = \frac{\Lcal}{2} \int_{0}^{T}|\dot x|^2\,dt,
\qquad
\Rcal = \frac{1}{\Lcal} \int_{0}^{T}\scalar{p}{x}\,dt,
\]
while
\[
\Lcal^2 |y|^3 \left(p\circ t_y   \right)=
\left(\Lcal^2 |x|^{3/2} p\right)\circ t_y
\]
and
\[
\begin{split}
\Lcal^2\int_{0}^{1} |y(\xi)|^2
\left(\int_\tau^{\xi}|y|^3\scalar{\dot p \circ t_y}{y}\right)\,d\xi  &=
 \int_{0}^{T}\left( \int_{t}^s\scalar{\dot p }{x}\right)\,ds\\
&=
\int_{0}^{T}\left( \int_{t}^T\scalar{\dot p }{x}\right)\,ds -\int_{0}^{T}\left( \int_{s}^T\scalar{\dot p }{x}\right)\,ds\\
&=
T\int_{t}^T\scalar{\dot p }{x} -\int_{0}^{T}t \scalar{\dot p }{x}\,dt
\end{split}
\]
We conclude that equation \eqref{eq:ELeqB bis} transforms into the following equation for $x=x(t)$:
\[
\ddot x - \frac12 |x|^{-2}\langle x,\ddot x\rangle x + \frac14|x|^{-2}|\dot x|^2 x = \frac12\left[ \frac{1}{{T}} \int_{0}^{T}\left(\frac12 |\dot x|^2 - \frac{1}{|x|} -\scalar{p}{x} -
t\scalar{\dot p}{x} \right)\,dt + \int_{t}^{T}\scalar{\dot p}{x} \right]|x|^{-2}x + p,
\]
that is
\begin{equation}\label{eq:ELx}
\ddot x  = \frac12\left[ \langle x,\ddot x\rangle - \frac12|\dot x|^2   + C + \int_{t}^{T}\scalar{\dot p}{x} \right]|x|^{-2}x + p
\end{equation}
where
\begin{equation}\label{eq:defofC}
C := \frac{1}{{T}} \int_{0}^{T}\left(\frac12 |\dot x|^2 - \frac{1}{|x|} -\scalar{p}{x} -
t\scalar{\dot p}{x} \right)\,dt.
\end{equation}
In particular, multiplying \eqref{eq:ELx} by $x$ we obtain
\begin{equation*}\label{eq:senzabeta1}
 \langle x,\ddot x\rangle = -\frac12|\dot x|^2 + C + \int_{t}^{T}\scalar{\dot p}{x} +2\scalar{p}{x}
\end{equation*}
and, substituting into \eqref{eq:ELx},
\begin{equation}\label{eq:senzabeta2}
\ddot x  =
\left[ C -\frac12 |\dot x|^2+ \scalar{p}{x}+ \int_{t}^{T}\scalar{\dot p}{x}
\right]|x|^{-2} x  +p,\qquad t\in(t_1,t_2).
\end{equation}
Let $\beta=\beta(t)$ be such that $\ddot x - p = \beta x$. Then
\begin{equation}\label{eq:passaggio in ELx}
\beta |x|^2 = C - \frac12|\dot x|^2 + \scalar{p}{x} +  \int_{t}^{T}\scalar{\dot p}{x} ,\qquad t\in(t_1,t_2).
\end{equation}
Since $x$ is $C^2$, we have that $\beta$ is of class $C^1$ in $(t_1,t_2)$.
Hence we can differentiate to obtain
\[
\dot\beta|x|^2 + 2\beta \langle x,\dot x\rangle =
-\scalar{\ddot x - p}{\dot x} = -\scalar{\beta x}{\dot x} ,\qquad t\in(t_1,t_2).
\]
Hence $\dot \beta |x|^2 +3 \beta \langle x,\dot x\rangle=0$, which implies
$\beta  = -\mu/|x|^3$ for some constant $\mu\in\R$.
\end{proof}

\begin{remark}
By the proof of Proposition \ref{prop:EL_for_B} we have that the constant $\mu$ appearing in
\eqref{eq:beta} is related to the constant $C$ defined in \eqref{eq:defofC}. More precisely,
substituting \eqref{eq:beta} into \eqref{eq:senzabeta2} we infer
\begin{equation}\label{eq:muvsC}
C =  \frac12 |\dot x|^2 - \frac{\mu}{|x|} -  \scalar{p}{x}
- \int_{t}^{T}\scalar{\dot p}{x}, \qquad t\in(t_1,t_2).
\end{equation}
\end{remark}
\begin{lemma}\label{coro:mu}
Under the assumptions of Proposition \ref{prop:EL_for_B} we have that
\begin{itemize}
\item if either $y(\tau_1^+)=0$ or $y(\tau_2^-)=0$ then, in \eqref{eq:beta}, $\mu\ge0$;
\item if $(\tau_1,\tau_2)=(0,1)$ (with either collisions at the extrema or not) then, in \eqref{eq:beta}, $\mu=1$.
\end{itemize}
\end{lemma}
\begin{proof}
First of all, for concreteness, let us assume $y(\tau_1^+)=0$. Then, by \eqref{eq:muvsC} we obtain
\[
-\frac{1}{|x|}\mu \leq C + \scalar{p}{x} +  \int_{t}^{T}\scalar{\dot p}{x}.
\]
Taking the limit as $t\to t_1^+$ we deduce that the constant $\mu$ can not be strictly negative.

On the other hand, let $(\tau_1,\tau_2)=(0,1)$. Then \eqref{eq:muvsC} holds true with $(t_1,t_2)=(0,T)$. Integrating on $(0,T)$ and recalling the definition of $C$ in \eqref{eq:defofC} we obtain that $\mu = 1$.
\end{proof}
\begin{proof}[End of the proof of Proposition \ref{prop:EL_for_B}]
The proposition follows by Lemmas \ref{lem:regularityEL}, \ref{lem:ELconmu} and
\ref{coro:mu}.
\end{proof}

\subsection{Analysis of \texorpdfstring{$\BLC$}{BLC}}\label{sec:ELeqBLC}
Now, for $z=z_1+i z_2\in \C\cong\R^2$, let $\Phi(z) = \Phi_{\text{LC}}(z) = z^2$, see Case 2 in the Section
\ref{sec:euristic}. The corresponding functional is
\[
\BLC : H^1(0,1; \C) \to \R \cup \{+\infty\}
\]
\begin{equation}\label{eq:BLC}
\BLC(z) := \frac{1}{\Lcal(z)}
\Qcal(z) +
\Lcal(z)\Big[ 1+ \Rcal(z)\Big],
\qquad \Bcal(0):=+\infty,
\end{equation}
where
\begin{equation}\label{eq:tutti gli altriLC}
\begin{split}
\Lcal(z) &:= \frac{T}{\int_0^1|z|^2\,d\tau},\\
\Qcal(z) &:= 2 \int_0^1 | z'|^2\,d\tau,
\\
\Rcal(z) &:= \int_0^1
|z|^2 \scalar{p\circ t_z}{z^2}\,d\tau,\\
t_z(\tau) &:= \Lcal(z) \int_{0}^{\tau}|z(\xi)|^2\,d\xi.
\end{split}
\end{equation}
The main difference with respect to the previous section consists in the fact that $\BLC$ is now of class $C^1$ in the whole $H^1(0,1;\C)\setminus\{0
\}$, regardless of possible collisions.
\begin{proposition}\label{prop:EL_for_B_LC}
Let $z\in H^1(0,1;\C)\setminus\{0\}$ satisfy
\begin{equation*}
\frac{d}{d\eps} \left[\BLC(z+\eps \varphi)\right]_{\eps=0} = 0\qquad \text{for every }\varphi\in\Dcal(0,1;\C).
\end{equation*}
Then $z \in C^4([0,1])$, the map $\tau \mapsto t_z(\tau)$ is invertible on $[0,1]$ with inverse $t\mapsto \tau_z(t)$,
and
\[
x(t)=z^2(\tau_z(t))
\]
is a generalized solution of equation \eqref{eq:perturbed_kepler}.
\end{proposition}
To prove Proposition \ref{prop:EL_for_B_LC}, as a first step, we determine the Euler-Lagrange equation associated to $\BLC$.
\begin{lemma}\label{lem:EL_for_B_LC}
Let $z\in H^1(0,1; \C) \setminus \{0\}$ be as in Proposition \ref{prop:EL_for_B_LC}, then
\begin{equation}\label{eq:ELeqB_LC}
z'' =
\frac{\Lcal}{2T}\left(\Qcal-\Lcal^2(1+\Rcal)\right) z + \delta_z,
\end{equation}
where $\Lcal$, $\Qcal$ and $\Rcal$ are evaluated at $z$ and
\[
\delta_z  := \frac{\Lcal^2}{2}\left[ \scalar{p\circ t_z}{z^2} +\bar z^2(p\circ t_z)
+ \frac{\Lcal^2}{T} \Delta_z\right]z,
\]
where
\[
\Delta_z  := \int_{0}^{1}|z(\xi)|^2
\left(\int_{\tau}^{\xi}|z(s)|^2\scalar{\dot p(t_z(s))}{z^2(s)}\right)\,d\xi.
\]
In particular $z \in C^3([0,1])$.
\end{lemma}
\begin{proof}
The proof mainly retraces the one of Lemmas \ref{lem:EL_for_B} and \ref{lem:EL_for_B_bis}. The main difference consists in the term arizing from the differentiation of $\Qcal(z)$, indeed in this case the analogous of \eqref{eq:diffQ} simplifies into
\[
\frac{d}{d\eps} \left[ \Qcal(y+\eps\varphi)\right]_{\eps=0}
= 4\int_{0}^{1} \scalar{z'}{\varphi'},
\]
for every $\varphi\in\Dcal(0,1;\C)$.
Regularity of $z$ and equation \eqref{eq:ELeqB_LC} follow at once (recall that we are assuming $p\in C^1$).
\end{proof}
\begin{remark}\label{rem:BLC'}
For future purposes we notice that, for every $z\not\equiv0$, $v\in H^1(0,1; \C)$,
\[
\BLC'(z)[v]= \frac{4}{\Lcal} \int_{0}^{1} \left[\scalar{z'}{v'} + \scalar{\frac{\Lcal}{2T}\left(\Qcal-\Lcal^2(1+\Rcal)\right) z + \delta_z}{v}\right],
\]
where $\delta_z$ has been introduced in the previous lemma.
\end{remark}
\begin{lemma}\label{le:zeri}
Let $z$ be a critical point of $\BLC$ and $\tau^* \in [0,1]$ be such that $z(\tau^*)=0$. Then $z'(\tau^*)\neq 0$ and the set
$Z := \{\tau \in [0,1] : z(\tau)=0 \}$ is finite.
\end{lemma}
\begin{proof}
Notice that equation \eqref{eq:ELeqB_LC} can be written as
\begin{equation}
\label{eq:EL_LC_veloce}
z'' = g(\tau)z
\end{equation}
for some continuous, complex valued, function $g$.
Let us assume that $z'(\tau^*)=0$; then, by uniqueness of the Cauchy problem associated to the previous equation, $z \equiv 0$, a contradiction. Finally, if $Z$ is not finite, then it must have some accumulation point which can not be a simple zero of $z$.
\end{proof}
\begin{corollary}\label{coro:regularityEL_LC}
The function $t_z\in C^4([0,1])$ is invertible, with inverse $\tau_z\in C([0,T])$ which is $C^4$ outside the finite set
$t_z(Z)$.
\end{corollary}
\begin{proof}
The result follows by the definition of $t_z$ and by the elementary inverse function theorem.
\end{proof}
Motivated by the previous corollary we define, for a suitable $N \ge 1$, the points $0=t_0 < t_1 < \ldots < t_N =1$ in such a way that
\[
(0,T) \setminus t_z(Z) = \bigcup_{i=1}^N (t_{i-1},t_i).
\]
Notice that $\tau_i := \tau_z(t_i)$ is such that $z(\tau_i)=0$ at least for $i=1,\ldots,N-1$ ($\tau_0=0$ and $\tau_N=1$ may or may not be collision instants).
\begin{lemma}\label{lem:ELconmu_LC}
Let $x$ as in Proposition \ref{prop:EL_for_B_LC} and let $\{t_0,\ldots,t_N\}$ as above. Then $x \in C([0,1];\R^2)$ is $C^2$ outside collisions. Moreover, the function
\begin{equation} \label{eq:LCnuova}
t \mapsto |\dot x(t)|^2 |x(t)| = \frac{4}{\Lcal^2}|z'(\tau_z(t))|^2
\end{equation}
is continuous in $[0,1]$ and,
for every $i=1,\ldots,N$
there exists $\mu_i >0$ such that
\begin{equation}\label{eq:beta_LC_i}
\ddot x = -\mu_i \frac{x}{|x|^3} + p(t),\qquad t\in(t_{i-1},t_i).
\end{equation}
\end{lemma}
\begin{proof}
By Corollary \ref{coro:regularityEL_LC}, for any $t \in [0,T]$ and $\tau \in [0,1]$ we have that
$x(t)=z^2(\tau_z(t))$. Then,
\[
\zeta(t) := z(\tau_z(t))
\qquad\implies\qquad
x(t)=\zeta^2(t).
\]
In particular $x$ is continuous in $[0,T]$. Restricting to $(\tau_{i-1},\tau_i)$ and $(t_{i-1},t_i)$ respectively, we compute
\[
z'(\tau_z(t)) = \frac{\Lcal}{2}|x(t)| \zeta^{-1}(t)\dot x(t)
\]
so that equation \eqref{eq:LCnuova} follows on each $(t_{i-1},t_i)$. By Corollary \ref{coro:regularityEL_LC}, $|\dot x|^2|x|$ can be extended to a continuous function in the whole $[0,T]$, still satisfying \eqref{eq:LCnuova}.
Differentiating once more and recalling that
$2\scalar{a}{b} = a\bar b + \bar a b$
we obtain
\[
\begin{split}
z''\circ\tau_z &= \frac{\Lcal^2}{2} |x|^2 \zeta^{-1}
\left( \ddot x -\frac12 x^{-1}
\dot x^2 +
\frac{\langle x,\dot x\rangle}{|x|^2} \dot x
\right) \\
& = \frac{\Lcal^2}{2} |x|^2 \zeta^{-1}
\left( \ddot x -\frac{\bar x \dot x^2}{2|x|^2} +
\frac{x\dot{\bar x} + \bar x \dot x}{2|x|^2} \dot x
\right)
=
\frac{\Lcal^2}{2} |x|^2 \zeta^{-1}
\left( \ddot x +\frac12
\frac{|\dot x|^2}{|x|^2} x
\right).
\end{split}
\]
On the other hand, reasoning as in the proof of Lemma \ref{lem:ELconmu} we obtain
\[
\frac{\Lcal^3}{2T} \left[\frac{\Qcal}{\Lcal^2}-1-\Rcal + \Lcal \int_{0}^{1}|z(\xi)|^2
\left(\int_{\tau}^{\xi}|z|^2\scalar{\dot p \circ t_z}{z^2}\right)\,d\xi\right]z
= \frac{\Lcal^2}{2}\left[ C
+ \int_{t}^{T}\scalar{\dot p}{x} \right]\zeta
\]
where
\begin{equation}
\label{eq:defofC_LC}
C:= \frac{1}{{T}} \int_{0}^{T}\left(\frac12 |\dot x|^2 - \frac{1}{|x|} -\scalar{p}{x} -
t\scalar{\dot p}{x} \right)\,dt.
\end{equation}
Finally
\[
\frac{\Lcal^2}{2}\left[ \scalar{p\circ t_z}{z^2} + \bar z^2(p\circ t_z) \right]z
=
\frac{\Lcal^2}{2}\left[ \scalar{p}{x}\zeta + |x|^2 \zeta^{-1}\,p \right].
\]
Substituting in \eqref{eq:ELeqB_LC} we obtain
\begin{equation}\label{eq:senzabeta2_LC}
\ddot x  =
\left[ C -\frac12 |\dot x|^2+ \scalar{p}{x}+ \int_{t}^{T}\scalar{\dot p}{x}
\right]|x|^{-2} x  +p,\qquad t\in(t_{i-1},t_i),
\end{equation}
which is the same equation obtained in \eqref{eq:senzabeta2}. Arguing as in Lemma \ref{lem:ELconmu} we obtain the existence of a constant $\mu_i \in \R$ such that \eqref{eq:beta_LC_i} holds; furthermore, by an analogue of Lemma \ref{coro:mu} for $\BLC$, $\mu_i\ge 0$ and, if $(\tau_0,\tau_1)=(0,1)$ (with or without collisions at the extrema), then $\mu_1=1$. Hence we are left to prove that $\mu_i>0$ in case, say, $x(t_i)=0$. Assume by contradiction that $\mu_i=0$. Then $\ddot x = p$ in $(t_{i-1},t_i)$ and we deduce that $\dot x$ is continuous up to $t_i^-$.
Equation \eqref{eq:LCnuova} implies that
\[
z(\tau_i)= z'(\tau_i)= 0,
\]
in contradiction with Lemma \ref{le:zeri}.
\end{proof}
\begin{lemma}\label{lem:mu=1}
In the same assumptions of Lemma \ref{lem:ELconmu_LC} we have, for every $i=1,\dots,N$,
\begin{equation}\label{eq:LClevel}
\mu_i =1\qquad\text{and}\qquad |z'(\tau_i)|^2=\frac{\Lcal^2}{2}.
\end{equation}
\end{lemma}
\begin{proof}
To start with, we prove that $\mu_i = \mu_{i+1}$, for every $i$. From \eqref{eq:beta_LC_i} and
\eqref{eq:senzabeta2_LC}, for every $i$,
\[
-\mu_i = C|x| -\frac12 |\dot x|^2|x|+ |x|\scalar{p}{x}+|x| \int_{t}^{T}\scalar{\dot p}{x},
\qquad t\in(t_{i-1},t_i).
\]
Letting $t\to t_i^\pm$ and using \eqref{eq:LCnuova},
\begin{equation}\label{eq:pre_LClevel}
\mu_i = \frac{2}{\Lcal^2}|z'(\tau_i)|^2 = \mu_{i+1}.
\end{equation}
This shows that $\mu_1 = \dots = \mu_N =:\mu$.
Using  \eqref{eq:beta_LC_i},  \eqref{eq:defofC_LC} and \eqref{eq:senzabeta2_LC} we obtain
\begin{equation}\label{eq:muvsC_LC}
\frac12 |\dot x|^2 - \frac{\mu}{|x|} -  \scalar{p}{x}
- \int_{t}^{T}\scalar{\dot p}{x}
= \frac{1}{{T}} \int_{0}^{T}\left(\frac12 |\dot x|^2 - \frac{1}{|x|} -\scalar{p}{x} -
t\scalar{\dot p}{x} \right)\,dt,
\qquad
t\not\in\{t_1,\dots,t_N\}.
\end{equation}
Integrating on $[0,T]$ we obtain $\mu=1$. Then the lemma follows by \eqref{eq:pre_LClevel}.
\end{proof}
\begin{proof}[End of the proof of Proposition \ref{prop:EL_for_B_LC}]
We are left to prove the third point of Definition \ref{def:gen_sol}, that is, continuity of
energy and direction across collisions. Let $t_* \in (0,T)$ be such that $x(t_*) = z(\tau_*) = 0$.
Recall that, by Lemma \ref{lem:EL_for_B_LC}, $z\in C^3([0,1],\C)$. Moreover, by \eqref{eq:LClevel}
and \eqref{eq:EL_LC_veloce}, $|z'(\tau_*)| = \Lcal/ \sqrt{2}$ and $z''(\tau_*) = g(\tau_*) z(\tau_*)=0$.
In particular, for $\tau$ approaching $\tau_*$,
\begin{equation}\label{eq:alpha}
\begin{split}
z(\tau) &= (\tau-\tau_*)\alpha(\tau) \qquad \text{with } |\alpha(\tau_*)|=\frac{\Lcal}{\sqrt{2}}\neq 0\text{ and }\alpha\text{ is continuous},\\
z'(\tau) &= z'(\tau_*) +(\tau-\tau_*)^2\beta(\tau) \qquad \text{with } |z'(\tau_*)|=\frac{\Lcal}{\sqrt{2}}\text{ and }\beta\text{ is continuous}.
\end{split}
\end{equation}

As far as the energy continuity is concerned, we have that, by \eqref{eq:LCnuova},
\[
h(t) := \frac12 |\dot x|^2 - \frac{1}{|x|}= \frac{2|z'|^2-\Lcal^2}{{\Lcal^2}|z|^2}.
\]
Using \eqref{eq:alpha} we obtain
\[
\lim_{t \to t_*^\pm} h(t)
 = \lim_{\tau \to \tau_*^{\pm}}  \frac{4\scalar{\beta(\tau)}{z'(\tau_*)}(\tau-\tau_*)^2 + o((\tau-\tau_*)^2)}{\Lcal^2 |\alpha(\tau)|^2(\tau-\tau_*)^2} = \frac{8}{\Lcal^4}
\scalar{\beta(\tau_*)}{z'(\tau_*)}.
\]	

Analogously,
\[
\lim_{t \to t_*^\pm} \frac{x}{|x|} = \lim_{\tau \to \tau_*^{\pm}}
\frac{z^2}{|z|^2}
= \lim_{\tau \to \tau_*^{\pm}}
\frac{\alpha^2(\tau)(\tau-\tau_*)^2}{|\alpha^2(\tau)|(\tau-\tau_*)^2}
= \frac{2}{\Lcal^2}z'(\tau_*)^2.
\qedhere
\]	
\end{proof}
%
\subsection{Analysis of \texorpdfstring{$\BKS$}{BKS}}\label{sec:ELeqBKS}
Finally, for $z=z_0+z_1i+z_2j+z_3k \in \HH\cong\R^4$, let $\Phi\from\HH\to\II\HH\cong\R^3$,
$\Phi(z) = \Phi_{\text{KS}}(z) = \bar z i z$, see Case 3 in Section \ref{sec:euristic}. In this case, it is natural to choose
$p\in \II\HH$. Now the corresponding
functional is
\[
\BKS : H^1(0,1; \HH) \to \R \cup \{+\infty\}
\]
\begin{equation}\label{eq:BKS}
\BKS(z) := \frac{1}{\Lcal(z)}
\Qcal(z) +
\Lcal(z)\Big[ 1+ \Rcal(z)\Big],
\qquad \Bcal(0):=+\infty,
\end{equation}
where
\begin{equation}\label{eq:tutti gli altriKS}
\begin{split}
\Lcal(z) &:= \frac{T}{\int_0^1|z|^2\,d\tau},\\
\Qcal(z) &:= 2 \int_0^1 | z'|^2\,d\tau,
\\
\Rcal(z) &:= \int_0^1
|z|^2 \scalar{p\circ t_z}{\bar z i z}\,d\tau,\\
t_z(\tau) &:= \Lcal(z) \int_{0}^{\tau}|z(\xi)|^2\,d\xi.
\end{split}
\end{equation}
Recall that, in Section \ref{sec:euristic}, we established the correspondence between $\BKS$ under the validity
of condition \eqref{eq:varietaperKS}. Actually, as we are going to show, this condition can be appreciably weakened.
\begin{proposition}\label{prop:EL_for_B_KS}
Let $p\in C^1([0,T];\II\HH)$, and let $z\in H^1(0,1;\HH)\setminus\{0\}$ satisfy
\begin{equation}\label{eq:varBKSvphi}
\frac{d}{d\eps} \left[\BKS(z+\eps \varphi)\right]_{\eps=0} = 0\qquad \text{for every }\varphi\in\Dcal(0,1;\HH).
\end{equation}
Then $z \in C^3([0,1])$, and the map $\tau \mapsto t_z(\tau)$ is invertible on $[0,1]$, with inverse $t\mapsto \tau_z(t)$.

Furthermore, if $z$ satisfies also
\begin{equation}\label{eq:constrBKS_point}
\scalar{z'(\tau^*)}{iz(\tau^*)}=0,\qquad \text{for some }\tau^*\in[0,1],
\end{equation}
then
\[
x(t)=\bar z(\tau_z(t))i z(\tau_z(t))
\]
is a generalized solution of equation \eqref{eq:perturbed_kepler}.
\end{proposition}
As usual, we start deducing the Euler-Lagrange equations and proving some regularity results.
\begin{lemma}\label{lem:EL_for_B_KS}
Let $z\in H^1(0,1; \HH) \setminus \{0\}$ satisfy \eqref{eq:varBKSvphi}. Then
\begin{equation}\label{eq:ELeqB_KS}
z'' =
\frac{\Lcal}{2T}\left(\Qcal-\Lcal^2(1+\Rcal)\right) z + \delta_z,
\end{equation}
where $\Lcal$, $\Qcal$ and $\Rcal$ are evaluated at $z$ and
\[
\delta_z  := \frac{\Lcal^2}{2}\left[ \scalar{p\circ t_z}{\bar z i z} z -
i z |z|^2(p\circ t_z) + \frac{\Lcal^2}{T} \Delta_z z\right]
\]
where
\[
\Delta_z  := \int_{0}^{1}|z(\xi)|^2
\left(\int_{\tau}^{\xi}|z(s)|^2\scalar{\dot p(t_z(s))}{\bar z(s) i z(s)}\right)\,d\xi.
\]
In particular {$z \in C^3([0,1])$}. Moreover, if $\tau^* \in [0,1]$ is such that $z(\tau^*)=0$ then $z'(\tau^*)\neq 0$ and the set $Z := \{\tau \in [0,1] : z(\tau)=0 \}$ is finite.
Furthermore, the function $t_z\in C^3([0,1])$ is invertible, with inverse $\tau_z\in C([0,T])$ which is $C^3$ outside the finite set $t_z(Z)$.
\end{lemma}
\begin{proof}
The lemma follows reasoning as in Lemmas \ref{lem:EL_for_B}, \ref{lem:EL_for_B_LC} and taking into account that,
since $\scalar{v}{w} = \Re(\bar w v) = \scalar{\bar v}{\bar w}$,
\[
\begin{split}
\frac{d}{d\eps} &\left[ \int_0^1
|z|^2 \scalar{p\circ t_z}{ (\bar z+\eps \bar \varphi) i (z+\eps\varphi)}\,d\tau\right]_{\eps=0}
=
\int_0^1
|z|^2 \scalar{p\circ t_z}{ \bar \varphi i z + \bar z i \varphi} \,d\tau
\\
&=
\int_0^1
|z|^2 \scalar{\bar p\circ t_z}{- \bar z i\varphi } \,d\tau + \int_0^1
|z|^2 \scalar{p\circ t_z}{ \bar z i \varphi} \,d\tau
= -2 \int_0^1
|z|^2 \scalar{i z p\circ t_z}{ \varphi} \,d\tau
\end{split}
\]
where we used the fact that $\bar p = - p$.
We proceed in the proof writing equation \eqref{eq:ELeqB_KS} as a linear second order equation in $\R^4$ (with coefficients depending on $\tau$ via $z$),
and we conclude arguing as in the proofs of Lemma \ref{le:zeri} and Corollary \ref{coro:regularityEL_LC}.
\end{proof}
The equation satisfied by $z$ allows to relate conditions \eqref{eq:constrBKS_point} and \eqref{eq:varietaperKS}.
\begin{lemma}\label{lem:constrBKS}
Let $z\in H^1(0,1; \HH) \setminus \{0\}$ satisfy \eqref{eq:varBKSvphi}. Then
\[
\scalar{z'}{iz}\qquad\text{is constant on }[0,1].
\]
In particular, if $z$ satisfies also \eqref{eq:constrBKS_point}, then
\[
\scalar{z'}{iz}\equiv0\qquad\text{on }[0,1].
\]
\end{lemma}
\begin{proof}
	By Lemma \ref{lem:EL_for_B_KS} $z$ is regular up to the boundary, and it solves equation \eqref{eq:ELeqB_KS}. Such equation
can be written as
\[
z'' = a(\tau)z + b(\tau)izp
\]
for some real-valued functions $a$ and $b$. Then
\[
\frac{d}{dt} \scalar{z'}{iz} = b(\tau)\scalar{izp}{iz} = b(\tau)|z|^2 \scalar{p}{1} = 0. \qedhere
\]
\end{proof}
To proceed, we use Lemma \ref{lem:EL_for_B_KS} in order to write the interval $(0,T)$ as the union of disjoint interval
$(t_{i-1},t_i)$, $i=1,\ldots,N$,
for some $N \ge 1$, in such a way that $z(\tau_i)=0$, $\tau_i := \tau_z(t_i)$, at least for $i=1,\ldots,N-1$.
\begin{lemma}\label{lem:ELconmu_KS}
Let \eqref{eq:varBKSvphi} and \eqref{eq:constrBKS_point} hold true. Let $x$ be defined as in Proposition \ref{prop:EL_for_B_KS} and $\{t_0,\ldots,t_N\}$ as above. Then $x \in C([0,1])$ is $C^2$ outside collisions. Moreover, the function
\begin{equation} \label{eq:LCnuova_KS}
t \mapsto -\dot x^2(t) |x(t)| = |\dot x(t)|^2 |x(t)| = \frac{4}{\Lcal^2}|z'(\tau_z(t))|^2
\end{equation}
is continuous in $[0,1]$ and,
for every $i=1,\ldots,N$
there exists $\mu_i >0$ such that
\begin{equation}\label{eq:beta_LC_i_KS}
\ddot x = -\mu_i \frac{x}{|x|^3} + p(t),\qquad t\in(t_{i-1},t_i).
\end{equation}
\end{lemma}
\begin{proof}
Let us define $\zeta(t) := z(\tau_z(t))$ so that
\[
x(t)=\bar \zeta(t)i \zeta(t),\qquad |x(t)|=|\zeta(t)|^2
\]
are continuous in $[0,T]$. Notice that, by Lemma \ref{lem:constrBKS},
\[
\scalar{\dot \zeta(t)}{i \zeta(t)}= \dot \tau_z(t) \scalar{z'(\tau_z(t))}{i z(\tau_z(t))} = 0
\]
for every $t$. Restricting to $(\tau_{i-1},\tau_i)$ and $(t_{i-1},t_i)$ respectively, we can compute
\[
\dot x(t)=\dot{\bar \zeta}(t) i \zeta(t)
+ \bar \zeta(t)i \dot \zeta(t) = 2\bar \zeta(t)i \dot \zeta(t).
\]
Since $\dot \tau_z(t) = 1/(\Lcal|x(t)|)$, we have
\[
\dot \zeta(t) = -\frac{1}{2|x(t)|}i\zeta(t) \dot x(t)
\qquad \text{and} \qquad
z'(\tau_z(t)) = \Lcal|x(t)|\dot \zeta(t) = -\frac{\Lcal}{2}i \zeta(t)\dot x(t)
\]
so that equation \eqref{eq:LCnuova_KS} follows on each $(t_{i-1},t_i)$. By Lemma \ref{lem:EL_for_B_KS}, $|\dot x|^2|x|$ can be extended to a continuous function in the whole $[0,T]$, still satisfying \eqref{eq:LCnuova_KS}.
Differentiating once more
we obtain
\[
\begin{split}
z''(\tau_z(t)) &= \Lcal|x(t)| \frac{d}{dt} z'(\tau_z(t)) = -\frac{\Lcal^2}{2} \left(
\frac12 \zeta(t)
\dot x^2(t) +
|x(t)|i\zeta(t)\ddot x(t)
\right)\\
 &=  \frac{\Lcal^2}{2} \left(
\frac12
|\dot x(t)|^2 \zeta(t)-
|x(t)|i\zeta(t)\ddot x(t)
\right)
\end{split}
\]
(recall that, since $x\in \II\HH$, then $\dot x^2 = - |\dot x|^2$). On the other hand, reasoning as in the proof of Lemmas \ref{lem:ELconmu}, \ref{lem:ELconmu_LC} we obtain
\[
\frac{\Lcal^3}{2T} \left[\frac{\Qcal}{\Lcal^2}-1-\Rcal + \Lcal \int_{0}^{1}|z(\xi)|^2
\left(\int_{\tau}^{\xi}|z|^2\scalar{\dot p \circ t_z}{\bar z i z}\right)\,d\xi\right]z
= \frac{\Lcal^2}{2}\left[ C
+ \int_{t}^{T}\scalar{\dot p}{x} \right]\zeta
\]
where, as usual,
\begin{equation}
\label{eq:defofC_KS}
C:= \frac{1}{{T}} \int_{0}^{T}\left(\frac12 |\dot x|^2 - \frac{1}{|x|} -\scalar{p}{x} -
t\scalar{\dot p}{x} \right)\,dt.
\end{equation}
Finally
\[
\frac{\Lcal^2}{2}\left[ \scalar{p\circ t_z}{\bar z i z} z -iz |z|^2(p\circ t_z) \right]
=
\frac{\Lcal^2}{2}\left[ \scalar{p}{x}\zeta - i  \zeta |x| \,p \right]\circ t_z.
\]
Substituting in \eqref{eq:ELeqB_KS} we obtain
\[
- |x|i\zeta\ddot x = \left[C - \frac12 |\dot x|^2
 + \scalar{p}{x} + \int_{t}^{T}\scalar{\dot p}{x} \right] \zeta - i  \zeta |x| \,p .
\]
and finally, multiplying on the left by $\bar \zeta i$,
\begin{equation}\label{eq:senzabeta2_KS}
\ddot x  =
\left[ C -\frac12 |\dot x|^2+ \scalar{p}{x}+ \int_{t}^{T}\scalar{\dot p}{x}
\right]|x|^{-2} x  +p,\qquad t\in(t_{i-1},t_i),
\end{equation}
which is the same equation obtained in \eqref{eq:senzabeta2} and \eqref{eq:senzabeta2_LC}. At this point,
we can conclude by reasoning as at the end of the proof of Lemma \ref{lem:ELconmu_LC}.
\end{proof}
\begin{remark}\label{rem:xdotKS}
We stress for future reference that the calculations above yield
\[
\dot x(t) = \frac{2}{|z (\tau_z(t) )|^2\Lcal(z)} \bar z (\tau_z(t) )i z'(\tau_z(t)).
\]
whenever $z$ satisfies \eqref{eq:varBKSvphi}, \eqref{eq:constrBKS_point} and $x(t)\neq 0$.
\end{remark}
\begin{proof}[End of the proof of Proposition \ref{prop:EL_for_B_KS}]
Starting from Lemmas \ref{lem:EL_for_B_KS}, \ref{lem:constrBKS}, \ref{lem:ELconmu_KS}, and in particular
from equations \eqref{eq:LCnuova_KS}, \eqref{eq:defofC_KS}, \eqref{eq:senzabeta2_KS}, we can first show an
analogue of Lemma \ref{lem:mu=1}, and then conclude as in the end  of the proof of Proposition \ref{prop:EL_for_B_LC}.
\end{proof}

\subsection{Critical points may not correspond to generalized solutions}
\label{sec:counterexamples}

In this section we provide examples of solutions of the Euler-Lagrange equations associated to $\Bcal$ and $\BKS$ which do
not correspond to generalized solutions of \eqref{eq:perturbed_kepler}. Of course, we will construct such examples
by violating the additional assumptions of Propositions \ref{prop:EL_for_B} and \ref{prop:EL_for_B_KS}.

\begin{examplem} Assume that $y\in H^1(0,1;\R^d)$ satisfies \eqref{eq:varBvphi} both on $(0,\tau_*)$ and $(\tau_*,1)$, with
$y(\tau_*)=0$. Assume that $y$ is differentiable in $\tau_*$, with $y'(\tau_*) = a\neq0$. Then it can not correspond to a generalized solution of \eqref{eq:perturbed_kepler}.
\end{examplem}
\noindent Indeed, let $x$ be defined on  $(0,t^*)$ and $(t^*,T)$, according to Proposition \ref{prop:EL_for_B}. Since
$y(\tau) =  a(\tau-\tau^*) + o(\tau-\tau^*)$ as $ \tau\to\tau_*$ we obtain
\[
\lim_{t\to t_*^\pm} \frac{x(t)}{|x(t)|} = \lim_{\tau\to \tau_*^\pm} \frac{y(\tau)|y(\tau)|}{|y(\tau)|^2} = \pm \frac{a}{|a|},
\]
so that $x$ does not satisfy Definition \ref{def:gen_sol} at $t_*$.

\begin{examplem}
In dimension $d=1$ the functional $\Bcal$ is of class $C^1$ on $H^1(0,1;\R)$. Nonetheless, its critical points do not
necessarily correspond to generalized solutions of \eqref{eq:perturbed_kepler}.
\end{examplem}

\noindent
We consider the functional $\Bcal$ in dimension $d=1$, with $p\equiv 0$:
\[
\Bcal (y) = \frac{2}{T}\int_0^1y^2 \int_0^1(y')^2 + \frac{T}{\int_0^1y^2}.
\]
For concreteness, we work with Dirichlet boundary conditions, i.e. we consider the restriction of $\Bcal$
to $H^1_0(0,1;\R)\setminus\{0\}$. Assume that $y\in H^1_0(0,1;\R)\setminus\{0\}$ satisfies $\Bcal'(y)[\varphi] = 0$
for every $\varphi \in \Dcal(0,1;\R)$. By direct computations we obtain that $y$ satisfies
\begin{equation}\label{eq:casofacile}
\begin{cases}
\frac{2}{T} A y - \frac{2}{T} B y'' - \frac{T}{B^2} y = 0\\
y(0)=y(1)=0,
\end{cases}
\end{equation}
where
\[
A:= \int_0^1(y')^2,\qquad B:= \int_0^1y^2.
\]
We obtain that problem \eqref{eq:casofacile}  admits nontrivial solutions if and only if
\[
\frac{T^2}{2B^3} - \frac{A}{B} = n^2\pi^2, \qquad n=1,2,\dots
\]
with solutions
\[
y_n(\tau) = k_n \sin ( n\pi\tau), \qquad\text{where } k_n = \pm\left(\frac{T \sqrt2}{\pi n}\right)^{1/3}.
\]
Now, if $n=1$ then $y_1(\tau)>0$ in $(0,1)$ and all the assumptions of Proposition \ref{prop:EL_for_B}
are satisfied on such interval; as a consequence, the corresponding $x$ is a generalized solution to the unperturbed Kepler problem
(indeed, it corresponds to a ejection-collision motion, with collisions in the first and last time instants). On the other hand, in case $n\ge2$ $y_n$ vanishes at some interior point, but it is $C^1$ and its zeroes are simple, therefore it can not correspond to a generalized solution by the previous example.

\begin{examplem}\label{rem:mudiversi}
Assume that $y\in H^1(0,1;\R^d)$ satisfies \eqref{eq:varBvphi} on any collisionless subinterval. Then \eqref{eq:beta}
may hold true with different values of $\mu$, depending on the corresponding interval.
\end{examplem}
%
%
%
%
%
%
%
%
%
%
\noindent
In general, if some internal collision occurs, one can not expect that $\mu=1$, even though
the assumptions of Proposition \ref{prop:EL_for_B} hold in any collision-free subinterval.
More precisely, let us assume that $y\in H^1(0,1;\R^d)$
is such that
\[
(0,1)\setminus\{\tau:y(\tau)=0\} =: \bigcup_{j\in\mathcal{J}} I_j,
\]
where each $I_j$ is an open interval and the index set $\mathcal{J}$ is at most countable. If the
assumptions of Proposition \ref{prop:EL_for_B} hold true in each $I_j$, then there exist
coefficients $\mu_j\ge0$ such that
\[
\ddot x = -\mu_j \frac{x}{|x|^3} + p(t),\qquad t\in I_j.
\]
Then, reasoning as in the proof of Lemma \ref{coro:mu}, we can only show that
\begin{equation}\label{eq:somme_mui}
\sum_{j} \int_{I_j} \frac{\mu_i}{|x|} = \int_{0}^T \frac{1}{|x|}.
\end{equation}

On the other hand, we can construct an example in which the above facts actually occur, with different $\mu_i$. Let us consider the
unperturbed Kepler problem $p\equiv0$, and let us assume that $y\in H^1_0(0,1;\R^d)$ satisfies the assumptions
of Proposition \ref{prop:EL_for_B},
with $(\tau_1,\tau_2)=(0,1)$. The existence of such a $y$, which corresponds to a degenerate Keplerian ellipse,
can be obtained for instance as a corollary of Theorem \ref{thm:direct_method} ahead. For every $h_1,h_2 >0$ such that
\[
h_1^2 + h_2^2 = 2^{1/3},
\]
we define the function $w \in H^1_0(0,1;\R^d)$ as
\[
w(\tau) :=
\begin{cases}
h_1\, y (2\tau)   & 0<\tau\le 1/2\\
h_2\, y (2-2\tau)   & 1/2 \le \tau <1.
\end{cases}
\]
We claim that
\begin{equation}\label{eq:doppiobump}
\frac{d}{d\eps} \left[\Bcal(w+\eps \varphi)\right]_{\eps=0} = 0\qquad \text{for every }\varphi\in\Dcal(0,1/2)\cup\Dcal(1/2,1).
\end{equation}
Then Proposition \ref{prop:EL_for_B} applies on both subintervals, so that we can define as usual the function $x(t) = x_w(t)$.
Finally, by direct calculations, one can show that
\begin{equation}\label{eq:doppiokeplero}
\begin{split}
\ddot x &= -2^{2/3}h_1^2\, \frac{x}{|x|^3}\qquad \text{ in }\left(0,2^{-1/3}h_1^2 T\right),\\
\ddot x &= -2^{2/3}h_2^2\, \frac{x}{|x|^3}\qquad \text{ in }\left(2^{-1/3}h_1^2 T,T\right).
\end{split}
\end{equation}
We postpone the proofs of \eqref{eq:doppiobump} and \eqref{eq:doppiokeplero} in the appendix.
\begin{examplem}\label{examp:BKS_not_per}
Periodic critical points of $\BKS$ may not correspond to generalized solutions of \eqref{eq:perturbed_kepler},
in case (\ref{eq:constrBKS_point}) fails.
\end{examplem}
\noindent
In Section \ref{sec:ELeqBKS} we proved that if $z$ is stationary for $\BKS$ with respect to
compactly supported variations, and furthermore \eqref{eq:constrBKS_point} holds true, i.e.
\[
\scalar{z'(\tau^*)}{iz(\tau^*)}=0,\qquad \text{for some }\tau^*\in[0,1],
\]
then $x(t) = \bar z(\tau_z(t)) i z(\tau_z(t))$ is a solution of the perturbed Kepler problem
(under the identification $\II\HH\cong\R^3$). Of course, if $z$ satisfies either
Dirichlet or Neumann homogeneous boundary conditions on $\{0,1\}$, then the above condition
follows. Then a natural question is whether periodic critical points of $\BKS$ do satisfy such
condition, too. In the following we provide an example showing that in general this is not the
case, and furthermore the corresponding $x$ is not a generalized solution of the perturbed Kepler problem.

For concreteness, let $p\equiv0$ and $T=1$, so that the functional $\BKS$ writes as
\[
\BKS(z) = \frac{1}{\Lcal(z)}\Qcal(z) + \Lcal(z) =
2\int_0^1|z|^2 \int_0^1|z'|^2 + \frac{1}{\int_0^1|z|^2}.
\]
Then, as before, $z$ is a critical point of $\BKS$ in $H^1(\R/\Z;\HH)$
if and only if
\begin{equation}\label{eq:easyBKS}
\begin{cases}
2 A z - 2 B z'' - \frac{1}{B^2} z = 0\\
z(0)=z(1), \quad z'(0)=z'(1),
\end{cases}
\qquad\text{where } A:= \int_0^1|z'|^2,\  B:= \int_0^1|z|^2.
\end{equation}
Now, let
\[
z(\tau) := \Lambda\left(\sin{2\pi \tau} + i \cos{2\pi \tau} + j \sin{2\pi \tau} + k\cos{2\pi \tau}\right)
\]
Then $A=8\pi^2\Lambda^2$, $B=2\Lambda^2$, and $z$ satisfies \eqref{eq:easyBKS} provided
$\Lambda^{-6} = 128 \pi^2$. On the other hand,
\[
\scalar{z'}{iz} = 2\pi\Lambda^2\left(-\cos^2{2\pi \tau} - \sin^2{2\pi \tau} - \cos^2{2\pi \tau}
- \sin^2{2\pi \tau}\right) = -4\pi\Lambda^2.
\]
Hence \eqref{eq:constrBKS_point} does not hold. Actually, if we try to perform the usual change of variable we obtain $\tau=\tau_z(t)=t$ and
\[
x(t) = \bar z(\tau_z(t)) i z(\tau_z(t)) = 2\Lambda^2 k,
\]
which is not a solution of the Kepler problem.

On the contrary, let us consider
\[
w(\tau) := \Lambda\left[\sin{2\pi \tau} + i \cos{2\pi \tau} + j \cos{2\pi \tau} + k\sin{2\pi \tau}\right].
\]
Also $w$ satisfies \eqref{eq:easyBKS}, with the same value of $\Lambda$. Then
$\scalar{w'}{iw}\equiv0$, and actually
\[
\begin{split}
x(t) &= \bar w(\tau_w(t)) i w(\tau_w(t)) = 2\Lambda^2 \left[ j (\cos^2 2\pi t - \sin^2 2\pi t)
+ k (2 \sin{2\pi t} \cos{2\pi t}) \right]\\
&= 2\Lambda^2 \left[ j \cos 4\pi t + k \sin{4\pi t} \right],
\end{split}
\]
which is a Keplerian circular motion in the $jk$-plane:
\[
\ddot x = - 16\pi^2 x = -\frac{x}{8\Lambda^6} = -\frac{x}{|x|^3}.
\]

\section{The Dirichlet problem for \texorpdfstring{$\Bcal$}{B}}\label{sec:Diri}

This section deals with the functional $\Bcal$ introduced in \eqref{eq:B} (corresponding to Case 1 discussed in Section \ref{sec:euristic}). Actually the same arguments can be applied with minor simplifications also to $\BLC$ and $\BKS$.

We will prove the following result.
\begin{theorem}\label{thm:direct_method}
Let $p\in C([0,T];\R^d)$. Then
\[
\min_{H^1_0(0,1;\R^d)} \Bcal
\]
is achieved and any minimizer $y$ is such that $y({\tau} )\neq 0$, for any ${\tau}  \in (0,1)$. In particular if $p \in C^1(0,T)$ then any minimizer $y$ corresponds to a generalized solution $x$ of \eqref{eq:perturbed_kepler}, with collisions in $t=0$ and $t=T$.
\end{theorem}
\begin{corollary}
Assume that $p \in C^1(\R;\R^d)$ is $2T$-periodic and even. Then the minimizer $y$ in the above theorem can be extended as an even, $2$-periodic loop such that the corresponding $x$ is an
even, $2T$-periodic generalized solution of \eqref{eq:perturbed_kepler}, with collision at
$t=kT$, $k\in\Z$.
\end{corollary}
A result similar to the above corollary was obtained by Rabinowitz in \cite{Rab1994}, for a different class of problems. More precisely, he deals with autonomous Hamiltonian systems, although treating
more general singularities. His construction of periodic generalized solutions is based on that of brake-collision orbits, which can be extended to periodic ones. It is worth mentioning that, even though he deals with a weaker notion of generalized solution, by construction he finds solutions which fulfill also Definition \ref{def:gen_sol}.
\begin{lemma}\label{lem:coercive}
For every $y \in H^1_0(0,1;\R^d)$, $y \not\equiv 0$, it holds
\[
\Bcal(y) \geq \frac14 \frac{\| y'\|_2^2}{\Lcal(y)} + \Lcal(y) -4\|p\|_\infty^2{T}^3
\geq \|y'\|_2 -4\|p\|_\infty^2{T}^3.
\]
\end{lemma}
\begin{proof}
To start with, we infer that
\[
\Bcal(y) \geq  \frac{\|y'\|_2^2}{2\Lcal(y)} + \Lcal(y) - \Lcal(y)\|p\|_\infty \int_0^1 |y|^4.
\]
Since
\[
|y({\tau})|^2 = \int_0^{\tau} 2\langle y,y'\rangle \leq
2 \|y\|_2 \|y'\|_2
\]
we deduce the Gagliardo-Nirenberg inequality
\[
\int_0^1 |y|^4 \leq  \|y\|^2_{\infty}\|y\|^2_2 \leq 2 \|y\|^3_2\|y'\|_2.
\]
Recalling the definition of $\Lcal(y)$ we obtain
\[
\Bcal(y) \geq \frac{\|y'\|_2^2}{2\Lcal(y)} + \Lcal(y) - 2\|p\|_\infty {T}^{3/2}\frac{\|y'\|_2}{\sqrt{\Lcal(y)}}.
\]
Recalling the elementary inequalities
\[
2\|p\|_\infty {T}^{3/2}\frac{\|y'\|_2}{\sqrt{\Lcal(y)}} \leq
\frac14 \frac{\|y'\|_2^2}{\Lcal(y)}  + 4\|p\|_\infty^2 {T}^3,
\]
and
\[
\|y'\|_2 \leq \frac14 \frac{\|y'\|_2^2}{\Lcal(y)} + \Lcal(y),
\]
we easily conclude.
\end{proof}
\begin{corollary}\label{cor:bondT}
Let $(y_n)_n \subset H^1_0(0,1;\R^d)$. Then
\[
\Bcal(y_n) \leq M
\quad \implies \quad
\|y_n'\|_2 \leq C_1(M), \; 0 < C_2(M) \leq \Lcal(y_n) \leq C_3(M),
\]
for some constants $C_1(M) = C_3(M) = M+4\|p\|_\infty^2{T}^3$ and $C_2(M) = \pi^2 {T}\left[C_1(M)\right]^{-2}$ (by Poincar\'e inequality).
\end{corollary}
\begin{lemma}\label{lem:min_seq}
Let $(y_n)_n \subset H^1_0(0,1;\R^d)$ be a minimizing sequence for the functional $\Bcal$ on $H^1_0(0,1;\R^d)$. Then for any ${\theta} \in (0,1)$ there exist $\bar\delta >0$ and $N \in \mathbb{N}$ such that
\[
\max_{[{\theta},1]} |y_n| \geq \bar\delta,
\qquad
\forall n \geq N.
\]
An analogous result holds in $[0,{\theta}]$.
\end{lemma}
\begin{proof}
By contradiction let  us assume that, for some ${\theta} \in (0,1)$, there exists a sequence $(\delta_n)_n$ such that
\[
\max_{[{\theta},1]} |y_n| = \delta_n
\quad \text{and} \quad
\delta_n \to 0 \text{ as } n \to +\infty.
\]
For $k=\frac{2{\theta}}{{\theta}+1}<1$ and $\ell_{n}({\tau}) = \frac{2y_n({\theta})}{{\theta}-1}({\tau}-1)$, we
define (see Fig. \ref{fig:var1})
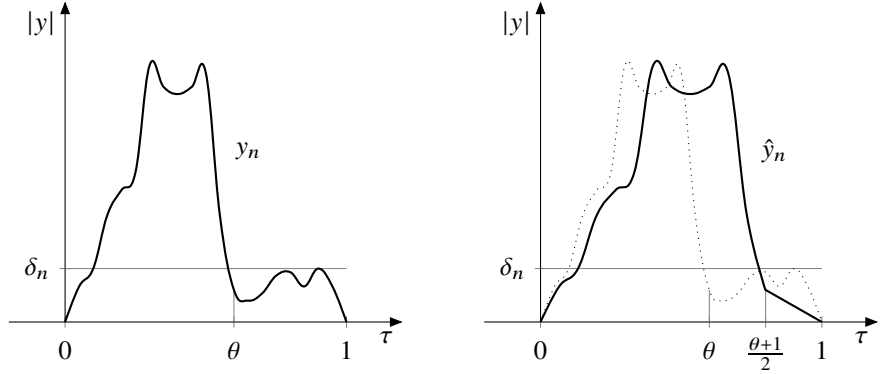
\begin{figure}[t]
\begin{center}
  \begin{tikzpicture}[baseline]
    \begin{axis}[
    every axis x label/.style={at={(current axis.right of origin)},anchor=north west},
      width=.46\linewidth,
      xmin=-0.2,
      xmax=1.2,
      ymin=0,
      ymax=.9,
    axis y line=middle, axis x line=bottom,
    typeset ticklabels with strut,
    axis line style={->},
    xlabel=${\tau}$, ylabel=$|y|$,
    x label style={below left},
    y label style={below left},
    xtick={0,.6,1},
    xticklabels={$0$,${\theta}$,$1$},
    ytick={.15},
    yticklabels={$\delta_n$},
    ]
\addplot[very thin, gray] coordinates{
(0,.15) (1,.15)};
\draw[very thin, gray] (0.6,0) -- (0.6,.09);
\addplot[smooth,tension=.7,thick] coordinates{
(	0   	,	0	)
(	0.05	,	0.1	)
(	0.10	,	0.15	)
(	0.15	,	0.3	)
(	0.20	,	0.37	)
(	0.25	,	0.42	)
(	0.30	,	0.72	)
(	0.35	,	0.66	)
(	0.40	,	0.64	)
(	0.45	,	0.66	)
(	0.50	,	0.7	)
(	0.55	,	0.3	)
(	0.60	,	0.09	)
(	0.65	,	0.06	)
(	0.70	,	0.08	)
(	0.75	,	0.13	)
(	0.80	,	0.14	)
(	0.85	,	0.1	)
(	0.90	,	0.15	)
(	0.95	,	0.1	)
(	1.00	,	0	)};
\draw (.65,.48)  node {$y_n$};
\end{axis}
\end{tikzpicture}
\hspace{.06\linewidth}
  \begin{tikzpicture}[baseline]
    \begin{axis}[
    every axis x label/.style={at={(current axis.right of origin)},anchor=north west},
      width=.46\linewidth,
      xmin=-0.2,
      xmax=1.2,
      ymin=0,
      ymax=.9,
    axis y line=middle, axis x line=bottom,
    typeset ticklabels with strut,
    axis line style={->},
    xlabel=${\tau}$, ylabel=$|y|$,
    x label style={below left},
    y label style={below left},
    xtick={0,.6,.8,1},
    xticklabels={$0$,${\theta}$,$\frac{{\theta}+1}{2}$,$1$},
    ytick={.15},
    yticklabels={$\delta_n$},
    ]
\addplot[very thin, gray] coordinates{
(0,.15) (1,.15)};
\draw[very thin, gray] (0.6,0) -- (0.6,.09);
\draw[very thin, gray] (0.8,0) -- (0.8,.09);
\addplot[smooth, tension=.7, dotted] coordinates{
(	0   	,	0	)
(	0.05	,	0.1	)
(	0.10	,	0.15	)
(	0.15	,	0.3	)
(	0.20	,	0.37	)
(	0.25	,	0.42	)
(	0.30	,	0.72	)
(	0.35	,	0.66	)
(	0.40	,	0.64	)
(	0.45	,	0.66	)
(	0.50	,	0.7	)
(	0.55	,	0.3	)
(	0.60	,	0.09	)
(	0.65	,	0.06	)
(	0.70	,	0.08	)
(	0.75	,	0.13	)
(	0.80	,	0.14	)
(	0.85	,	0.1	)
(	0.90	,	0.15	)
(	0.95	,	0.1	)
(	1.00	,	0	)};
\addplot[smooth,tension=.7,thick] coordinates{
(	0   	,	0	)
(	0.0666	,	0.1	)
(	0.1333	,	0.15	)
(	0.2	,	0.3	)
(	0.2666	,	0.37	)
(	0.3333	,	0.42	)
(	0.4	,	0.72	)
(	0.4666	,	0.66	)
(	0.5333	,	0.64	)
(	0.6	,	0.66	)
(	0.6666	,	0.7	)
(	0.7333	,	0.3	)
(	0.8	,	0.088885	)
};
\addplot[thick] coordinates{
(0.798,.0909) (1,0)};
\draw (.83,.48)  node {$\hat y_n$};
\end{axis}
\end{tikzpicture}
\caption{test function for Lemma \ref{lem:min_seq}, as defined
in equation \eqref{eq:var1}.\label{fig:var1}}
\end{center}
\end{figure}
\begin{equation}\label{eq:var1}
\hat y_n({\tau}) :=
\begin{cases}
y_n(k{\tau}), \; {\tau} \in \left[0,\frac{{\theta} +1}{2}\right],\smallskip \\
\ell_{n}({\tau}), \; {\tau} \in \left(\frac{{\theta} + 1}{2},1\right].
\end{cases}
\end{equation}
We claim that there exists $C>0$, independent of $n$, and a sequence $\gamma_n$,
with $\gamma_n\to0$ as $n\to+\infty$, such that
\begin{equation}\label{eq:contr}
\Bcal(\hat y_n) \leq \Bcal(y_n) -C +\gamma_n, \; \text{as } n \to +\infty.
\end{equation}
In order to do that, let us first estimate the terms involving first derivatives:
\begin{multline*}
\int_0^1 |\hat y_n'|^2 = k^2 \int_0^{\frac{{\theta}+1}{2}} |y_n'(k\tau)|^2 \,d\tau + \frac{4}{({\theta}-1)^2}|y_n({\theta})|^2\frac{1-{\theta}}{2} = \\
= k \int_0^{\theta} |y_n'|^2  + \frac{2}{{\theta}-1}|y_n({\theta})|^2
\leq k \int_0^1 |y_n'|^2+ \frac{2}{1-{\theta}}\delta_n^2,
\end{multline*}
and, with similar computations,
\[
\int_0^1 \frac{\langle \hat y_n,\hat y'_n\rangle^2}{|\hat y_n|^2}
\leq k \int_0^1 \frac{\langle y_n, y'_n\rangle^2}{| y_n|^2}+ \frac{2}{1-{\theta}}\delta_n^2.
\]
These estimates imply that
\begin{equation}\label{eq:stima3}
\Qcal(\hat y_n) = k\Qcal(y_n) + \frac{4}{1-{\theta}}\delta_n^2.
\end{equation}
Next, we obtain asymptotic expansions for the remaining terms.
Since $|\hat y_n|\le\delta_n$ on $\left[\frac{{\theta} + 1}{2},1\right]$ and $|y_n|\le\delta_n$ on $\left[{\theta},1\right]$, a direct computation shows that
\begin{equation*}\label{eq:tre}
\|\hat y_n\|_2^2 = \frac{1}{k} \|y_n\|_2^2+  O(\delta_n^2).
\end{equation*}
This is equivalent to
\begin{equation}\label{eq:stima1}
\frac{1}{\Lcal(\hat y_n)} = \frac1k \frac{1}{\Lcal(y_n)} + O(\delta_n^2).
\end{equation}
We know from Corollary \ref{cor:bondT} that $\Lcal(y_n)$ lies between two positive constants. Hence the previous expansion leads automatically to
\begin{equation}\label{eq:stima2}
{\Lcal(\hat y_n)} = k {\Lcal(y_n)} + O(\delta_n^2).
\end{equation}
The last term to estimate is $\Rcal(\hat y_n)$. First we define
\[
\hat t_n({\tau}) = \Lcal(\hat y_n) \int_0^{\tau} |\hat y_n(\xi)|^2\,d\xi,
\qquad
t_n({\tau}) = \Lcal(y_n) \int_0^{\tau} |y_n(\xi)|^2\,d\xi.
\]
These functions are related by the identity
\begin{equation}\label{eq:due}
\hat t_n({\tau}/k) = \frac{\Lcal(\hat y_n) }{k\Lcal(y_n)}\,t_n({\tau})\qquad\text{if }{\tau}\in\left[
0, \textcolor{red}{\theta}\right] .
\end{equation}
From the definition of $\Rcal$ and Corollary \ref{cor:bondT},
\begin{multline*}
\Rcal(\hat y_n) = \int_0^1 |\hat y_n|^3 \left\langle \hat y_n, p\circ \hat t_n \right\rangle  =\\
= \int_0^{\frac{{\theta}+1}{2}} |y_n(k{\tau})|^3 \left\langle y_n(k{\tau}), p\left(\hat t_n({\tau})\right)\right\rangle \,d\tau + O(\delta_n^4) = \\
= \frac{1}{k} \int_0^{{\theta}} |y_n({\tau})|^3 \left\langle y_n({\tau}), p\left(t_n({\tau})\right)\right\rangle \,d{\tau} + \alpha_n + O(\delta_n^4),
\end{multline*}
where
\[
\alpha_n =  \frac{1}{k} \int_0^{{\theta}} |y_n({\tau})|^3 \left\langle y_n({\tau}), p\left(\hat t_n({\tau}/k)\right)-p\left(t_n({\tau})\right)\right\rangle \,d{\tau} .
\]
In view of \eqref{eq:stima2}, \eqref{eq:due} and the uniform continuity of $p$, we conclude that $\alpha_n\to0$. Note that the rate of convergence will be of order of
$\omega(\delta_n^2)$, where $\omega$ is a modulus of continuity of the function $p$. Using once again that $|y_n|\le \delta_n$ on $[{\theta},1]$, we conclude that
\begin{equation}\label{eq:stima4}
\Rcal(\hat y_n) = \frac{1}{k} \Rcal(y_n) + o(1),
\;\text{ as } n \to +\infty.
\end{equation}
From \eqref{eq:stima1}, \eqref{eq:stima3}, \eqref{eq:stima2} and \eqref{eq:stima4} we
deduce that
\[
\Bcal(\hat y_n) \le \Bcal(y_n) - (1 - k)\Lcal(y_n) + \gamma_n,
\]
with $\gamma_n\to0$. Since $k \in (0,1)$ and $\Lcal(y_n)$ is bounded away from 0, we
deduce the existence of a strictly positive constant $C$ such that \eqref{eq:contr} holds, which contradicts the nature of the
sequence $(y_n)_n$.
\end{proof}
\begin{proposition}\label{prop:min_seq}
Let $(y_n)_n \subset H^1_0(0,1;\R^d)$ be a minimizing sequence for the functional $\Bcal$ on $H^1_0(0,1;\R^d)$, such that $y_n \rightharpoonup y_\infty$ in $H^1_0(0,1;\R^d)$. Then
\[
|y_{\infty}({\tau})| > 0, \qquad \forall {\tau} \in (0,1).
\]
\end{proposition}
\begin{proof}
Recall that any minimizing sequence has a limit point, by Corollary \ref{cor:bondT}. By contradiction, let us assume that the limit function $y_\infty$ vanishes at some ${\tau}_0 \in
(0,1)$. We define the interval $[{\tau}^*,{\tau}^{**}]$ as the connected component of $\{{\tau}:y_\infty({\tau})=0\}$ containing ${\tau}_0$, in such a way that
\begin{equation}\label{eq:null_interval}
y_\infty|_{[{\tau}^*,{\tau}^{**}]}\equiv 0,\qquad
y_\infty|_{[a,b]}\not\equiv 0\text{ for any }[a,b]\supsetneq[{\tau}^*,{\tau}^{**}]
\end{equation}
(notice that it may happen that ${\tau}^*={\tau}^{**}={\tau}_0$). By Lemma \ref{lem:min_seq} and by the uniform convergence of the minimizing sequence, we deduce the existence of $\bar\delta>0$ such that
\[
\max_{[0,{\tau}^*]} |y_\infty({\tau})| \geq \bar\delta
\quad \text{and} \quad
\max_{[{\tau}^{**},1]} |y_\infty({\tau})| \geq \bar\delta,
\]
so that $0<{\tau}^*\leq {\tau}^{**}<1$. Let us now fix $\delta \in (0,\bar \delta)$ which will be specified in the following.
The uniform convergence of $y_n$ to $y_\infty$ guarantees that, for $n$ sufficiently large,
\begin{equation} \label{eq:delta/10}
\max_{[0,{\tau}^*]} |y_n({\tau})| > \delta, \quad
\max_{[{\tau}^{**},1]} |y_n({\tau})| > \delta, \quad
|y_n({\tau}^*)| < \frac{\delta}{10},
\quad \text{and} \quad
|y_n({\tau}^{**})| < \frac{\delta}{10}.
\end{equation}
Consequently, the following sequences are well defined
\[
\begin{split}
a_n=a_n(\delta) &:= \max \{{\tau} < {\tau}^* : |y_n({\tau})|= \delta\},
\\
b_n=b_n(\delta) &:= \min \{{\tau} > {\tau}^{**} : |y_n({\tau})|= \delta\};
\end{split}
\]
furthermore
\[
|y_n({\tau})| < \delta \text{ on } (a_n,b_n)
\qquad \text{and} \qquad
|y_n(a_n)| = |y_n(b_n)| = \delta.
\]
Note that, up to subsequences, we have
\[
a_n = a_n(\delta) \to a_\infty(\delta)
\quad \text{and} \quad
b_n = b_n(\delta) \to b_\infty(\delta),
\]
where, by uniform convergence, $|y_\infty(a_\infty)|=|y_\infty(b_\infty)|=\delta$, $|y_\infty|\leq \delta$ on
$(a_\infty,b_\infty)$.
Then \eqref{eq:null_interval} implies
\[
a_\infty(\delta) \to {\tau}^*
\quad \text{and} \quad
b_\infty(\delta) \to {\tau}^{**},
\qquad \text{as } \delta \to 0.
\]
Recalling that $(y_n)_n$ converges uniformly to $y_\infty$ we conclude that for any $\eps>0$
\begin{equation} \label{eq:small_intervals}
{\tau}^*-a_n <\eps
\quad \text{and} \quad
b_n - {\tau}^{**} < \eps,
\end{equation}
for $\delta$ sufficiently small and $n > N(\delta)$.

In order to obtain a contradiction we consider, for $n$ large, the sequence (see
Fig. \ref{fig:var2})
\begin{equation}\label{eq:var2}
\hat y_n({\tau}) :=
\begin{cases}
\delta \hat U_n({\tau}), \quad \text{on } (a_n,b_n), \\
y_n({\tau}), \quad \text{on } [0,1] \setminus (a_n,b_n),
\end{cases}
\end{equation}
where the path $\hat U_n({\tau})$, ${\tau} \in (a_n,b_n)$, traces the (shortest) arc of geodesic on the unitary sphere in $\R^d$ joining $y_n(a_n)/\delta$ with
$y_n(b_n)/\delta$ and $|\hat U'_n({\tau})|$ is constant.

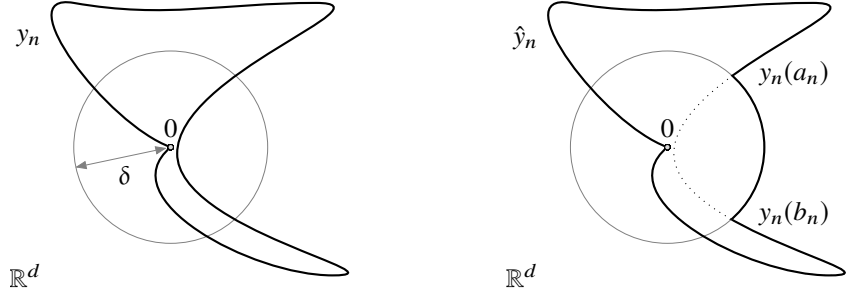
\begin{figure}[t]
\begin{center}
  \begin{tikzpicture}[baseline]
    \begin{axis}[
    disabledatascaling,
      width=.46\linewidth,
     axis equal={true},
    xmin=-.75, xmax=1.2,
    ymin=-1, ymax=1,
    axis lines=none,
    ]
\draw[very thin, gray] (0,0) circle[radius=.6];
\addplot[mark options={fill=gray!30!white}, only marks, mark size=1]
(0,0);
\draw[thick] (0,0) to[out=160, in=170] (-.6,.9) to[out=-10, in=180] (.9,.9) to[out=0, in=90] (.04,-.04)  to[out=270, in=0] (1.,-.8)  to[out=180, in=220] (0,0);
\draw[very thin, gray, <->] (-0.587,-.125) -- node[sloped,pos=.5,below,black] {$\delta$} (-0.02935,-.00625);
\draw (-.88,.68)  node {$y_n$};
\draw (-.9,-.8)  node {$\R^d$};
\draw (0,0)  node[above] {$0$};
\end{axis}
\end{tikzpicture}
\hspace{.08\linewidth}
  \begin{tikzpicture}[baseline]
    \begin{axis}[
    disabledatascaling,
      width=.46\linewidth,
     axis equal={true},
    xmin=-.75, xmax=1.2,
    ymin=-1, ymax=1,
    axis lines=none,
    ]
\draw[very thin, gray] (0,0) circle[radius=.6];
\addplot[mark options={fill=gray!30!white}, only marks, mark size=1]
(0,0);
\draw[dotted] (0,0) to[out=160, in=170] (-.6,.9) to[out=-10, in=180] (.9,.9) to[out=0, in=90] (.04,-.04)  to[out=270, in=0] (1.,-.8)  to[out=180, in=220] (0,0);
\begin{scope}
\path[clip]  (0.381,-.453) -- (0.3,-0.52) -- (0.02,-.35) -- (0.01,.4) -- (0.311,.587) --
(.506,.298) --  (1.1,0.5) -- (1.1,1) -- (-.8,1)
-- (-.8,-.9) -- (1.2,-.9) -- (1.2,-.3) -- cycle;
\draw[thick] (0,0) to[out=160, in=170] (-.6,.9) to[out=-10, in=180] (.9,.9) to[out=0, in=90] (.04,-.04)  to[out=270, in=0] (1.,-.8)  to[out=180, in=220] (0,0);
\end{scope}
\begin{scope}
\path[clip]
(.3,-.452) -- (.7,-.452) -- (.7,.678) -- (.3,.378) -- cycle;
\draw[thick] (0,0) circle[radius=.6];
\end{scope}
\draw (-.88,.68)  node {$\hat y_n$};
\draw (.78,.46)  node {$y_n(a_n)$};
\draw (.78,-.41)  node {$y_n(b_n)$};
\draw (-.9,-.8)  node {$\R^d$};
\draw (0,0)  node[above] {$0$};
\end{axis}
\end{tikzpicture}
\caption{test function for Proposition \ref{lem:min_seq}, as defined
in equation \eqref{eq:var2}.\label{fig:var2}}
\end{center}
\end{figure}
As in \eqref{eq:contr}, we claim to prove that $\Bcal(\hat y_n)$ strictly lowers $\Bcal(y_n)$ uniformly, at least for $n$ sufficiently large.
We argue similarly to the previous lemma estimating $\Lcal(\hat y_n)$, $\Qcal(\hat y_n)$, $\Rcal(\hat y_n)$ in terms of the same functionals evaluated at $y_n$.
In the present case the variation $\hat y_n$ differs from $y_n$ on the interval $(a_n,b_n)$ hence we have
\[
\|y_n\|_2^2 \leq \|\hat y_n\|_2^2\leq \|y_n\|_2^2 +\delta^2(b_n-a_n),
\]
so that, being $(a_n,b_n) \subset (0,1)$
\begin{equation}\label{eq:stima5}
0\leq \frac{1}{\Lcal(\hat y_n)} - \frac{1}{\Lcal(y_n)} \leq \frac{\delta^2}{{T}}(b_n - a_n)
\leq \frac{\delta^2}{{T}}
\end{equation}
and
\begin{equation}\label{eq:stima6}
\Lcal(y_n)(1-\beta_n) \leq \Lcal(\hat y_n) \leq \Lcal(y_n), \quad \beta_n = \delta^2\frac{(b_n-a_n)}{\|y_n\|_2^2+\delta^2(b_n-a_n)}.
\end{equation}
Furthermore, using the continuity of $p$ and equations \eqref{eq:stima5} and \eqref{eq:stima6} we obtain
\begin{equation}\label{eq:stima7}
\Rcal(\hat y_n) = \Rcal(y_n) + o(1) \text{ as } n \to +\infty.
\end{equation}
In order to compute the difference between the kinetic terms we introduce the radial and angular variables in $\R^d$; for any $y \in \R^d$ we write
\[
y = rU, \text{ with } r \geq 0, \text{ and } U \in \mathbb{S}^{d-1}.
\]
Given $y \in H^1_0(0,1;\R^d)$, we have that
$\scalar{U({\tau})}{U'({\tau})}=0$ and $|y({\tau})|'=r'({\tau})$,
hence
\[
\Qcal(y) = \frac12 \int_0^1 \left[ 4(r')^2 + r^2|U'|^2 \right].
\]
Since the radial part of $\hat y_n$ is constant and equal to $\delta$ on $(a_n,b_n)$, we
obtain
\begin{multline*}
\Qcal(\hat y_n) - \Qcal(y_n) = \frac12 \int_{a_n}^{b_n} \left[ \delta^2 |\hat U_n'|^2 -4(r'_n)^2 -r_n^2|U'_n|^2 \right] \leq \\
\leq \frac12 \delta^2 \frac{|U_n(b_n)-U_n(a_n)|^2}{b_n-a_n} -2\int_{a_n}^{b_n} (r'_n)^2
\leq 2\delta^2 \underbrace{\frac{1}{b_n-a_n}}_{(A)}-2 \underbrace{\int_{a_n}^{b_n}
(r'_n)^2}_{(B)}.
\end{multline*}
Now, on the one hand,
\[
4(A) = \min_{{\tau} \in (a_n,b_n)} \left(\frac{1}{{\tau}-a_n}+\frac{1}{b_n-{\tau}}\right)
\leq  \frac{1}{{\tau}^*-a_n}+\frac{1}{b_n-{\tau}^{**}};
\]
on the other hand, using Cauchy-Schwarz inequality and Eq. \eqref{eq:delta/10},
\[
\begin{split}
(B) &\geq \int_{a_n}^{{\tau}^*} (r'_n)^2  + \int_{{\tau}^{**}}^{b_n} (r'_n)^2  \geq
\frac{|r_n({\tau}^*)-r_n(a_n)|^2}{{\tau}^*-a_n}
+
\frac{|r_n(b_n)-r_n({\tau}^{**})|^2}{b_n-{\tau}^{**}}\\
&\geq
\left(\frac{9}{10}\right)^2 \delta^2
\left(\frac{1}{{\tau}^*-a_n}+\frac{1}{b_n-{\tau}^{**}}\right).
\end{split}
\]
We can then deduce
\begin{equation}\label{eq:stima8}
\Qcal(\hat y_n) - \Qcal(y_n) \leq \delta^2 \left[ \frac{1}{2} - 2\left(\frac{9}{10}\right)^2 \right]
\left(\frac{1}{{\tau}^*-a_n}+\frac{1}{b_n-{\tau}^{**}}\right)  < -\delta^2\left(\frac{1}{{\tau}^*-a_n}+\frac{1}{b_n-{\tau}^{**}}\right).
\end{equation}
Using Corollary \ref{cor:bondT} and Eqs. \eqref{eq:stima5}-\eqref{eq:stima6}-\eqref{eq:stima7}-\eqref{eq:stima8},
we can estimate the difference
\begin{multline*}
\Bcal(\hat y_n) - \Bcal(y_n) = \left[\Qcal(\hat y_n) - \Lcal(y_n)\Lcal(\hat y_n)\Rcal(\hat y_n)\right]\left(\frac{1}{\Lcal(\hat y_n)}-\frac{1}{\Lcal(y_n)}\right) + \\
\frac{1}{\Lcal(y_n)}\left[ \Qcal(\hat y_n) - \Qcal(y_n) \right] +
\Lcal(y_n) \left[\Rcal(\hat y_n) - \Rcal(y_n)\right] \leq  \\ \leq \delta^2
\left[ C_1 -C_2\left(\frac{1}{{\tau}^*-a_n}+\frac{1}{b_n-{\tau}^{**}}\right) \right],
\end{multline*}
where $C_1$ and $C_2$ are positive constants not depending on $\delta$ and $n$.
By virtue of Eq. \eqref{eq:small_intervals}, choosing $\delta>0$ sufficiently small, we contradict the minimality of the sequence $(y_n)_n$.
\end{proof}
In order to conclude the proof of Theorem \ref{thm:direct_method} we need the following quite general result.
\begin{lemma}\label{lem:wc_moduli}
Let $y_\infty \in H^1_0(0,1;\R^d)$ and $(y_n)_n \subset H^1_0(0,1;\R^d)$ be such that
\begin{enumerate}
\item[(i)] $y_n \rightharpoonup y_\infty$ in $H^1_0(0,1;\R^d)$;
\item[(ii)] $|Z|=0$, where $Z:=\left\{ {\tau} \in (0,1) : y_\infty({\tau})=0 \right\}$.
\end{enumerate}
Then $|y_n| \rightharpoonup |y_\infty|$ in $H^1_0(0,1;\R)$.
\end{lemma}
\begin{proof}
We have already observed that $|y_\infty|,|y_n|
\in H^1_0(0,1;\R)$, for any $n$; we denote $|y_\infty|',|y_n|'$ their weak derivatives, as in
equation \eqref{eq:weak_der}. Our claim is to prove that for any $\psi \in H^1_0(0,1)$
\[
\left\langle |y_n|,\psi \right\rangle_{H^1_0(0,1)} \to \left\langle |y_\infty|,\psi \right\rangle_{H^1_0(0,1)}, \quad \text{as } n \to +\infty,
\]
or, equivalently, that given any $\varphi \in L^2(0,1)$ and  $\varepsilon >0$,
\[
\limsup_{n \to +\infty}\left|\int_0^1 |y_n|' \varphi - \int_0^1 |y_\infty|' \varphi\right| < \eps.
\]
Since the measure of the set $Z$ is zero, we can find a compact set $K_\eps \subset (0,1)$ such that
\[
K_\eps \cap Z = \emptyset
\quad \text{ and } \quad
|(0,1)\setminus K_\eps| < \eps.
\]
The uniform convergence of $(y_n)_n$ to $y_\infty$ implies that there exists $N=N(\eps)$ such that for any $n \geq N$
\[
y_n({\tau}) \neq 0,\quad
\forall {\tau} \in K_\epsilon,
\]
and hence
\[
\frac{y_n({\tau})}{|y_n({\tau})|} \to \frac{y_\infty({\tau})}{|y_\infty({\tau})|}, \quad
\forall {\tau} \in K_\epsilon.
\]
The pointwise convergence of the sequence $\frac{y_n}{|y_n|}$ to $\frac{y_\infty}{|y_\infty|}$ and Egorov's Theorem guarantee the existence of a measurable set $S_\eps \subset K_\eps$ such that
\[
|K_\eps \setminus S_\eps| < \eps
\]
and
\begin{equation}\label{eq:uniform_conv}
\frac{y_n}{|y_n|} \to \frac{y_\infty}{|y_\infty|},
\quad \text{uniformly on $S_\eps$}.
\end{equation}

Since weakly convergent sequences are bounded, there exists $C>0$
such that $\|y_n\|\le C$ for each $n$. We deduce that, given $\varphi \in L^2(0,1)$,
\[
\left| \int_0^1 |y_n|' \varphi - \int_{S_\eps} |y_n|' \varphi \right| \leq  \int_{(0,1)\setminus S_\eps} \left||y_n|'\right|\left| \varphi \right|
\leq C\|\varphi\|_{L^2\left((0,1)\setminus S_\eps \right)},
\]
and the continuity of the integral of a measurable function with respect to the measure of
the domain implies that
\[
\|\varphi\|_{L^2\left((0,1)\setminus S_\eps \right)} \leq C(\eps),
\]
where $C(\eps)$ (which actually depends on $\varphi$) vanishes as $\eps \to 0$.

Since $|y_\infty|'\varphi$ is integrable on $(0,1)$, there exists $c(\eps) \to 0$ as $\eps \to 0$ such that

\[
\left| \int_0^1 |y_\infty|' \varphi - \int_{S_\eps} |y_\infty|' \varphi \right| \leq c(\eps).
\]
We deduce
\[
\left|\int_0^1 |y_n|' \varphi - \int_0^1 |y_\infty|' \varphi\right|
\leq
CC(\eps) + c(\eps) + \left|\int_{S_\eps} (|y_n|'-|y_\infty|') \varphi \right|.
\]
In order to estimate the last term we compute the weak derivative of the absolute value as
\[
\int_{S_\eps} (|y_n|'-|y_\infty|') \varphi =
\underbrace{
\int_{S_\eps}  \left \langle  \frac{y_n}{|y_n|}-\frac{y_\infty}{|y_\infty|},y_n' \right\rangle \varphi}_{I_1}
+\underbrace{
\int_{S_\eps}  \left \langle  \frac{y_\infty}{|y_\infty|},y_n'-y'_\infty \right\rangle \varphi
}_{I_2}.
\]
By \eqref{eq:uniform_conv},
\[
|I_1| \leq \left\| \frac{y_n}{|y_n|}-\frac{y_\infty}{|y_\infty|} \right\|_{L^{\infty}(S_\eps)}C \|\varphi\|_{L^2} \to 0, \quad \text{as } n \to +\infty.
\]
Furthermore, denoting with $\chi_{S_\eps}$ the characteristic function of $S_\eps$, we write
\[
I_2 = \int_0^1 \langle  y_n'-y'_\infty, \psi \rangle, \quad \text{where }
\psi = \chi_{S_\eps}\varphi \frac{y_\infty}{|y_\infty|}
\]
and this quantity tends to 0 by weak convergence of $(y_n)_n$ to $y_\infty$.
\end{proof}
\begin{proof}[Proof of Theorem \ref{thm:direct_method}]
Let $(y_n)_n \subset H^1_0(0,1;\R^d)$ be a minimizing sequence for $\Bcal$ such that $y_n$ tends to $y_{\infty}$ weakly in $H^1_0(0,1;\R^d)$.
By uniform convergence, we deduce that $\Lcal(y_n) \to \Lcal(y_\infty)$ and, by dominated convergence, $\Rcal(y_n) \to \Rcal(y_\infty)$.
Furthermore, since $\Bcal(y_n)\leq M$ for some $M$, by Corollary \ref{cor:bondT}, $\Lcal(y_\infty) \neq 0$, and we also obtain, by the w.l.s.c. of the norm
\[
\frac{1}{\Lcal(y_\infty)} \|y'_\infty\|^2_2  \leq
\liminf_{n \to +\infty} \frac{1}{\Lcal(y_n)}
\|y'_n\|^2_2.
\]
To conclude we need to show that
\[
\frac{1}{\Lcal(y_\infty)} \int_0^1\left(|y_\infty|'\right)^2 \leq
\liminf_{n \to +\infty} \frac{1}{\Lcal(y_n)}
\int_0^1 \left(|y_n|'\right)^2.
\]
or, equivalently, that $\| |y_\infty|' \|_2
\leq \liminf \| |y_n|' \|_2$; we conclude applying Lemma \ref{lem:wc_moduli}, Proposition \ref{prop:min_seq}, Lemma \ref{lem:wc_moduli}, Proposition \ref{prop:EL_for_B} (and recalling 
the discussion after Definition \ref{def:gen_sol}).
\end{proof}

%

%
\section{Periodic generalized solutions in dimension 2
}\label{sec:periodicBLC}
%

In this section we assume that $p$ is $C^1$ on $\R$ and $T$-periodic, and we look for periodic generalized solutions to \eqref{eq:perturbed_kepler} in dimension $d=2$. To do this, one would be tempted to look for critical points of the functional $\BLC\from W_1 \to \R\cup\{+\infty\}$, where $W_1$ is the space of $1$-periodic loops:
\[
W_1 := \left\{ z\in  H^1([0,1];\C) : z(1) = z(0)\right\}.
\]
As mentioned in the introduction, the main obstruction in this direction is that the Palais-Smale condition is not satisfied
in this setting.

We recall that a sequence $(z_n)_n$ is a (PS) sequence at level $\sigma$ for $\BLC$ if
\[
\BLC(z_n)=\sigma + o(1),\qquad
\|\BLC'(z_n)\|=o(1) \text{ as } n\to\infty,
\]
and that $\BLC$ satisfies the (PS) condition at level $\sigma$ if any such a sequence
admits a strongly convergent subsequence.

To show that $\BLC\from W_1 \to \R\cup\{+\infty\}$ does not satisfy the (PS) condition, we take $p\equiv 0$ and observe that any sequence $(z_n)_n$ of constant
functions with $|z_n|\to+\infty$ satisfies the  (PS) condition. Indeed,
from the definition of $\BLC$ and Remark \ref{rem:BLC'},
\[
\BLC(z_n) = \Lcal(z_n),\qquad \BLC'(z_n)[v] = -\frac{2}{T} \Lcal(z_n)^2
\int_0^1\scalar{z_n}{v}.
\]
In particular, $\|\BLC'(z_n)\|\le 2T/|z_n|^3$.

To recover the Palais-Smale property, we will search
for critical points of $\BLC$ in the space of anti-periodic functions
\[
\splc := \left\{ z\in  H^1([0,1];\C) : z(1) = - z(0)\right\}.
\]
Notice that if $z \in\splc$ then $|z|^2$ can be extended as a $1$-periodic function;
as a consequence, the function $t_z$ defined in \eqref{eq:tutti gli altriLC} is such that
\[
t_z(\tau+k)= t_z(\tau)+kT, \qquad \forall k \in \Z,
\]
and finally the function $\tau \mapsto p \circ t_z$ is 1-periodic and $C^1(\R)$.

We will show the following result.
\begin{theorem}\label{thm:periodicLC}
Let $p\in C^1(\R/(T\Z);\R^2)$. Then there exist infinitely many distinct critical points of $\BLC$ in the
space $\splc$ of antiperiodic orbits, corresponding to infinitely many $T$-periodic generalized solutions of the perturbed Kepler problem \eqref{eq:perturbed_kepler} in dimension $d=2$.
\end{theorem}
Notice that the functional $\BLC$ is even. In order to prove the theorem, we will show that it satisfies the Palais-Smale (PS) condition at every level, and that it is bounded below. This will allow to exploit the theory of Krasnoselskii's genus, which we briefly recall here below (we follow \cite[Ch. 10]{AM}).

Let $\Acal:=\{A\in \splc\setminus\{0\}: A=-A, \ A\text{ is closed}\}$. The \emph{genus} of $A$ is defined as
\[
\gamma(A) := \inf\{n:\exists \phi \in C(A;\R^n\setminus\{0\}),\ \phi\text{ odd}\};
\]
if such a $\phi$ does not exist we define $\gamma(A) = +\infty$, while $\gamma(\emptyset) = 0$.
Moreover, let
\[
\Acal_m :=\{A \subset \Acal: A\text{ is compact and }\gamma(A)\ge m\},\qquad
\sigma_m := \inf_{A\in\Acal_m} \sup_{A} \BLC.
\]
We are going to exploit the following well-known result.
\begin{proposition}[{\cite[Prop. 10.8]{AM}}]\label{prop:AM}
Each finite $\sigma_m$ is a critical level for $\BLC$ provided the (PS) condition holds at level $\sigma_m$. Moreover, if $\sigma_m = \sigma_{m+1}$ for some $m$, then there exist infinitely many critical points at level $\sigma_m$.
\end{proposition}
The proof of this result relies on the fact that if a functional satisfies the Palais-Smale condition at some non-critical level $\sigma$, then it is possible to continuously deforme a $(\sigma+\eps)$-sublevel into a $(\sigma-\eps)$-one (and this deformation can be done preserving symmetry). In the previous context this would contradict the minimax definition of the levels
$\sigma_m$, since the genus of a set is not decreasing under continuous deformations.

Actually, even though the previous result was originally stated for functionals $J$ which are $C^1$  in the whole space and such that $0$ is not a critical point of $J$ at level $\sigma_m$, it readily applies to $\BLC$, which is $C^1$ only outside the origin but has the strong property of continuity at $z=0$,
\[
\BLC(z)\to+\infty\qquad\text{as }\|z\|_2\to0.
\]
Therefore, for each $a<b<+\infty$, the sets $\left\{z \in \splc : a \le \BLC(z) \le b\right\}$
are closed in $\splc$ and the deformation argument applies with no restriction.

In order to apply Proposition \ref{prop:AM} we need some preliminary lemmas.
\begin{lemma}\label{lem:GN_antiper}
For every $z\in\splc$,
\[
\|z\|_{\infty}^2 \le 2 \|z\|_{2}\|z'\|_{2}.
\]
\end{lemma}
\begin{proof}
To start with we notice that, for every $\tau_1\le\tau_2$,
\[
\Big| |z(\tau_2)|z(\tau_2) - |z(\tau_1)|z(\tau_1) \Big| \le
 \int_{\tau_1}^{\tau_2}\left| \frac{d}{d\tau}(|z|z) \right|\,d\tau \le
 \int_{\tau_1}^{\tau_2}2|z||z'|\,d\tau \le 2\|z\|_2\|z'\|_2
\]
(recall \eqref{eq:weak_der}).
Since $|z(1)|z(1) + |z(0)|z(0) = 0$ we obtain, for every $\tau\in[0,1]$,
\[
\begin{split}
2|z(\tau)|^2 \le \Big| |z(\tau)|z(\tau) - |z(0)|z (0)\Big| + \Big| |z(1)|z(1) - |z(\tau)|z(\tau) \Big| \le 4\|z\|_2\|z'\|_2,
\end{split}
\]
and the lemma follows.
\end{proof}
\begin{lemma}\label{lem:BLConLc1}
Let $z\in\splc\setminus\{0\}$, $\alpha:=\|z\|_2\|z'\|_2$, $\beta:= \|z\|^2_2$. Then
\[
\frac{1}{T} \BLC(z)\ge  \frac{2}{T^2}\alpha^2  - 2\|p\|_\infty \alpha + \frac{1}{\beta}.
\]
In particular,
\[
\inf_{\splc} \BLC \ge -\frac{\|p\|_\infty^2T^2}{2}.
\]
\end{lemma}
\begin{proof}
By Lemma \ref{lem:GN_antiper} we have that $\|z\|_4^4 \leq \|z\|_\infty^2\|z\|_2^2 \leq
2\| z\|_{2}^3\| z'\|_{2}$.
Then
	\[
	|\Rcal(z)| = \left|\int_0^1 \scalar{p\circ t_z}{z^2|z|^2}\right|
	\leq \int_0^1 |p\circ t_z||z|^4
	\leq 2\|p\|_\infty \| z\|^3_{2}\| z'\|_{2},
	\]
and the lemma follows, since
\[	
\frac1T \BLC(z) = \frac{2}{T^2}\|z\|^2_2 \|z'\|^2_2 + \frac{1}{\|z\|^2_2}\left(1 + \Rcal(z)\right) .
\qedhere
\]
\end{proof}
\begin{lemma}\label{lem:boundedPS}
For every $a\ge\inf_{\splc} \BLC$ there exist positive constants $C_i=C_i(a)$, $i=1,2,3$, such that
\[
\BLC(z)<a
\qquad\implies\qquad
\begin{cases}
\|z'\|_2 \le C_1(a)
\\
C_2(a) \le \|z\|_2\le C_3(a).
\end{cases}
\]
\end{lemma}
\begin{proof}
With the notations of the previous lemma, we have
\[
\frac{2}{T^2}\alpha^2  - 2\|p\|_\infty \alpha \le a
\qquad\text{and}\qquad
\frac{1}{\beta} \le a + \frac{\|p\|_\infty^2T^2}{2},
\]
so that $\alpha$ is bounded above and $\beta = \|z\|_2^2$ is bounded away from $0$, and the existence of $C_2$ follows. Then
\[
\|z'\|_2^2 = \frac{\alpha^2}{\beta}
\]
is bounded above, and also the existence of $C_1$ follows. Finally, since
$\|z\|_2^2 \le \|z\|_\infty^2$, Lemma \ref{lem:GN_antiper} implies the PoincarÃ© inequality
$\|z\|_2 \le 2\|z'\|_2$, and also the existence of $C_3$ follows.
\end{proof}
The previous results allow to prove the Palais-Smale property for $\BLC$.
\begin{lemma}\label{lem:PS}
The functional $\BLC$ satisfies the Palais-Smale condition in $\splc$ at any level $a\ge\inf_{\splc} \BLC$.
\end{lemma}
\begin{proof}
Let $(z_n)_n\subset\splc$ be a (PS) sequence for $\BLC$ at level $a$. By Lemma \ref{lem:boundedPS} we have that $\|z_n'\|_2\le C_1$, $0< C_2 \le \|z_n\|_2 \le C_3$,
where the constants are independent of $n$. Then, up to a subsequence,
$z_n \rightharpoonup z\not\equiv 0$ weakly in $\splc$ and uniformly. Since $(z_n-z)_n$ is
bounded, using Remark \ref{rem:BLC'} we have that
\[
\BLC'(z_n)[z_n-z] = \frac{4}{\Lcal_n} \int_{0}^{1} \left[\scalar{z_n'}{z_n'-z'} + \scalar{\frac{\Lcal_n}{2T}\left(\Qcal_n-\Lcal_n^2(1+\Rcal_n)\right) z_n + \delta_{z_n}}{z_n-z}\right] = o(1)
\]
(recall that $\delta_z$ does not depend on $z'$).
Since $z_n-z\to0$ strongly in $L^2$, and all the terms are bounded, we have
\[
\int_{0}^{1} \scalar{\frac{\Lcal_n}{2T}\left(\Qcal_n-\Lcal_n^2(1+\Rcal_n)\right) z_n + \delta_{z_n}}{z_n-z} = o(1),
\]
and thus
\[
o(1) = \int_{0}^{1} \scalar{z_n'}{z_n'-z'}  = \int_{0}^{1} \scalar{z_n'}{z_n'} - \int_{0}^{1} \scalar{z_n'}{z'} = \int_{0}^{1} \scalar{z_n'}{z_n'} - \int_{0}^{1} \scalar{z'}{z'} + o(1).
\]
Then $\|z_n'\|_2 \to \|z'\|_2$, which, together with the weak convergence, yields the strong one, concluding the proof.
\end{proof}
\begin{proposition}\label{prop:infinitelycpBLCX-1}
The functional $\BLC$ admits infinitely many critical points in $\splc$.
\end{proposition}
\begin{proof}
The proposition follows from Proposition \ref{prop:AM}: indeed, each $\sigma_m>-\infty$ because of Lemma \ref{lem:BLConLc1}; each $\sigma_m<+\infty$ because $\Acal_m$ is not empty (for instance, it contains homeomorphic symmetric images of $\sphere^{m-1}$); (PS) holds at any $\sigma_m$, by Lemma \ref{lem:PS}.
\end{proof}
Of course, critical points of $\BLC$ in $\splc$ are solutions of a boundary value problem for the corresponding Euler-Lagrange equation.
\begin{lemma}\label{lem:boundcond_LC}
Let $z$ be a critical point of $\BLC$ in $\splc$. Then $z \in C^3([0,1])$ satisfies
\[
\begin{cases}
z'' =
\frac{\Lcal}{2T}\left(\Qcal-\Lcal^2(1+\Rcal)\right) z + \delta_z, &\tau\in(0,1),\\
z(1) = -z(0), \qquad z'(1) = -z'(0),
\end{cases}
\]
where $\delta_z$ is defined in Lemma  \ref{lem:EL_for_B_LC}.
\end{lemma}
\begin{proof}
Since $\Dcal(0,1;\C)\subset \splc$, the regularity of $z$ and equation \eqref{eq:ELeqB_LC} follow by Lemma \ref{lem:EL_for_B_LC}. As a consequence, for every $v\in \splc$ we can integrate by parts in Remark \ref{rem:BLC'}, obtaining
\[
\begin{split}
0 &= \BLC'(z)[v]= \frac{4}{\Lcal} \int_{0}^{1} \left[\scalar{z'}{v'} + \scalar{\frac{\Lcal}{2T}\left(\Qcal-\Lcal^2(1+\Rcal)\right) z + \delta_z}{v}\right]\\& = \frac{4}{\Lcal} \left[\scalar{z'(1)}{v(1)} - \scalar{z'(0)}{v(0)}\right] =
 \frac{4}{\Lcal} \scalar{z'(1)+z'(0)}{v(1)},
\end{split}
\]
and the lemma follows choosing $v(\tau) = [z'(1)+z'(0)]\cos(\pi\tau) \in \splc$.
\end{proof}
\begin{proof}[End of the proof of Theorem \ref{thm:periodicLC}]
By Proposition \ref{prop:infinitelycpBLCX-1}, we are left to show that if $z$ is a critical point of $\BLC$ in $\splc$ then
\[
x(t) = z^2(\tau_z(t))
\]
is a generalized $T$-periodic solution of \eqref{eq:perturbed_kepler}.

Again, since $\Dcal(0,1;\C)\subset \splc$,  Proposition \ref{prop:EL_for_B_LC} applies, so that
$x$ is a generalized solution of \eqref{eq:perturbed_kepler}. To prove that it is $T$-periodic we observe that, on the one hand,
\[
x(T) = z^2(\tau_z(T)) = z^2(1) = (-z(0))^2 = z^2(\tau_z(0)) = x(0).
\]
On the other hand, recalling Lemma \ref{lem:ELconmu_LC}, we know that
\[
z(\tau_z(t)) z'(\tau_z(t)) = \frac{\Lcal}{2}|x(t)| \dot x(t)\text{ outside collisions, and }\quad
|\dot x(t)|^2 |x(t)| = \frac{4}{\Lcal^2}|z'(\tau_z(t))|^2 \text{ on }[0,T].
\]
Therefore, if $z(0)\neq0$, we can use the first equality to show $\dot x (T) = \dot x (0)$; in case $z(0)=0$, we can use the second one and argue as in the end of the proof of Proposition \ref{prop:EL_for_B_LC} to show that
\[
\lim_{t \to 0^+} \left[\frac12 |\dot x|^2 - \frac{1}{|x|}\right]
=
\lim_{\tau\to 0^+} \frac{2|z'|^2-\Lcal^2}{{\Lcal^2}|z|^2}
=
\lim_{\tau\to 1^-} \frac{2|z'|^2-\Lcal^2}{{\Lcal^2}|z|^2}
=
\lim_{t \to T^-} \left[\frac12 |\dot x|^2 - \frac{1}{|x|}\right]
\]
and
\[
\lim_{t \to 0^+} \frac{x}{|x|}
=
\lim_{\tau \to 0^+}\frac{z^2}{|z|^2}
=
\frac{2}{\Lcal^2}z'(0)^2
=
\frac{2}{\Lcal^2}z'(1)^2
=
\lim_{\tau \to 1^-}\frac{z^2}{|z|^2}
=
\lim_{t \to T^-} \frac{x}{|x|}
.
\]
To conclude, we show that if $z_1$, $z_2$ are critical points of $\BLC$ in $\splc$, with
$z_1\neq\pm z_2$, then $x_1\neq x_2$ with $x_j(t) = z_j^2(\tau_{z_j}(t))$, $j=1,2$. Indeed,
if $x_1=x_2=x$ then $|z_1(\tau)| = |z_2(\tau)|$, implying that $\tau_{z_1} = \tau_{z_2}$.
Morever $z_1^2 = z_2^2$ and therefore $\Lcal_1 = \Lcal_2 = \Lcal$, $\Qcal_1 = \Qcal_2$,
$\Rcal_1 = \Rcal_2$, $\frac{1}{z_1} \delta_{z_1} = \frac{1}{z_2} \delta_{z_2}$. In view
of Lemma \ref{lem:boundcond_LC}, both $z_1$ and $z_2$ are solutions of the same second order linear differential equation. Let $\tau_*\in [0,1]$ be such that $z_1(\tau_*)\neq0$.  Then
$z_2(\tau_*)=\pm z_1(\tau_*)$ and, writing $\tau_* = \tau_{z_j} (t_*)$
\[
z_1(\tau_*) z'_1(\tau_*) = \frac{\Lcal}{2} |x(t_*)|\dot x(t_*) = z_2(\tau_*) z'_2(\tau_*).
\]
In consequence, $z'_2(\tau_*) = \pm z'_1(\tau_*)$. By the uniqueness of the initial value problem, either $z_1=z_2$ or $z_1 = -z_2$ on $[0,1]$.
\end{proof}

\section{Periodic generalized solutions in dimension 3
}\label{sec:periodicBKS}

In order to obtain periodic solutions $x=x(t)$ to the perturbed Kepler problem in dimension $3$, one would like to adapt the arguments of Section \ref{sec:periodicBLC} to the functional $\BKS$. As a further difficulty, we know from Example \ref{examp:BKS_not_per} that the variational principle for $\BKS\from H^1(\R/\Z;\HH) \to \R \cup\{+\infty\}$ is not consistent with the periodic problem for \eqref{eq:perturbed_kepler}. To overcome this difficulty we will consider the manifold $\Mcal$ composed by all non-trivial $H^1$ functions satisfying a condition of Floquet type, namely
\begin{equation}\label{eq:preM}
z(\tau +1) = \xi z(\tau), \qquad \xi \in \sphere^1,
\end{equation}
where $\sphere^1 := \{e^{i\theta} = \cos\theta + i \sin\theta:\theta\in\R\}$ is the unit circle in the plane $\C\subset\HH$. The 
variational principle for $\BKS\from\Mcal \to \R \cup\{+\infty\}$ will be consistent with the periodic problem. 
\begin{theorem}\label{thm:periodicKS}
Let $p\in C^1(\R/(T\Z);\R^3)$. Then there exist infinitely many $T$-periodic generalized solutions of the perturbed Kepler problem \eqref{eq:perturbed_kepler} in dimension $d=3$.
\end{theorem}
\begin{remark}
The use of different domains, $\splc$ or $\Mcal$, in dimensions $d=2$ and $d=3$, is related to the different topology of the corresponding regularization maps. For $d=2$, the Levi-Civita map $\Phi_{\text{LC}} \from z \in \C\setminus \{0\} \to x = z^2\in \C\setminus \{0\} $ has finite fibers $\Phi_{\text{LC}}^{-1}(x) = \{z,-z\}$. For $d=3$, the  Kustaanheimo-Stiefel map $\Phi_{\text{KS}} \from z \in \HH\setminus \{0\} \to x = \bar z i z \in \II\HH\setminus \{0\} $ has $\sphere^1$-fibers $\Phi_{\text{KS}}^{-1}(x) =  \{\xi z : \xi \in\sphere^1\}$.
\end{remark}
We recall that, when the forcing term $p$ takes values in the plane, the functional $\BKS\from\Mcal\to\R \cup\{+\infty\}$ can be seen as an extension of $\BLC\from W_1 \to\R\cup\{+\infty\}$, see Remark \ref{rem:lcplane}. Recalling the discussion at the beginning of Section \ref{thm:periodicLC}, this implies that the Palais-Smale property will not hold in general for $\BKS$ on $\Mcal$ (notice that, taking $\xi =1$ in \eqref{eq:preM}, we have that $\Mcal$ contains also the $1$-periodic loops in $\HH$). 
Nonetheless, we will show that the weaker Palais-Smale-Cerami property holds at positive levels of $\BKS$. This will provide enough compactness to obtain the existence of infinitely many critical points.

The proof of Theorem \ref{thm:periodicKS} is divided in two parts: first we will specify the suitable manifold
$\Mcal$, encoding \eqref{eq:preM}, and provide differential and Riemannian structures; secondly, we will prove the existence of critical points of $\BKS$ on $\Mcal$.

\subsection{The manifold \texorpdfstring{$\Mcal$}{M}}

In principle, functions $z$ such that $\bar z i z$ is periodic, are defined for all $\tau\in\R$, hence they can be seen as elements of the space $H^1_{\loc}(\R;\HH)$. In order to set the problem in a more convenient functional space, one would be tempted to consider restrictions of such functions to the interval $(0,1)$ (as we did in the $2$-dimensional case). As a matter of fact, as we will show in the following, the right choice is to consider functions defined on the interval $(0,2)$. For easier notation, in the following we denote
\[
\spc := H^1(0,2;\HH).
\]

To start our construction, for every $\xi\in \S^1 \subset \C \subset \HH$ (i.e.
$\xi = e^{i\alpha}$, with $\alpha\in\R$) we consider the vector space
\[
W_\xi=\{z\in \spc:z(\tau+1) = \xi z(\tau),\ \tau\in(0,1)\},
\]
endowed with the inner product induced by $\spc$. In particular, we write $W := W_1$ for the space of $1$-periodic loops in $\spc$, while $W_{-1}$ denotes the space of anti-periodic ones.
It is easy to check that
\begin{equation}\label{eq:int}
W_{\xi_1}\cap W_{\xi_2} = \{0\}\qquad\text{whenever }\xi_1\neq\xi_2.
\end{equation}
Given $\alpha\in\R$ and $z\in \spc$, we define a new function $E_\alpha z\in \spc$ by the formula
\[
E_\alpha z(\tau) = e^{i\alpha\tau}z(\tau).
\]
Then $E_\alpha$ induces a linear isomorphism between $W_{\xi_1}$ and $W_{\xi_2}$, with
$\xi_2 = e^{i\alpha} \xi_1$. The inverse operator is $E_{-\alpha}$. By direct computations,
\[
\|E_\alpha z\|^2 = \|z\|_2^2 + \|z'+i\alpha z\|_2^2 \le \|z\|_2^2 + 2(\|z'\|_2^2+\alpha^2
\| z\|_2^2) \le \mu^2(\alpha)\|z\|^2,
\]
where $\|\cdot\| = \|\cdot\|_{\spc}$, $\|\cdot\|_2 = \|\cdot\|_{L^2(0,2;\HH)}$ and $\mu^2(\alpha) = \max\{2,1+2\alpha^2\}$. Summing up
\begin{equation}\label{eq:est}
\frac{1}{\mu(\alpha)}\|z\| \le \|E_\alpha z\| \le \mu(\alpha) \|z\|
\qquad
\text{ and }\mu(\alpha) = \max\left\{\sqrt2,\sqrt{1+2\alpha^2}\right\}.
\end{equation}
In particular, each $W_\xi$ is isomorphic to $W$.

Define
\begin{equation}\label{eq:defM}
\Mcal := \bigcup_{\xi\in\S^1} (W_\xi\setminus\{0\})=\{E_\alpha w:w\in W\setminus\{0\},\ \alpha\in\R\}.
\end{equation}
This is a disjoint union in view of \eqref{eq:int}.
\begin{lemma}\label{lem:BKS_per}
Let $z\in\Mcal$. Then the functions
\[
\bar z i z,\qquad \bar z i z',\qquad |z|,\qquad|z'|,
\]
can be extended to $1$-periodic functions.
\end{lemma}
\begin{proof}
By assumption, $z\in W_\xi\setminus\{0\}$, for some $\xi=e^{i\alpha}$. Thus, for every
$\tau\in(0,1)$,
\[
\bar z(\tau+1) i z(\tau+1) = \bar z(\tau) e^{-i\alpha} i e^{i\alpha} z(\tau) = \bar z(\tau)  i  z(\tau).
\]
In particular, by continuity, $\bar z(2) i z(2) =\bar z(1) i z(1) =\bar z(0) i z(0) $, and the property of $\bar z i z$ follows. The other properties follow from analogous computations, possibly in a.e. sense when $z'$ is involved.
\end{proof}
\begin{remark}
The previous lemma shows that the correspondence $z\mapsto \bar ziz$ maps $\Mcal\subset X$ into $H^1(0,1;\HH)$. Notice that, in principle, different elements of $\Mcal$ may have the same
restriction to $(0,1)$, as for instance
\[
z_1(\tau) = |\sin(\pi \tau)|\in W,\qquad z_2(\tau) = \sin(\pi \tau)\in W_{-1}.
\]
This explains the choice to work in $X$.
\end{remark}
We are going to show the following results.
\begin{proposition}\label{prop:Mcal}
Under the previous notation:
\begin{enumerate}
\item\label{i:1} $\Mcal$ is a $C^\infty$ submanifold of $\spc$, modelled on $W\times\R$;
\item\label{i:2} the tangent space at $\Mcal\ni z=E_\alpha w$, $w\in W$, is
\[
\begin{split}
T_{E_\alpha w}\Mcal &= \{E_\alpha(\Delta + i\delta I\cdot w) : \Delta \in W,\ \delta\in\R \} \\
&= E_\alpha(W \oplus (i I\cdot w)\R),
\end{split}
\]
where $I(\tau)=\tau$ is the identity on $[0,2]$;
\item\label{i:3} the geodesic distance on $\Mcal$ (induced by the embedding $\Mcal \hookrightarrow X$) satisfies
\[
\dist_{\Mcal} (z_1,z_2) \ge \|z_1-z_2\|.
\]
\end{enumerate}
\end{proposition}
\begin{remark}
It is worth noticing that both the dimension and the codimension of $\Mcal$ are infinite: this follows from
the splitting
\[
X = W \oplus W_{-1} \oplus V_4,
\]
where $V_4$ is the $4$-dimensional space
\[
V_4 = \{qI:q\in\HH\}.
\]
Indeed, given $z$ in $\spc$, we have the decomposition
\[
z = P_W z + P_{W_{-1}} z + z_{**},
\]
where $z_{**} = \frac{z(2)-z(0)}{2}I \in V_4$, $z = z_* + z_{**}$, and
\(
P_{W_{\pm1}} z = \frac12 \left[z_*(\tau) \pm z_{*}(\tau+1) \right].
\)
Notice that $z_*$ can be extended to a $2$-periodic function, so that the previous formulas are well-defined.
\end{remark}
\begin{remark}\label{rem:norm_tangent}
Notice that, as a submanifold, $\Mcal$ inherits the Riemannian structure
of $\spc$. In particular, for every $v\in T_z \Mcal$,
\[
\|v\|_{T_z \Mcal} = \|v\|_{\spc} = \|v\|.
\]
Moreover the geodesic distance on $\Mcal$ is defined as
\[
\dist_\Mcal (z_0,z_1) = \inf \left\{\int_0^1 \left\|\frac{d}{ds}\gamma(s)\right\|\,ds :
\gamma : [0,1]\to\Mcal\text{ smooth, }\gamma(i) = z_i \right\},
\]
where as usual $\|\cdot\|$ denotes the norm in $\spc$.
\end{remark}
The rest of the section is devoted to the proof of the above proposition. To start with we notice that, since $\Mcal$ is a subset of $X$,
it is a metric space with the induced distance.
\begin{lemma}\label{lem:functional}
The functional
\[
\Mcal\ni z \mapsto \xi\in \sphere^1,
\]
where $z\in W_\xi$,
satisfies
\[
|\xi(z_1) - \xi(z_2)| \le \frac{4}{\max_i\|z_i\|_2}
\|z_1-z_2\|.
\]
In particular, $\xi$ is continuous.
\end{lemma}
\begin{proof}
We multiply the identity $z(\tau+1) = \xi z(\tau)$ by $\bar z(\tau)$ to the right and integrate
over $[0,1]$. We obtain
\[
\xi = \frac{2}{\|z\|_2^2} \int_0^1 z(\tau+1)\bar z(\tau)\,d\tau
\]
(recall that, by Lemma \ref{lem:BKS_per}, $\|z\|_2^2 = 2\int_0^1|z|^2 = 2\int_1^2|z|^2$).
The functional $z\mapsto\xi$ is $0$-homogeneous:
\[
\xi(\lambda z) = \xi(z)\text{ if }\lambda>0.
\]
Assuming first that both $\|z_i\|_2^2=2$ we obtain
\[
\begin{split}
|\xi(z_1) - \xi(z_2)| & = \left|\int_0^1\left[z_1(\tau+1)\bar z_1(\tau) - z_2(\tau+1)\bar z_2(\tau)\right]\right|\,d\tau\\
&\le \int_0^1\left|z_1(\tau+1) - z_2(\tau+1)\right|\cdot\left|\bar z_1(\tau)\right|\,d\tau
+\int_0^1\left|z_2(\tau+1)\right|\cdot\left|\bar z_1(\tau) - \bar z_2(\tau)\right|\,d\tau
\\
&\le \|z_1\|_{L^2(0,1)}\|z_1-z_2\|_{L^2(1,2)} + \|z_2\|_{L^2(1,2)}\|z_1-z_2\|_{L^2(0,1)}
\le \sqrt2 \|z_1-z_2\|_{L^2(0,2)}.
\end{split}
\]
In the general case, assume for concreteness $\|z_2\|_2\le\|z_1\|_2$. Then
\[
\begin{split}
|\xi(z_1) - \xi(z_2)| & = \left|\xi\left(\frac{z_1\sqrt2}{\|z_1\|_2}\right) - \xi\left(\frac{z_2\sqrt2}{\|z_2\|_2}\right)\right|
\le 2 \left\|\frac{z_1}{\|z_1\|_2}-\frac{z_2}{\|z_2\|_2}\right\|_{2}\\
&= \frac{2}{\|z_1\|_2\|z_2\|_2} \Big\|\|z_2\|_2 z_1-\|z_2\|_2 z_2+\|z_2\|_2 z_2-\|z_1\|_2 z_2\Big\|_{2}
\le \frac{4}{\|z_1\|_2} \left\| z_1- z_2\right\|_{2}.\qedhere
\end{split}
\]\end{proof}
By now, $\Mcal$ is a metric space. Now we are going to induce on it a structure of smooth submanifold of $\spc$, modeled on the Hilbert space $W\times\R$. To this end, we consider the map
\[
\Phi : W \times \R \to \spc,\qquad (w,\alpha) \mapsto E_\alpha w.
\]
Since $W \times \R$ is an Hilbert space, in the following we identify it with its
tangent space. Morever, we recall that $W$ is a subspace of $\spc$, thus we use in it the norm
$\|\cdot\| = \|\cdot\|_{H^1(0,2;\HH)}$.
\begin{lemma}\label{lem:Phismooth}
$\Phi$ is $C^\infty$. Moreover, for every $(w,\alpha)$, $(\Delta,\delta)\in W \times \R$
\[
\Phi'(w,\alpha)[\Delta,\delta] = E_\alpha \left(\Delta + i\delta I \cdot w \right).
\]
\end{lemma}
\begin{proof}
The parameterized curve $\Ecal:\R \to \spc$, $\alpha\mapsto E_\alpha 1$ can be differentiated and the velocity vector
is $\dot \Ecal (\alpha) = i I \cdot \Ecal(\alpha)$. This is a direct consequence of the identity $\Ecal(\alpha)(\tau)
= e^{i\alpha\tau}$. It is now easy to deduce that $\Ecal$ is $C^\infty$. In fact the successive derivatives are easily
obtained from the formula for the velocity vector. Furthermore, let the map $T$ be defined as
\[
W \times \spc \ni (w,\varphi) \mapsto T(w,\varphi) := \varphi \cdot w \in \spc.
\]
We have that $T$ is bilinear and continuous, thus it is smooth as well. Since $\Phi = T \circ (\mathrm{Id}\times \Ecal)$,
the chain rules implies that $\Phi$ is $C^\infty$. Moreover
\[
\Phi'(w,\alpha)[\Delta,\delta] = E_\alpha \Delta + \delta \dot\Ecal (\alpha) w = E_\alpha \left(\Delta + i\delta I \cdot w \right).
\qedhere
\]
\end{proof}
Next we address the injectivity and surjectivity properties of $\Phi$. Given an open interval $I\subset\R$, we denote with $|I|$ its lenght and with
\[
\Phi_I := \left.\Phi\right|_{(W\setminus\{0\})\times I},\qquad
\Ucal_I := \Phi((W\setminus\{0\})\times I).
\]
\begin{lemma}\label{lem:inj_Phi}
Under the previous notation,
\begin{itemize}
\item $\Phi(w,\alpha+2\pi) = \Phi(E_{2\pi} w , \alpha)$, for every $(w,\alpha)\in W \times \R$;
\item if $|I|> 2\pi$ then $\Phi_I$ is not injective and $\Ucal_I = \Mcal$;
\item if $|I|\le 2\pi$ then $\Phi_I$ is injective and $\Ucal_I \subsetneq \Mcal$.
\end{itemize}
\end{lemma}
\begin{proof}
The first part is trivial, since
\[
\Phi(w,\alpha+2\pi)(\tau) = e^{i(\alpha+2\pi)\tau}w(\tau) = e^{i\alpha\tau}e^{i2\pi\tau}w(\tau) = \Phi(E_{2\pi} w , \alpha).
\]
Concerning the injectivity of $\Phi_I$, we observe that
\[
\Phi(w_1,\alpha_1) = \Phi(w_2,\alpha_2)
\iff
\exists k\in\Z : \alpha_2 = \alpha_1 + 2k\pi\text{ and }
w_2 = E_{-2k\pi}w_1.
\]
Finally, concerning the surjectivity, we have that
\[
z\in W_\xi
\iff
\exists (w,\alpha)\in W\times\R : z=\Phi(w,\alpha)\text{ and }
\xi = e^{i\alpha}. \qedhere
\]
\end{proof}
\begin{lemma}\label{lem:Phi_open}
$\Phi$ is open as a map from $(W\setminus\{0\}) \times \R$ onto $\Mcal$.
\end{lemma}
\begin{proof}
Let us fix $(w,\alpha)\in W\times\R$, $z=\Phi(w,\alpha)$ and $r>0$. We will prove that there exists $\rho>0$ such that if $z_1\in\Mcal$ and $\|z-z_1\|<\rho$ then there exists $(w_1,\alpha_1)\in W\times \R$ with $\|w-w_1\|+|\alpha-\alpha_1|<r$ such that $z_1 =\Phi(w_1,\alpha_1)$. Define $\xi=e^{i\alpha}$ and let $\log$ denote the holomorphic branch of the logarithm defined on $\C\setminus\{\lambda\xi:\lambda<0\}$ with $\log\xi = i\alpha$. In view of Lemma \ref{lem:functional}, we can find $\rho_1>0$ such that if $z_1\in\Mcal$ and $\|z-z_1\| <\rho_1$ then $\xi(z_1)$ belongs to the above domain and $\alpha_1 = -i \log \xi(z_1)$ is such that $|\alpha-\alpha_1|<r/2$. Define $w_1(\tau) = e^{-i\alpha_1\tau} z(\tau)$. Then $w_1\in W$, and $\|w-w_1\|<r/2$ if $z$ and $z_1$ are sufficiently close, say $\|z-z_1\|<\rho_2$.
\end{proof}
As a consequence of the above lemma, $\Ucal_I$ is open in $\Mcal$, for every open $I$.
\begin{lemma}\label{lem:Usubman}
If $|I|\le 2\pi$ then $\Ucal_I$ is a $C^\infty$ submanifold in $\spc$, modeled on $W\times\R$.
\end{lemma}
\begin{proof}
The lemma follows by \cite[p. 178-179]{AMR98}, in particular by Thms. 3.5.7 and 3.5.9 there.
According to such results, to prove the lemma we have to show that
\begin{enumerate}
\item $\Phi_I$ is injective;
\item $\Phi_I$ is open as a map from $(W\setminus\{0\})\times I$ onto $\Ucal_I$;
\item $\Phi_I$ is an immersion, i.e. $\Phi'_I(w,\alpha)$ has trivial kernel and closed range
for every $(w,\alpha)\in (W\setminus\{0\})\times I$.
\end{enumerate}
Notice that the above properties are sufficient because the manifolds are modeled over Hilbert spaces; for ones modeled on Banach spaces, we should ask for a closed split range $\Phi'_I$.

The first property follows from Lemma \ref{lem:inj_Phi}. As for the second one, it is a consequence of Lemma \ref{lem:Phi_open}.

To show 3., we recall from Lemma \ref{lem:Phismooth}
that
\[
\Phi'(w,\alpha)[\Delta,\delta] = E_\alpha \left(\Delta + i\delta I \cdot w \right).
\]
Since $E_\alpha$ is an isomorphism, the points $(\Delta,\delta)\in \ker\Phi'_I(w,\alpha)$ should satisfy
$\Delta + i\delta I \cdot w=0$. Assume that $(\Delta,\delta) \neq 0$, then $\delta\neq0$ and the previous identity is equivalent to
\[
\Delta(\tau+1) + i\delta(\tau+1) w(\tau+1) =
\Delta(\tau) + i\delta \tau w(\tau) = 0,
\]
for every $\tau\in[0,1]$. Since $\Delta$ and $w$ are $1$-periodic, we are led to $w\equiv0$. This is a contradiction, since $w\in W\setminus\{0\}$. Once we have proved that the kernel is trivial, we observe that the range of $\Phi'(w,\alpha)$ is $E_\alpha(W\oplus (iI\cdot w)\R)$. Since  $W\oplus (iI\cdot w)\R$ is closed in $\spc$ and $E_\alpha$ is an isomorphism, we deduce that the range is also closed in $\spc$.
\end{proof}
\begin{proof}[Proof of Proposition \ref{prop:Mcal}.]
We first show that $\Mcal$ is a submanifold of $\spc$. Let us take the intervals $I_1 = (-\pi,\pi)$
and $I_2 = (0,2\pi)$. By Lemma \ref{lem:inj_Phi} we have that $\Ucal_{I_1} \cup \Ucal_{I_2} = \Ucal_{(-\pi,2\pi)}
=\Mcal$, where each $\Ucal_{I_i}$ is open in $\Mcal$ by Lemma
\ref{lem:Phi_open}. In view of Lemma \ref{lem:Usubman}, we are left to show that the diffeomorphisms $\Phi_{I_i}$
from $W\times I_i$ onto $\Ucal_{I_i}$,  $i=1,2$, are compatible. This is again a consequence of Lemma \ref{lem:inj_Phi},
since
\[
\Phi_{I_2}^{-1} \circ \Phi_{I_1} (w,\alpha) =
\begin{cases}
(E_{-2\pi}w,\alpha + 2\pi) & \alpha \in (-\pi,0)\\
(w,\alpha) & \alpha \in (0,\pi),
\end{cases}
\]
and a similar formula also holds for $\Phi_{I_1} \circ \Phi_{I_2}^{-1}$.

The expression of $T_{E_\alpha w}\Mcal$ follows from Lemma \ref{lem:Phismooth}, since $\Phi'(w,\alpha)$ is an isomorphism
between $W\times\R$ and such tangent space.

Finally, once $\Mcal$ is a submanifold of an Hilbert space, it inherits the corresponding Riemmanian structure.
Then the geodesic distance $\dist_\Mcal$ is well defined, and of course
\[
\dist_\Mcal(z_1,z_2) \ge \|z_1-z_2\|.
\]
\end{proof}

\subsection{Variational principles for \texorpdfstring{$\BKS$}{BKS} on
\texorpdfstring{$\Mcal$}{M}}

In the following, we assume that $p$ is $C^1$ and $T$-periodic. Moreover, for any $z\in\Mcal$, when
we write $z=E_\alpha w$ we understand that $w$ is $1$-periodic.

By definition, functions in the manifold $\Mcal$ are defined in the interval
$0<\tau<2$, while in Section \ref{sec:ELeqBKS} we developed the theory
for the functional $\BKS$ on functions defined in $0<\tau<1$. To proceed, we have
two possibilities: either we have to restrict
the functions of $\Mcal$ on $(0,1)$, or to extend the definition of $\BKS$ to functions
defined on $(0,2)$. Actually, the two points of view turn out to be equivalent: indeed,
retracing the arguments in Section \ref{sec:euristic} and \ref{sec:ELeqBKS}, it is possible to
see that the natural definition for doubled intervals is
\[
\begin{split}
\BKS_2 &: H^1(0,2; \HH) \to \R \cup \{+\infty\}\\
\BKS_2(z) &:= \frac{1}{\Lcal_2(z)}
\Qcal_2(z) +
\Lcal_2(z)\Big[ 2+ \Rcal_2(z)\Big],
\qquad \Bcal_2(0):=+\infty,
\end{split}
\]
where
\[
\Lcal_2(z) := \frac{2T}{\int_0^2|z|^2},\quad
\Qcal_2(z) := 2 \int_0^2 | z'|^2,
\quad
\Rcal_2(z) := \int_0^2
|z|^2 \scalar{p\circ t_z}{\bar z i z},\quad
{t_z}_2(\tau) := \Lcal_2(z) \int_{0}^{\tau}|z|^2
\]
(in such a way that ${t_z}_2(2)=2T$). Then, since $p$ is $T$-periodic, direct
computations show that
\[
\BKS_2(z) = 2\BKS(\pro{z})\qquad\text{for every  }z\in\Mcal,
\]
where $\BKS$ is defined as usual (recall Lemma \ref{lem:BKS_per}). For this reason we will work directly with $\widetilde\BKS(z) = \BKS(\pro{z})$. We first check that critical points of $\widetilde\BKS$ correspond to generalized solutions of the perturbed Kepler problem; next, we will address the existence of such critical points.
\begin{lemma}\label{lem:crit_BKS_ass1}
Let $z\in\Mcal$, $z=E_\alpha w$, be a critical point for $\widetilde\BKS$ on $\Mcal$. Then $\pro{z}$ satisfies assumption \eqref{eq:varBKSvphi} of
Proposition \ref{prop:EL_for_B_KS}, namely
\[
\frac{d}{d\eps} \left[\BKS(\pro{z}+\eps \varphi)\right]_{\eps=0} = 0\qquad \text{for every }\varphi\in \Dcal(0,1;\HH).
\]
As a consequence, $\pro{z}\in C^3([0,1])$ satisfies the Euler-Lagrange equation \eqref{eq:ELeqB_KS},
\begin{equation}\label{eq:critrestr_bis}
\scalar{z'(1^-)}{v(1)} - \scalar{z'(0^+)}{v(0)} = 0\qquad \text{for every }v\in W_\xi,\text{ where }\xi = e^{i\alpha},
\end{equation}
and
\begin{equation}\label{eq:curve}
\scalar{z'(1)}{iz(1)}=0.
\end{equation}
\end{lemma}
\begin{proof}
First, let us take $v\in W_\xi$, where $\xi = e^{i\alpha}$. Then $z + \eps v \in W_\xi\subset\Mcal$, for every $\eps$ small.
Since $z = E_\alpha w$ is a critical point of $\widetilde\BKS$ on $\Mcal$ we have that
\[
\frac{d}{d\eps} \left[\widetilde\BKS(z+\eps v)\right]_{\eps=0} = 0\qquad \text{for every }v\in W_\xi.
\]
Since the restriction $z\mapsto \pro{z}$ is linear, this is equivalent to
\begin{equation}\label{eq:critrestr}
\frac{d}{d\eps} \left[\BKS(\pro{z}+\eps \pro{v})\right]_{\eps=0} = 0\qquad \text{for every }v\in W_\xi.
\end{equation}
Given $\varphi\in\Dcal(0,1;\HH)$, we extend it to a function $\tilde\varphi\in W_{\xi}$:
\[
\tilde\varphi(\tau):=
\begin{cases}
 \varphi(\tau) & \text{if }\tau\in[0,1]\\
e^{i\alpha} \varphi(\tau-1) & \text{if }\tau\in[1,2].
\end{cases}
\]
From the previous identity we deduce that
\[
\frac{d}{d\eps} \left[\BKS(\pro{z}+\eps \varphi)\right]_{\eps=0} = 0\qquad \text{for every }\varphi\in \Dcal(0,1;\HH).
\]

Once \eqref{eq:varBKSvphi} is satisfied, Lemma \ref{lem:EL_for_B_KS} implies that $\pro{z}$ is $C^3([0,1])$ and satisfies the Euler-Lagrange equation \eqref{eq:ELeqB_KS}. Then, reasoning as in Lemma \ref{lem:boundcond_LC}, we can integrate by parts in \eqref{eq:critrestr} and use the equation to obtain \eqref{eq:critrestr_bis}.

It remains to prove the identity \eqref{eq:curve}. We consider the path $[0,1]\ni \eps \mapsto \gamma_\eps \in\ \Mcal$,
$\gamma_\eps(\tau)=e^{i\eps \tau} z(\tau)$. Then
\[
\frac{d}{d\eps} \left[\widetilde\BKS(\gamma_\eps)\right]_{\eps=0} = 0.
\]
Using again that $z$ solves the Euler-Lagrange equation, we conclude that
\[
0 = \widetilde\BKS'(z) [\dot\gamma_0] = \frac{4}{\Lcal} \left[ \scalar{z'(1^-)}{\dot\gamma_0(1)} - \scalar{z'(0^+)}{\dot\gamma_0(0)} \right],
\]
where $\dot\gamma_0 (\tau) =  \left.\frac{d\gamma_\eps(\tau) }{d\eps}\right|_{\eps=0} = i\tau z(\tau)$.
\end{proof}
We have all the ingredients to prove the following result.
\begin{proposition}\label{prop:tilde}
Let $p$ be $C^1$ and $T$-periodic, and let $z\in\Mcal$ be a critical point for $\widetilde\BKS$ on $\Mcal$. Then
\[
x(t)=\bar z(\tau_z(t))i z(\tau_z(t)),
\]
where $\tau_z$ is defined as in Proposition \ref{prop:EL_for_B_KS}, is a $T$-periodic generalized solution of equation \eqref{eq:perturbed_kepler}.
\end{proposition}
\begin{proof}
In view of Lemma \ref{lem:crit_BKS_ass1} and Proposition \ref{prop:EL_for_B_KS}, we have that
$\tau_z$ and $x$ are well defined, of class $C^2$, and that $x$ is a generalized solution of equation \eqref{eq:perturbed_kepler}
in $(0,T)$. The fact that $x$ is (extendable as) a $T$-periodic function follows from Lemma \ref{lem:BKS_per} and
Remark \ref{rem:xdotKS}, in case $x(0) = x(T) \neq0$. Hence we are left to prove that, in case $x(0) = x(T) = 0$,
\[
\lim_{t \to 0^+} \frac12 |\dot x|^2 - \frac{1}{|x|}
=
\lim_{t \to T^-} \frac12 |\dot x|^2 - \frac{1}{|x|}
\qquad
\text{and}
\qquad
\lim_{t \to 0^+} \frac{x}{|x|}
=
\lim_{t \to T^-} \frac{x}{|x|}
.
\]
This can be done as for collisions in $(0,T)$, which were treated in Proposition \ref{prop:EL_for_B_KS}. More precisely,
we exploit the fact
that the map defined in equation \eqref{lem:ELconmu_KS}, namely
\[
t \mapsto -\dot x^2(t) |x(t)| = |\dot x(t)|^2 |x(t)| = \frac{4}{\Lcal^2}|z'(\tau_z(t))|^2
\]
tends to the same limit as $t\to0^+$ and $t\to T^-$ (see also the end of the proof of Theorem \ref{thm:periodicLC}).
\end{proof}
Once the role of $\TBKS$ is clarified, we turn to the variational framework. Since
$\Mcal$ is symmetric, i.e. $-\Mcal = \Mcal$, we can follow the lines of Section
\ref{sec:periodicBLC}, defining the genus of  $A\in\Acal:=\{A\in \Mcal\setminus\{0\}: A=-A, \ A\text{ is closed}\}$ as
\[
\gamma(A) := \inf\{n:\exists \phi \in C(A;\R^n\setminus\{0\}),\ \phi\text{ odd}\},
\]
and, accordingly,
\[
\Acal_m:=\{A \subset \Acal: A\text{ is compact and }\gamma(A)\ge m\},\qquad
\sigma_m := \inf_{A\in\Acal_m} \sup_{A} \widetilde\BKS.
\]
Since $W\subset \Mcal$, also in this case we have that $\Acal_m\neq \emptyset$, for every $m$. In particular, $\sigma_m<+\infty$ for every $m$.

As we already mentioned, the main difference with the planar case is that now the Palais-Smale
condition can not hold true: for instance, if $p$ has non-zero average on $[0,T]$, it is easy to
construct diverging (PS) sequences at any level of $\widetilde\BKS$, by considering constant loops.
To overcome this difficulty, we show that a weaker compactness property holds true, at least at
positive levels. A sequence $(z_n)_n\subset \Mcal$ is said to be a Palais-Smale-Cerami (PSC) sequence at level $\sigma$ for $\widetilde\BKS$ if, for some fixed $\hat z\in\Mcal$,
\begin{equation}\label{eq:PSC}
\widetilde\BKS(z_n)=\sigma + o(1),\qquad
\|\nabla \TBKS(z_n)\|(1+\dist_\Mcal (z_n,\hat z))=o(1) \text{ as } n\to\infty
\end{equation}
(recall Remark \ref{rem:norm_tangent}). Accordingly, $\TBKS$ satisfies the (PSC) condition at level $\sigma$ if any such a sequence admits a
strongly convergent subsequence. Such condition was introduced in \cite{MR581298}, see also
\cite{MR591554,MR713209}. The (PSC) condition is slightly weaker than the (PS) one, while the
most important implications are retained, see \cite[Ch. II, Rmk. 2.5]{struwe}. In particular, it is possible to show the following result.
\begin{proposition}\label{prop:AMCerami}
Each finite $\sigma_m$ is a critical level for $\TBKS$ provided the (PSC) condition holds at level $\sigma_m$.
\end{proposition}
\begin{proof}[Sketch of the proof]
Actually, this proposition is a version of Proposition \ref{prop:AM}, with (PS) replaced by (PSC). Again, the argument is based on the deformation lemma, therefore it is enough to assume the regularity of $\TBKS$ and the completeness of $\Mcal$ only on sublevels of  $\TBKS$, see \cite[Ch. II, Remarks after Thm. 5.7]{struwe}. The fact that (PSC) is enough to define a pseudo-gradient flow, and hence to prove a deformation lemma, is very well known in the literature, see e.g. \cite[Thm. 4.7]{MR1313161}. Actually, our proposition can be proved also applying directly Teorema $(*)$ in the original papers by Cerami \cite{MR581298,MR591554} to the quotient manifold $\Mcal/\sim$,
where $\sim$ is the equivalence relation induced by the involution $z\mapsto -z$.
\end{proof}
In order to prove the (PSC) condition for $\TBKS$ on $\Mcal$ we need some preliminary lemmas.
\begin{lemma}\label{lem:boundedPSC}
For every positive $a,b$ there exist positive constants $C_i=C_i(a,b)$, $i=1,2$, such that, for every $y\in  H^1(0,1;\HH)$,
\[
\begin{cases}
\BKS(y)\le a\\
\|y\|_2^2\le b
\end{cases}
\qquad\implies\qquad
\begin{cases}
\|y'\|_2 \le C_1(a,b)
\\
C_2(a,b) \le \|y\|_2\le b^{1/2}.
\end{cases}
\]
\end{lemma}
\begin{proof}
The lemma follows in three steps, by reasoning as in Lemmas \ref{lem:GN_antiper}, \ref{lem:BLConLc1} and \ref{lem:boundedPS}, respectively. We refer to such lemmas for further details. For the sake of simplicity, we write $\alpha:=\|y\|_2\cdot\|y'\|_2$, $\beta:= \|y\|^2_2$.

\textbf{Step 1.} To start with, we claim that, for any $y \in H^1(0,1;\HH)$, the following Gagliardo-Nirenberg inequality holds:
\[
\|y\|_\infty^2 \le  2 \alpha + \beta.
\]
Indeed, this follows by integrating with respect to $\tau_0$ the elementary inequality
\[
|y(\tau)|^2 \le |y(\tau_0)|^2 +  \int_{0}^{1} 2|y||y'|\,d\tau.
\]

\textbf{Step 2.} As a consequence of Step 1, we have that
\[
	|\Rcal(y)| \leq \int_0^1 |p\circ t_y||y|^4
	\leq \|p\|_\infty \|y\|^2_{\infty} \| y\|^2_{2} \le \|p\|_\infty
	\left(2\alpha + \beta\right)\beta,
\]
and therefore
\begin{equation}\label{eq:effealfaKS}
\frac1T \BKS(y) = \frac{2}{T^2}\alpha^2 + \frac{1}{\beta}\left(1 + \Rcal(z)\right)
\ge \frac{2}{T^2}\alpha^2  - 2\|p\|_\infty \alpha + \frac{1}{\beta}
- \|p\|_\infty \beta \ge -\frac{\|p\|_\infty^2T^2}{2} - \|p\|_\infty \beta.
\end{equation}

\textbf{Step 3.} Finally, using the assumptions, the last estimate implies
\[
\frac{2}{T^2}\alpha^2  - 2\|p\|_\infty \alpha \le \frac1T a + \|p\|_\infty b
\qquad\text{and}\qquad
\frac{1}{\beta} \le \frac1T a + \|p\|_\infty b + \frac{\|p\|_\infty^2T^2}{2},
\]
so that $\alpha$ is bounded above and $\beta = \|y\|_2^2$ is bounded away from $0$, and the existence of $C_2$ follows. Then
\[
\|y'\|_2^2 = \frac{\alpha^2}{\beta}
\]
is bounded above, and also the existence of $C_1$ follows.
\end{proof}
\begin{lemma}\label{lem:weakPS_KS}
Let $(z_n)_n\subset\Mcal$ be a (PSC) sequence for  $\TBKS$ at level $\sigma>0$. Then there
exists $z\in\Mcal$ such that, up to subsequences, $z_n\rightharpoonup z$ weakly in $\spc$ and uniformly. In particular, $z_n$ is both bounded and bounded away from $0$.
\end{lemma}
\begin{proof}
We notice that $z_n \in
T_{z_n} \Mcal$ (indeed, writing $z_n = E_{\alpha_n} w_n$, then it corresponds to the choice
$(\Delta,\delta)=(w,0)$ in Proposition \ref{prop:Mcal}, \ref{i:2}). We deduce that
\[
\begin{split}
\left|\TBKS'(z_n)[z_n]\right| &\le \|\nabla \TBKS(z_n)\|\|z_n\| \le
\|\nabla \TBKS(z_n)\|\left(\|\hat z\|+\|z_n-\hat z\|\right)\\ &\le
\|\nabla \TBKS(z_n)\|(\|\hat z\|+\dist_\Mcal (z_n,\hat z)) = o(1),
\end{split}
\]
as $n\to+\infty$, where $\hat z\in\Mcal$ is such that \eqref{eq:PSC} holds.
For easier notation we write $y_n = \pro{z_n}$, so that
\[
\TBKS(z_n) = \BKS (y_n),\qquad
\TBKS'(z_n)[z_n] = \BKS' (y_n)[y_n].
\]
We deduce
\[
\begin{cases}
\BKS(y_n) = \dfrac{\Qcal(y_n)}{\Lcal(y_n)} + \Lcal(y_n) + \Lcal(y_n)\Rcal(y_n) = \sigma + o(1)\smallskip\\
\dfrac12\BKS'(y_n)[y_n] = 2\dfrac{\Qcal(y_n)}{\Lcal(y_n)} -\Lcal(y_n) + \Lcal(y_n)\Rcal(y_n)  = o(1),
\end{cases}
\]
as $n\to+\infty$ (the expression of $\BKS'(y_n)[y_n]$ can be obtained by using Lemma \ref{lem:EL_for_B_KS} or, in an easier way, by noticing that the functionals $\Qcal$, $\Lcal$ and $\Rcal$ are homogeneous of degree $2,-2$ and $4$ respectively). Subtracting the
above relations we infer
\[
\frac{T}{\|y_n\|_2^2} = \Lcal(y_n) \ge \Lcal(y_n) - \dfrac{\Qcal(y_n)}{2\Lcal(y_n)}
= \frac12\sigma + o(1).
\]
Resuming, we have that
\[
\begin{cases}
\BKS(y_n) \le \sigma + 1 \smallskip\\
\|y_n\|_2^2 \le \dfrac{2T}{\sigma} +1,
\end{cases}
\]
for $n$ large, and Lemma \ref{lem:boundedPSC} yields $\|y_n'\|_2\le C_1$, $0< C_2 \le \|y_n\|_2 \le C_3$, where the constants are independent of $n$. Then, up to a subsequence,
$y_n \rightharpoonup y\not\equiv 0$ weakly in $H^1(0,1;\HH)$ and uniformly. To conclude the proof of the lemma, we are left to show that $y$ can be extended to a function in $z\in\Mcal$, and
that $z_n\rightharpoonup z$ (again possibly up to subsequences).

Let $(\xi_n)_n\subset\S^1$ be such that $z_n \in W_{\xi_n}$. By compactness of $\S^1$ we can assume that, up to subsequences, $\xi_n\to \xi$. In particular, by uniform convergence,
\[
y(1) \leftarrow y_n(1) = z_n(1) = \xi_n z_n(0) = \xi_n y_n(0) \to \xi y(0).
\]
We deduce that
\[
z_n \rightharpoonup z :=
\begin{cases}
y(\tau) & \tau\in[0,1],\\
\xi y(\tau-1) & \tau\in[1,2].
\end{cases}
\in\Mcal
\qedhere
\]
\end{proof}
\begin{lemma}\label{lem:PS_KS}
The functional $\TBKS$ satisfies the (PSC) condition in $\Mcal$ at any level $\sigma>0$.
\end{lemma}
\begin{proof}
Let $(z_n)_n\subset\Mcal$ be a (PSC) sequence at level $\sigma>0$ and, say, $z_n \in \Mcal$.
By Lemma \ref{lem:weakPS_KS} $z_n\rightharpoonup z$ in $X$, and we can assume $z\in W_\xi$.
Moreover, we can choose $\alpha_n\to\alpha$ in such a way that $\xi_n = e^{i\alpha_n}$, $\xi = e^{i\alpha}$. In particular, $E_{-\alpha}z\in W$. We notice that, by direct computations,
\begin{equation}\label{eq:strong_in_PSC}
E_{\alpha_n-\alpha} z \to z\text{ strongly in }X.
\end{equation}
Let us consider the sequence $(z_n - E_{\alpha_n-\alpha} z)_n$: on the one hand, it is weakly convergent, and hence bounded; on the other hand, $z_n - E_{\alpha_n-\alpha} z\in
T_{z_n} \Mcal$ (indeed, writing $z_n = E_{\alpha_n} w_n$, then it corresponds to the choice
$(\Delta,\delta)=(w_n - E_{-\alpha}z,0)$ in Proposition \ref{prop:Mcal}, \ref{i:2}). Since
$(z_n)_n$ is a (PSC) sequence we deduce that
\[
\TBKS'(z_n)[z_n - E_{\alpha_n-\alpha} z] = o (1) \qquad\text{as }n\to+\infty.
\]
Let us write again $y_n=\pro{z_n}$, $y=\pro{z}$, and, with some abuse, $E_{\alpha_n-\alpha} y=\pro{E_{\alpha_n-\alpha} z}$; we can reason as in the proof of Lemma
\ref{lem:PS}, using the appropriate expression of $\BKS'$, to infer
\[
o(1) = \frac{\Lcal(y_n)}{4}\BKS'(y_n)[y_n - E_{\alpha_n-\alpha} y] =
\int_0^1\scalar{y_n'}{y_n' - \left(E_{\alpha_n-\alpha} y\right)'} + o(1) = \|y_n'\|_2^2 -
\|y'\|_2^2 + o(1),
\]
where we used \eqref{eq:strong_in_PSC}, $y_n\rightharpoonup y$, and the fact that $\Lcal(y_n)$  is bounded since we have a positive lower bound of $\|y_n\|_2$. We deduce that
$y_n\to y$ strongly in $H^1(0,1;\HH)$ and thus, reasoning as in Lemma \ref{lem:weakPS_KS},
that $z_n\to z$ strongly in $\Mcal$.
\end{proof}
To conclude, we have to show that $\sigma_m>0$, for some $m$.
Since $\sigma_{m+1}\ge \sigma_m$, for every $m$, this will imply the existence of infinitely
many positive critical points.
\begin{lemma}\label{lem:pre_sigma>0}
Let  $k\in\N^+$ and $A\in \Acal_{4k+1}$. Then there exists $z_A\in A$ such that
\begin{equation}\label{eq:zA1}
z_A(i/k) = 0\qquad\text{ for every }i=0,1,\dots,k-1
\end{equation}
(since $z_A\in\Mcal$, $z_A(i/k) = 0 $, $i=k,k+1,\dots,2k$, as well). In particular,
\begin{equation}\label{eq:zA2}
\|\pro{z_A}\|_2 \le \frac{1}{k\pi}\|\pro{z_A'}\|_2.
\end{equation}
\end{lemma}
\begin{proof}
Let $\phi : A \to \HH^{k}\cong\R^{4k}$ be defined as
\[
\phi(z) = (z(0),z(1/k),z(2/k),\dots,z((k-1)/k)).
\]
Since $\phi$ is odd and continuous, and $\gamma(A)\ge 4k+1$, by definition of genus we infer the existence of $z_A$ such that
$\phi(z_A)=0$, and \eqref{eq:zA1} follows. Then $\left.z_A\right|_{((i-1)/k,i/k)}\in
H^1_0((i-1)/k,i/k)$. As a consequence
\[
\int_{(i-1)/k}^{i/k} |z_A|^2\,d\tau \le \frac{1}{k^2\pi^2} \int_{(i-1)/k}^{i/k} |z'_A|^2\,d\tau,
\]
and also \eqref{eq:zA2} follows.
\end{proof}
\begin{lemma}\label{lem:sigma>0}
$\sigma_{m}\to+\infty$ as $m\to+\infty$.
\end{lemma}
\begin{proof}
Let $A$ be such that $\gamma(A)\ge 4k+1$, and let $z_A \in A$ be such that Lemma \ref{lem:pre_sigma>0} holds
true. As usual, we write $y_A=\pro{z_A}$, $\alpha:=\|y_A\|_2\cdot\|y_A'\|_2$, $\beta:= \|y_A\|^2_2$.
Then Lemma \ref{lem:pre_sigma>0} implies
\[
\beta \le \frac{1}{k\pi}\alpha.
\]
Using \eqref{eq:effealfaKS} we obtain
\[
\begin{split}
\frac1T \BKS(y_A) &\ge \frac{2}{T^2}\alpha^2  - 2\|p\|_\infty \alpha + \frac{1}{\beta}
- \|p\|_\infty \beta \ge \frac{2}{T^2}\alpha^2  - \|p\|_\infty\left(2 + \frac{1}{k\pi} \right)\alpha + \frac{k\pi}{\alpha}\\
&\ge \frac{2}{T^2} \alpha^2  - 3 \|p\|_\infty \alpha + \frac{k}{\alpha}=: g_k(\alpha).
\end{split}
\]
Now, by direct computation one checks that $g_k$ has a unique critical point $\alpha_k>0$, with 
$\min_{\alpha>0} g_k(\alpha) = g_k(\alpha_k)$. Moreover, $\alpha_k \sim \left(
\frac{T^2 k}{4}\right)^{1/3}$ as $k\to+\infty$, and therefore $g_k(\alpha_k)\to+\infty$ too. 
Then, we have that
\[
A\in \Acal_{4k+1}
\qquad\implies\qquad
\sup_{A} \widetilde\BKS \ge \widetilde\BKS(z_A) \ge T g_k(\alpha_k),
\]
so that $\sigma_m \ge T g_k(\alpha_k)$ whenever $m\ge 4k+1$ and the lemma follows.
\end{proof}
\begin{proof}[Proof of Theorem \ref{thm:periodicKS}]
By Lemma \ref{lem:sigma>0} we know that infinitely many values $\sigma_m$ are positive.
Then, Proposition \ref{prop:AMCerami} and Lemma \ref{lem:PS_KS} insures that such
$\sigma_m$ are critical values for $\TBKS$, with (distinct) critical points, say, $z_m$. By 
Proposition \ref{prop:tilde} we obtain that the corresponding $x_m (t) = \bar z_m(\tau_{z_m}(t))i z_m(\tau_{z_m}(t))$ are periodic generalized solutions of \eqref{eq:perturbed_kepler}.
Finally, since $\Acal(x_m) = \TBKS (z_m) = \sigma_m \to +\infty$, there exist infinitely 
many different solutions $x_m$.
\end{proof}

\appendix

\section{Some auxiliary results}

\subsection{Proof of Lemma \ref{lem:elementary}}\label{app:lemma}

Let us observe that the typical example for Lemma \ref{lem:elementary} is $a(t) = \frac32t^{1/3}$, with $A(t) = t^{2/3}$ and $B(\tau) = \tau^{3/2}$.

\begin{proof}[Proof of Lemma \ref{lem:elementary}] Since $\frac{1}{a(t)}$ is integrable, the set
\[
Z := \left\{ t \in [0,T] : a(t) =0 \right\}
\]
has zero measure. Hence, the function $A:[0,T] \to [0,\Xi]$ is continuous and strictly increasing. Therefore the inverse $B(\tau)$ is continuous. To prove that $B$ has a derivative we first observe that the quotient
\[
\Delta(\tau,h) = \frac{B(\tau+h)-B(\tau)}{h},
\qquad
\tau \in[0,\Xi],\, h \neq 0
\]
satisfies
\[
\frac{1}{\Delta(\tau,h)}
= \frac{1}{B(\tau+h)-B(\tau)}
\int_{B(\tau)}^{B(\tau+h)} \frac{d\xi}{a(\xi)}.
\]
We now distinguish two cases; first, we take $\tau \in [0,\Xi]$ such that $a(B(\tau))=0$. Then, for any $\varepsilon >0$ there exists $\delta>0$ such that
$a(\zeta)=|a(\zeta)|<\varepsilon$ if $|\zeta - B(\tau)| < \delta$. Then, if $|h|$ is so small that $|B(\tau+h) - B(\tau)| < \delta$, the inequality
\[
\left|\int_{B(\tau)}^{B(\tau+h)} \frac{d\xi}{a(\xi)}\right|
> \frac{|B(\tau+h) - B(\tau)|}{\varepsilon}
\]
holds. As a consequence, for any $\varepsilon>0$ we can choose $h$ sufficiently small such that  $|\Delta(\tau,h)| < \varepsilon$. This implies that $B'(\tau)=0$.

Second, let us take $\tau \in [0,\Xi]$ be such that $a(B(\tau))>0$. The continuity of $\frac{1}{a(t)}$
at $t=B(\tau)$ implies that
\[
\frac{1}{\Delta(\tau,h)} \to \frac{1}{a(B(\tau))},
\quad \text{as } h \to 0,
\]
therefore $B'(\tau) = a(B(\tau))$.

Summing up, the derivative of $B$ exists everywhere and $B'(\tau) = a(B(\tau))$ is a continuous function.
\end{proof}

\subsection{Details of Remark \ref{rem:mudiversi}}

By assumption, $y\in H^1_0((0,1);\R^d)$ satisfies $|y(\tau)|>0$ for $\tau\in(0,1)$ and
\begin{equation}\label{eq:app_el}
\frac{d}{d\eps} \left[\Bcal(y+\eps \psi)\right]_{\eps=0} = 0\qquad \text{for every }\psi\in\Dcal(0,1),
\end{equation}
where
\[
\Bcal(y) := \frac{1}{\Lcal(y)}\Qcal(y) + \Lcal(y).
\]
Moreover, by Proposition \ref{prop:EL_for_B}, $x_y=x_y(t)$ satisfies
\begin{equation}\label{eq:app_kep}
\ddot x_y = -\frac{x_y}{|x_y|^3}\qquad \text{in }(0,T).
\end{equation}
Now, let $\varphi\in\Dcal(0,1/2)$. By definition of $w$, we obtain
\[
\begin{split}
\int_0^{1/2} |w(\tau)+\eps\varphi(\tau)|^2\,d\tau &= h_1^2 \int_0^{1/2} \left|y(2\tau)+\frac{\eps}{h_1^2}\varphi(\tau)\right|^2\,d\tau =
\frac{h_1^2}{2} \int_0^{1} |y+\eps\psi|^2 \\
\int_{1/2}^1 |w(\tau)+\eps\varphi(\tau)|^2\,d\tau &= h_2^2 \int_{1/2}^1 |y(2-2\tau)|^2\,d\tau =  \frac{h_2^2}{2} \int_0^{1} |y|^2,
\end{split}
\]
where we wrote $\varphi(\tau) = h_1^2 \psi(2\tau)$, $\psi\in\Dcal(0,1)$. Analogously,
\[
\Qcal(w+\eps\varphi) = 2h_1^2\Qcal(y+\eps\psi) + 2 h_2^2 \Qcal(y).
\]
Direct calculations yield
\[
\frac{d}{d\eps} \left[\Bcal(w+\eps \varphi)\right]_{\eps=0}  = 2^{1/3}h_1^2 \frac{d}{d\eps} \left[\Bcal(y+\eps \psi)\right]_{\eps=0}=0
\]
by \eqref{eq:app_el}. Since analogous arguments hold when $\varphi\in\Dcal(1/2,1)$, \eqref{eq:doppiobump} follows.

Taking $\eps=0$ un the previous computations, and recalling that $h_1^2 + h_2^2 = 2^{1/3}$, we obtain
\[
\Lcal(w) = 2^{2/3} \Lcal(y), \qquad \Qcal(w) = 2^{4/3} \Qcal(y).
\]
Now, let $0<\tau<1/2$. We have
\[
t_w(\tau):= \Lcal(w) \int_{0}^{\tau}|w(\xi)|^2\,d\xi =
2^{-1/3} h_1^2 t_y(2\tau)
\qquad\iff\qquad
\tau_w(t) = \frac12 \tau_y\left(2^{1/3}h_1^{-2} t\right),
\]
whenever $0<t< 2^{-1/3}h_1^{2}T$, and
\[
x_w(t) := |w(\tau_w(t))|w(\tau_w(t)) = h_1^2 x_y\left(2^{1/3}h_1^{-2} t\right).
\]
Substituting into \eqref{eq:app_kep} we finally obtain
\[
\ddot x_w = -2^{2/3}h_1^2 \,\frac{x_w}{|x_w|^3}\qquad \text{in }(0,2^{-1/3}h_1^{2}T).
\]
Performing similar calculations in $(1/2,1)$ (or changing variable as $\tau \leftrightarrow 1-\tau$) we have
\[
\ddot x_w = -2^{2/3}h_2^2 \,\frac{x_w}{|x_w|^3}\qquad \text{in }(2^{-1/3}h_1^{2}T,T).
\]
In particular, this implies \eqref{eq:doppiokeplero} with $\mu_i = 2^{2/3} h_i^2$, and this holds for any choice
of positive $\mu_i$ satisfying
\[
\mu_1+\mu_2=2,
\]
in agreement with \eqref{eq:somme_mui}.

\section*{Acknowledgments} V.B. and G.V. have been partially supported
by the Italian PRIN-2015KB9WPT Grant: ``Variational methods, with applications to problems in 
mathematical physics and geometry'', by the ERC Advanced Grant 2013 n. 339958: ``Complex 
Patterns for Strongly Interacting Dynamical Systems -- COMPAT'', and by the INDAM-GNAMPA group.
R.O. has been partially supported by Spanish MINECO Grant with FEDER funds 
MTM2017-82348-C2-1-P. G.V. has been partially supported by the project Vain-Hopes  within the program VALERE - Università degli Studi della Campania "Luigi Vanvitelli".

\bibliography{bibliog_desing}
\bibliographystyle{abbrv}

\medskip
\small
\begin{flushright}
\noindent Vivina Barutello\\
Dipartimento di Matematica ``Giuseppe Peano'', Universit\`a di Torino,\\
Via Carlo Alberto 10, 10123 Torino, Italy\\
\verb"vivina.barutello@unito.it"\\
\medskip
\noindent Rafael Ortega\\
Departamento de Matem\'atica Aplicada, Universidad de Granada,\\
E-18071 Granada, Spain\\
\verb"rortega@ugr.es"\\
\medskip
\noindent Gianmaria Verzini\\
Dipartimento di Matematica, Politecnico di Milano\\
piazza Leonardo da Vinci 32, 20133 Milano, Italy\\
\verb"gianmaria.verzini@polimi.it"
\end{flushright}

\end{document}